\documentclass[11pt,a4paper,reqno]{amsart}

\usepackage{color,calc,graphicx,amsfonts,amsthm,amscd,epsfig,psfrag,amsmath,amssymb,enumerate,dsfont}
\usepackage{appendix}

\usepackage{amsmath,amssymb,graphics,epsfig,color,enumerate,psfrag}
\usepackage{verbatim}
\newtheorem{Theorem}{Theorem}[section]
\newtheorem{TheoremA}{Theorem}
\newtheorem{example}{Example}
\newtheorem{Lemma}[Theorem]{Lemma}
\newtheorem{Proposition}[Theorem]{Proposition}
\newtheorem{Corollary}[Theorem]{Corollary}
\newtheorem{Remark}[Theorem]{Remark}

\newtheorem{Claim}[Theorem]{Claim}

\newtheorem{Warning}[Theorem]{Warning}

 \definecolor{refkey}{gray}{0.8}
 \definecolor{labelkey}{gray}{0.8}

\definecolor{darkgreen}{rgb}{0,0.4,0}

\definecolor{light}{gray}{.9}

\newcommand{\cC}{\ensuremath{\mathcal C}}

\newcommand{\cF}{\ensuremath{\mathcal F}}

\newcommand{\cP}{\ensuremath{\mathcal P}}

\newcommand{\E}{{\ensuremath{\mathbb E}} }
\newcommand{\bbE}{{\ensuremath{\mathbb E}} }

\newcommand{\bbL}{{\ensuremath{\mathbb L}} }

\newcommand{\N}{{\ensuremath{\mathbb N}} }
\newcommand{\bbN}{{\ensuremath{\mathbb N}} }

\renewcommand{\P}{{\ensuremath{\mathbb P}} }
\newcommand{\bbP}{{\ensuremath{\mathbb P}} }

\newcommand{\bbQ}{{\ensuremath{\mathbb Q}} }
\newcommand{\R}{{\ensuremath{\mathbb R}} }
\newcommand{\bbR}{{\ensuremath{\mathbb R}} }

\newcommand{\bbX}{{\ensuremath{\mathbb X}} }
\newcommand{\bbY}{{\ensuremath{\mathbb Y}} }
\newcommand{\Z}{{\ensuremath{\mathbb Z}} }
\newcommand{\bbZ}{{\ensuremath{\mathbb Z}} }

\let\a=\alpha \let\b=\beta   \let\d=\delta  \let\e=\varepsilon
 \let\g=\gamma       \let\l=\lambda
      \let\o=\omega      
  \let\s=\sigma \let\t=\tau   
 \let\x=\xi \let\z=\zeta
\let\D=\Delta      
\let\O=\Omega 

\newcommand{\vrho}{ v_{{}_{X^\rho}}  }

\begin{document}


\title[1d Mott variable range hopping  with  external field]{The velocity of  1d Mott variable range hopping  with  external field}

\author[A.\ Faggionato]{Alessandra Faggionato}
\address{Alessandra Faggionato.
  Dipartimento di Matematica, Universit\`a di Roma ``La Sapienza''.
 P.le Aldo Moro 2, 00185 Roma, Italy}
\email{faggiona@mat.uniroma1.it}

\author[N.\ Gantert]{Nina  Gantert}
\address{Nina Gantert. Fakult\"at f\"ur Mathematik,  Technische Universit\"at M\"unchen.  85748 Garching, Germany}
\email{gantert@ma.tum.de}

\author[M.\ Salvi]{Michele Salvi}
\address{Michele Salvi.  Centre De Recherche en Math\'ematiques de la D\'ecision 
Universit\'e Paris - Dauphine, 
Place du Mar\'echal De Lattre De Tassigny, 75775 Paris, France}
\email{salvi@ceremade.dauphine.fr}

%

\maketitle
\begin{abstract}
  Mott variable range hopping is a fundamental mechanism for 
low--temperature electron conduction in disordered solids  in the regime of Anderson localization. In a mean field approximation, it reduces to  a random walk (shortly, Mott random walk)  on a random marked point process with possible long-range jumps.

  We consider here  the one-dimensional Mott  random walk and we add an external field (or a bias to the right). We show that the bias makes the walk transient, and investigate its linear speed. Our main results are conditions for ballisticity (positive linear speed) and for sub-ballisticity (zero linear speed), and the  existence  in the ballistic regime  of an invariant distribution   for the environment viewed from the walker,  which is mutually absolutely continuous with respect to  the original law of the environment. If the point process is a renewal process, the aforementioned  conditions result in a sharp criterion for ballisticity. Interestingly, the speed is not always continuous as a function of the bias.

\medskip


\smallskip

 \end{abstract} 

\section{Introduction}\label{intro}

Mott variable range hopping is a fundamental mechanism at the basis of 
low--temperature electron conduction in disordered solids (e.g.~doped semiconductors)  in the regime of Anderson localization (see \cite{AHL,MA,MD,POF,SE}). 
By localization, and using a mean--field approximation, Mott variable range hopping can be described by a suitable random walk $(\bbY_t)_{t \geq 0}$ in a random environment $\o$. The environment $\o$ is given by a  marked simple point process $\{(x_i, E_i)\}_{i\in\Z}$ with law $\bbP$. The sites $x_i\in \R^d$ correspond to  the points  in the disordered  solid around which the  conduction electrons  are  localized, and $E_i\in[-A,A]$ is  the ground energy of the associated localized wave function.   The random walk $\bbY_t$ has state space $\{x_i\}$ and can jump from a site  $x_i$ to any other site $x_k\not = x_i$ with probability  rate
\[
r_{x_i,x_k}(\o):= \exp \{ -|x_i-x_k|-\b  (|E_i|+|E_k|+ |E_i-E_k|)\}\,,
\]
$\beta$ being  the inverse temperature. 

We refer to \cite{CF,CFG,CFP,FM,FSS} for  rigorous results on the random walk $\bbY_t$, including the stretched exponential decay of the diffusion matrix as $\beta \to \infty$ in accordance with the physical Mott law for $d \geq 2$. Here  we  focus    on the one-dimensional case, i.e.~$\{x_i\}_{i\in \bbZ}\subset \bbR$ (we order the sites $x_i$'s in increasing order, with $x_0=0$), and study the effect of applying an external field. This  corresponds to modifying the above  jump rates $r_{x_i,x_k}(\o)$ by a factor $e^{ \l (x_k -x_i)}$, where $\l\in(0,1)$ has to be interpreted as the intensity of  the  external field. Moreover, we generalize the form of the jump rates, finally taking 
\[{
r^\l _{x_i,x_k}(\o):= \exp \{ -|x_i-x_k|+ \l (x_k-x_i)+ u(E_i, E_k)\}\,,}
\]
with $u$ a symmetric bounded  function. 
For simplicity, we keep the same notation $\bbY_t$ for the resulting random walk starting at the origin. 

 Under rather weak assumptions on the environment, we will show that $\bbY_t$ is a.s.~transient for  almost every environment  $\o$ (cf.~Theorem \ref{teo1}--(i)). In the rest of Theorem \ref{teo1} we  give two conditions in terms of the exponential moments of the inter--point distances, both assuring that  the asymptotic velocity $v_\bbY(\l) :=\lim _{t\to \infty} \frac{\bbY_t}{t}$ is well defined  and almost surely constant, that is, it does not depend on the realization of $\o$. Call $\bbE$ the expectation with respect to $\bbP$. The first condition, namely $\E\bigl[{\rm e}^{(1-\l)(x_1-x_0)}\bigr]<\infty\,$ and $u$ continuous, implies ballisticity, i.e.~$v_\bbY(\l)>0$. The second condition, namely $\E\bigl[{\rm e}^{(1-\l)(x_1-x_0)-(1+\l)(x_0-x_{-1})}\bigr]=\infty\,$,  implies sub-ballisticity, i.e.~$v_\bbY(\l)=0$.   In particular, if the points $\{x_i\}_{i\in\Z}$ are given by a renewal process, our two conditions give a sharp dichotomy (when $u$ is continuous). We point out that there are cases in which $v_\bbY(\l)$ is not continuous in $\lambda$ (see Example \ref{sorpresa} in Subsection \ref{esaedro}).

Under the condition leading to ballisticity we also  show that the Markov process given by the environment viewed from the walker  admits a stationary ergodic distribution $\bbQ^\infty$, which is mutually absolutely continuous to the original law $\bbP$ of the environment. Moreover we give  upper and lower bounds on the Radon--Nikodym derivative $\frac{d \bbQ^\infty}{d \bbP}$, showing that it is in $L^1(\bbP)$, and we characterize the asymptotic velocity as the  expectation of the local drift with respect to the measure $\bbQ^\infty$ (cf.~Theorem \ref{teo2}). 

The study of ballisticity for the Mott random walk is the first fundamental step towards proving the Einstein Relation, which states the proportionality of diffusivity and mobility of the process (see e.g.~\cite{GMP}). Among other important applications, the Einstein Relation would allow to conclude the proof of the physical Mott law, which was originally stated for the mobility of the process and has only been proved for its diffusivity (see \cite{CF}, \cite{FM} and \cite{FSS}). The Einstein Relation will be addressed in future work.
\smallskip

The techniques used to prove ballisticity and sub-ballisticity are different. In order to comment them it is convenient to refer to the discrete--time random walk\footnote{We use the convention $\N_+:= \{1,2,\dots \}$ and $\N:= \{0,1,2,\dots \}$}   $(X_n)_{n \in\N}$  on $\bbZ$ such that $X_n=i$ if after $n$ jumps the random walk $\bbY_t$ is at site $x_i$.  Due to our assumptions on the environment, the ballistic/sub-ballistic behavior of $(\bbY_t)_{t \geq 0}$ is indeed the same as that of  $(X_n)_{n \in\N}$, and therefore we focus on the latter.
\smallskip

 We first comment the ballistic regime. Considering first a generic  random walk on $\bbZ$  starting at the origin and a.s.~transient to the right,  ballisticity is usually derived by proving a law of large numbers (LLN)  for the hitting times $(T_n)_{n \geq 1}$, where $T_n $ is the first time the random walk reaches the half-line $[n,+\infty)$. 
In the case of nearest neighbor random walks,  $T_n$ is simply the hitting time of $n$, and considering an ergodic environment one can derive the LLN for $(T_n)_{n \geq 1}$ by showing that the sequence $(T_{n+1}- T_n)_{n \geq 1}$ is stationary and  mixing for the annealed law  as in \cite{A,So}. 
This technique cannot be applied in the present case, since 
our random walk has infinite range and much information about the environment to the right is known, when a site in $[n,+\infty)$ is visited for the first time.
A very useful tool is the method developed in \cite{CP} where the authors have studied ballisticity for a class of random walks on $\bbZ$ with arbitrarily long jumps. Their strategy is as follows. First one introduces for any positive integer $\rho$  a truncated random walk obtained from the original one by forbidding all jumps of length larger than $\rho$. The ergodicity of the environment and the finite range of the jumps allow to introduce a regenerative structure related to the times $T_{\rho n}$, and to analyze the asymptotic behavior of the $\rho$--truncated random walk.  In particular, one proves    that the environment viewed from the $\rho$--truncated random walk admits a stationary ergodic distribution $\bbQ^\rho$ which is mutually absolutely continuous to the original law of the environment. A basic ingredient here is the theory of cycle--stationarity and cycle--ergodicity (cf.~\cite[Chapter 8]{T} and \cite{CC} for an example in a  simplified setting).  
Finally, one proves that the sequence $(\bbQ^\rho)_{\rho\in\N_+}$
  converges weakly to a probability distribution $\bbQ^\infty$, which is indeed a stationary and ergodic distribution for the environment viewed from the random walker $(X_n)_{n\in\N}$ and is  also mutually absolutely continuous to the law of the environment $\bbP$. Since, as usual, the random walk can be written as an additive functional of the  environment viewed from the  random walker, one can apply Birkhoff's ergodic theorem and use the ergodicity   of $\bbQ^\infty$ to get the strong LLN for the random walk (hence its asymptotic velocity) for $\bbQ^\infty$--a.e.~environment. Using the fact that $\bbP \ll \bbQ^\infty$,  the above strong LLN holds  for $\bbP$--a.e.~environment, too. 
Finally, since the velocities of the $\rho$--truncated walks are uniformly bounded from below by a strictly positive constant and since they converge to the velocity of $(X_n)_{n\in\N}$ when $\rho\to\infty$, we obtain a ballistic behavior.

To analyze ballisticity we have used the same method as in \cite{CP}, although  one  cannot apply \cite[Theorems 2.3, 2.4]{CP} directly to  the present case, since some hypotheses are not satisfied  in our context. In particular, in \cite{CP} three  conditions (called E, C, D) are assumed, and only condition $C$ is satisfied  by our model.
 By means of estimates based on capacity theory, we are able to extend the method developed in \cite{CP} to the present case. 

\smallskip
We now move to sub-ballisticity (the regime of zero velocity is not covered in \cite{CP} and our method could be in principle applied to  random walks on $\bbZ$ with arbitrarily long jumps). 
 We define a coupling between the random walk  $(X_n)_{n \geq 0}$, a sequence of suitable $\N_+$--valued  i.i.d.~random variables $\xi_1, \xi_2,\dots$ with finite mean, and  an ergodic  sequence of random  variables $S_1,S_2, \dots$ with the following properties:
 Fix $\omega$ and call now $T_{k+1}$ the first time the random walk overjumps the
 point $\xi_1 + \cdots + \xi_k$. $S_k$ is a geometric random variable of parameter $s_k=s( \t_{ \xi_1+ \dots+ \xi_k} \o)$, where $\tau_\cdot$ is the usual shift and $s$ a deterministic function. The coupling guarantees that $X_{T_{k+1}}$  does not exceed $\xi_1 + \cdots + \xi_k+\xi_{k+1}$ and also ensures that the time  $T_{k+1}- T_k$
 is larger than $S_k$.
 Notice that 
  \begin{equation}\label{passatamutti}
  \frac{X_n}{n} \leq \frac{ X_{T_{k+1} } }{T_k}\leq \frac{  { \xi_1+ \dots + \xi_{k+1}} }{S_1+ S_2 + \dots+ S_k}  \qquad \text{ if } \qquad T_k \leq n < T_{k+1}\,,
  \end{equation}
and therefore  the sub-ballisticity of  $(X_n)_{n\geq 0}$ follows from 
 the LLN  for $(\xi_k)_{k\in\N_+}$  and the LLN for  $(S_k)_{k\in\N_+}$, since our assumption $\E\bigl[{\rm e}^{(1-\l)(x_1-x_0)-(1+\l)(x_0-x_{-1})}\bigr]=\infty\,$ implies that $\bbE[ 1/s(\o) ] =+\infty$.
 

\subsection{Outline}
In {Section \ref{notazione}} we rigorously introduce the (perturbed) Mott random walk in its continuous and discrete-time versions. Theorem \ref{teo1} states the transience to the right and gives  conditions implying  ballisticity or subballisticity. 
Theorem \ref{teo2}, deals with the Radon-Nikodym derivative of the invariant measure for the  environment  viewed from the walker
with respect to the original law $\bbP$ of the environment and gives a characterization of the limiting speed of the walk. {Subsection \ref{domenica}} comments the assumptions we made for the Theorems, while two important (counter-)examples can be found in {Subsection \ref{esaedro}}.

In {Section \ref{zetino}} we collect results on the $\rho$--truncated walks. 
Estimates of the effective conductances induced by these walks and of the time they spend on a given interval are carried out in {Subsections \ref{condottiero}} and {\ref{ECD}}, respectively. In {Subsection \ref{subsec:hit}} we show that the probability for them to hit a specific site to the right is uniformly bounded from below.

{Section \ref{rigenero}} introduces the regenerative structure for the $\rho$--truncated random walks and in {Subsection \ref{subsec:regeneration}} we give estimates on the regeneration times. The existence and positivity of the limiting speed for the truncated walks is proven in {Subsection \ref{veloce}}.

In {Section \ref{tronco}} we characterize the density of the invariant measure for the process viewed from the $\rho$--truncated walker with respect to the original law of the environment.
In {Subsection \ref{up_bound}} we bound the Radon-Nikodym derivative from above by an $L^1$ function, while in {Subsection \ref{lo_bound}} we give a uniform lower bound. In {Subsection \ref{deboluccio}} we finally pass to the limit $\rho\to\infty$ and obtain an invariant measure for the  environment viewed from  the non-truncated walker and show that it is also absolutely continuous   with respect to $\bbP$ (see Lemma \ref{chiave}).

To conclude, in {Sections \ref{proof_transience}, \ref{ballistic_part}} and {\ref{submarine}} we prove, respectively, parts (i), (ii) and (iii) of Theorem \ref{teo1}. Some technical results are collected in the Appendixes \ref{ape}, \ref{appendix:ergodic} and \ref{app_nn}.

\section{Mott random walk and main results  }\label{notazione}
One-dimensional  Mott random walk  is a particular random walk in a random environment. 
 The environment $\o$ is given by 
   a  double--sided sequence  $(Z_k, E_k) _{k \in \bbZ}$  of random variables,  with $Z_k \in (0,+\infty)$ and $E_k\in  \bbR$ for all $k \in \bbZ$. 
   We denote by $\O$ the  space of all environments,  by $\bbP$ and $\bbE$ the law of the environment  and the associated expectation, respectively. Given $\ell \in \bbZ$, we define the  shifted environment $\t_\ell \o$ as $\t_\ell \o:=(Z_{k+\ell}, E_{k+\ell}) _{k \in \bbZ}$. 
 Our main assumptions on the environment are the following:
\begin{itemize}
\item[(A1)] The sequence $( Z_k,E_k)_{k \in \bbZ}$ is stationary and ergodic with respect to shifts;
\item[(A2)]   $\bbE[Z_0] $ is finite;
\item[(A3)]  $\bbP ( \o = \t_\ell \o)=0$  for  all $\ell \in \bbZ$;   
\item[(A4)]   There exists $  d>0$ satisfying
 $\bbP(Z_0 \geq d)=1$.
\end{itemize}
We postpone to Subsection \ref{domenica} some comments  on the above assumptions. 

\smallskip 
   It is convenient to introduce the sequence $  ( x_k ) _{k \in \bbZ}$  of points on the real line, where $x_0=0$ and $x_{k+1}= x_k +Z_k$. Then the environment $\o$ can be thought also as the marked simple  point process  $  (x_k,E_k)_{k \in \bbZ}$, which will be denoted again by $\o$ (with some abuse of notation).   In this case, the $\ell$--shift reads $\t_\ell \o =  (x_{k+\ell }-x_\ell,E_{k+\ell} )_{k \in \bbZ}$. For physical reasons, $E_k$ is called the energy mark associated to point $x_k$, while $Z_k$ is the interpoint distance (between $x_{k-1}$ and 
 $x_k$).

Fix now  a symmetric and bounded (from below by $u_{\rm min}$ and from above by $u_{\rm max}$) measurable function $u: \bbR \times \bbR \to \bbR$. Given an environment $\o$, the Mott random walk $(\bbY_t)_{t \geq 0}$ is the continuous--time
random walk   on $\{ x_k\}_{k \in \bbZ} $ with  probability rate
\begin{equation}\label{eq_tassi}
{r_{x_i,x_k}(\o):= \exp \{ -|x_i-x_k|+u (E_i,E_k)\} }
\end{equation}
for a jump from $x_i$ to $x_k \not =x_i$.  For convenience,  we set
  $r_{x,x}(\o)\equiv 0$.   Note that  the Mott walk is well defined for $\bbP$--a.a.~$\o$. Indeed, since the interpoint distance is a.s.~at least $d$ and the 
function $u$ is uniformly bounded,  the holding time parameter $r_x(\o) := \sum_y r_{x,y} (\o)$ can be bounded from above by a constant $C>0$  uniformly in   $x \in \{ x_k\}_{k \in \bbZ}  $, hence explosion does not take place.

\smallskip

We now introduce a  bias $\l$ which corresponds to the intensity of the external field. For a fixed $\l \in [0,1)$, the biased Mott random walk  $(\bbY_t)_{t \geq 0}$ with environment $\o$ is the continuous--time random walk  on $\{ x_k\}_{k \in \bbZ} $  with probability rates
\begin{equation}\label{eq_eq_eq}
r^\l _{x,y}(\o)= e^{  \l   (y-x) } r_{x,y}(\o)
\end{equation}
for a jump from  $x$ to $y \not =x$.   For convenience,  we set
  $r^\l_{x,x}(\o)\equiv 0$ and denote the holding time parameter by 
 $ r^\l_x(\o):= \sum _{y } r^\l _{x,y}(\o)$. When $\l=0$, one recovers the original Mott random walk. Since, for $\l \in (0, 1)$, we have a.s.~
{$r^\l_x(\o) \leq \sum _{ k \in \bbZ} e^{ -(1-\lambda)|k| d  + u_{\rm max}}< \infty$, }the biased Mott random walk with environment $\o$ is well defined for $\bbP$--a.a.~$\o$. 
 
\smallskip
We can consider also the jump chain $(Y_n)_{ n \geq 0}$  associated to  the biased  Mott random walk (we call it the \emph{discrete time Mott random walk}). Given the environment $\o$,  $(Y_n)_{ n \geq 0
}$  is the discrete--time random walk on $\{ x_k\}_{k \in \bbZ} $  with jump probabilities
\begin{equation}\label{treno}
p^\l _{x,y}(\o):= \frac{r^\l _{x,y}(\o)}{r _x^\l (\o)}\,, \qquad x\not = y\,.\end{equation}
A similar definition holds for the unbiased case ($\l=0$).

\smallskip

The following result concerns transience, sub-ballisticity and ballisticity:
\begin{TheoremA}\label{teo1}
Fix $\l \in (0,1)$.
Then, for $\bbP$--a.a.~$\o$, the continuous time Mott random walk $(\bbY_t)_{t \geq 0}$ with environment $\o$, bias $\l$ and starting at the origin  satisfies the following properties:
\begin{itemize}
\item[(i)] Transience to the right: $\lim _{t \to \infty} \bbY_t=+\infty$ a.s. 

\item[(ii)] Ballistic regime: If  $\bbE\bigl[ e^{ (1-\l) Z_0}  \bigr] <+\infty $  and $u :\bbR \times \bbR  \to \bbR$ is  continuous, then 
the asymptotic velocity
\[
v_{\bbY}(\l):=\lim _{t \to \infty} \frac{\bbY_t}{t}
\]
 exists  a.s.,  it is deterministic, finite and strictly positive (an integral representation of $v_{\bbY}$ is given in Section \ref{ballistic_part}, see \eqref{int_rapr} and  \eqref{fatto}).

\item[(iii)] Sub--ballistic regime: If 
\begin{equation}\label{immacolata}
\E\bigl[{\rm e}^{(1-\l)Z_0-(1+\l)Z_{-1}}\bigr]=\infty\,,
\end{equation}
then 
\begin{equation}
v_{\bbY}(\l):=\lim _{t \to \infty} \frac{\bbY_t}{t}=0\,.
\end{equation}
In particular, if  $\bbE[Z_{-1}|Z_0]\leq C$  for some constant $C$ which does not depend on $\o$  and $\E[{\rm e}^{(1-\l)Z_0}]=\infty$,  then condition \eqref{immacolata} is satisfied and $v_{\bbY}(\l)=0$.

\end{itemize}

In addition,  for $\bbP$--a.a.~$\o$  the above properties remain valid (restricting to integer times $n \geq 0$) for the discrete time Mott random walk  $(Y_n)_{ n \geq 0}$ with environment $\o$, bias $\l$ and starting at the origin, and its velocity 
$v_Y(\l):= \lim  _{n \to \infty} \frac{Y_n}{n}$.
\end{TheoremA}

\begin{Remark}  

In the case $\l=0$ the Mott random walks $\bbY_t$ and $Y_n$ are recurrent and have a.s.~zero velocity. Recurrence follows from  \cite[Thm. 1.2--(iii)]{CFG} and the recurrence of   the spatially homogeneous discrete time random  walk on $\bbZ$ with probability to jump from $x$ to $y$ proportional to $e^{-|x-y|}$.  To see that the velocity is zero, set $\bbQ(d\o)= \frac{r_0(\o) }{ \bbE [r_0(\o) ]} \bbP( d \o)$. $\bbQ$ is  a reversible and  ergodic distribution for   the environment viewed from the discrete time  Mott random walker  $Y_n$ {(see \cite{CF}). By writing  $Y_n$ as an additive function  of the process ``environment viewed from the walker" and using the ergodicity of $\bbQ$, one gets that  $v_Y(\l=0)$ is zero a.s.,  for $\bbQ$--a.a.~$\o$ and therefore for $\bbP$--a.a.~$\o$. }
  Similarly, $v_\bbY(\l=0)=0$ a.s., for $\bbP$--a.a.~$\o$ (use that $\bbP$ is    reversible and  ergodic  for   the environment viewed from $\bbY_t$, see  \cite{FSS}).
%
  \end{Remark}


\begin{Remark}
If the random variables $Z_k$  are i.i.d.~(or even   when only   $Z_k,Z_{k+1}$ are independent  for every $k$)  and $u$ is continuous, the above theorem implies the following dichotomy:  $v_\bbY (\l)>0$ if and only if  $\bbE\bigl[ e^{ (1-\l) Z_0}  \bigr] <+\infty $, otherwise  $v_\bbY (\l)=0$. The same holds for $v_Y(\l)$.
  We point out that, if the $Z_k$'s are correlated,   $\bbE\bigl[ e^{ (1-\l) Z_0}  \bigr] =+\infty $ does not imply in general zero velocity (see Example \ref{velocino} in Section \ref{esaedro}).  
\end{Remark}

\begin{Remark}
 Theorem \ref{teo1} shows that there are cases in which the limiting speed $v_{\bbY}(\l)$ is not continuous in $\lambda$. See  Example \ref{sorpresa} in Section \ref{esaedro}. \end{Remark}

\begin{Remark}\label{remark:nearest}
When  considering  the  nearest neighbor random walk on $\{x_k\}_{k \in \bbZ}$ with probability rate for a jump from $x$ to a  neighboring site $y$ given by \eqref{eq_eq_eq},   the random walk is ballistic if and only if 
\begin{equation}\label{musica}
\sum _{i=1}^\infty  \exp \bigl\{
(1-\l) Z_0 - (1+\l ) Z_{-i} - 2 \l ( Z_{-i+1}+ \dots+ Z_{-1} )
\bigr\} < \infty\,. 
\end{equation}
A proof of this fact is given in Appendix   \ref{app_nn}.
As outlined in Remark \ref{estensione}, one can indeed  weaken the condition $\bbE\bigl[ e^{ (1-\l) Z_0}  \bigr] <+\infty $ to prove ballisticity, albeit at the cost of dealing with rather ugly formulas having some analogy with \eqref{musica}.
\end{Remark}

One of the most interesting technical results we use in the proof of Theorem \ref{teo1}, Part (ii), concerns the invariant measure for the environment seen from the point of view of the walker:

\begin{TheoremA}\label{teo2} Suppose that    $\bbE\bigl[e^{ (1-\l) Z_0}  \bigr] <+\infty $  and $u :\bbR \times \bbR  \to \bbR$ is  continuous.  Then the following holds:
\begin{itemize}
\item[(i)]
The environment viewed from the discrete time Mott random  walker $(Y_n)_{n \geq 0}$, i.e. the process $\bigl(\t_{ \phi (Y_n) } \o\bigr)_{n \geq 0}$ where $\phi(x_i)=i$, admits an invariant and ergodic  distribution $\bbQ^\infty$ absolutely continuous to $\bbP$ such that  
\begin{equation}
0<\gamma\leq \frac{d \bbQ^\infty}{ d \bbP} \leq F \,, \qquad \bbP\text{--a.s.}  
\end{equation}
for a suitable universal constant $\g$ and a function $F \in L^1( \bbP)$ (defined in \eqref{def_F}). 

\item[(ii)]
 By writing $\bbE ^\infty $ for the expectation with respect to $\bbQ^\infty$, the velocities $v_{\bbY}(\l)$, $v_Y(\l)$ can be expressed as 
 \begin{align} 
& v_\bbY(\l)=  v_Y(\l)  \bbE^\infty\Big[  1/ r_0 ^\l(\o) \Big]\,,\label{adsl1}  \\
&  v_Y(\l)= \bbE[Z_0] \bbE^\infty  \Big[ \sum _{k \in \bbZ} k \,p_{0,x_k}^\l (\o) \Big]\,,\label{adsl2}
 \end{align}
and  the expectations in \eqref{adsl1}, \eqref{adsl2} are finite and positive (recall that $ r_0 ^\l(\o) = \sum_{k } r_{0,k} ^\l (\o)$).
\end{itemize}
 \end{TheoremA}
 \begin{proof}
Theorem \ref{teo2}--(i) is part of Proposition \ref{digiuno} given at the end of Section \ref{tronco}. The proof of Theorem \ref{teo2}--(ii)  is part of Section \ref{ballistic_part}, more precisely \eqref{adsl1} and \eqref{adsl2} are an immediate consequence of \eqref{int_rapr}, \eqref{quasi_fatto} and the observation just after \eqref{radisson}.
 \end{proof}







\subsection{Comments on assumptions (A1)--(A4)}\label{domenica}
By Assumption (A1) both random sequences $(Z_k)_{k \in \bbZ}$ and $(E_k)_{k \in \bbZ}$ are stationary and ergodic with respect to shifts.  
The  physically interesting case is given by  two independent   random sequences $(Z_k)_{k \in \bbZ}$ and $(E_k)_{k \in \bbZ}$, the former  stationary and ergodic,  while the latter  given by i.i.d.~random variables.  In this case assumption (A1) is satisfied (see  Lemma \ref{indi} in 
 Appendix  \ref{appendix:ergodic}).

Assumption (A3) ensures that a.s.~the environment is not periodic. If the energy marks $E_k$ are i.i.d.~and non-constant, as in the physically interesting case, then (A3) is automatically fulfilled. Note that  the sequence $(Z_k)_{k \in \bbZ}$ could be periodic, without violating our assumptions (e.g.~take $Z_k =1$ for all $k \in \bbZ$).  

 Assumption (A4), corresponding to interpoint distances which are  uniformly bounded from below,   is not restrictive  from a physical viewpoint  and $d$ can be taken of the angstrom  order.  On the other hand, 
(A4) plays a crucial role in our proofs.

  \subsection{Examples}\label{esaedro} 

In this section we give two examples highlighting the importance of the assumptions in Theorem \ref{teo1} and showing some of its consequences.

\begin{example}\label{velocino} 
$\bbE\bigl[ e^{ (1-\l) Z_0}  \bigr] = \infty $ does in general {\sl not} imply that  $v_\bbY(\l)=0$, $v_Y(\l) = 0$.

\smallskip

\textnormal{We set $u(\cdot, \cdot) \equiv 0$ and take $p\in (0,1/2)$. We choose  $(Z_k)_{k \in \bbZ}$ as the reversible   Markov chain with values in $\{\g,2\g, 3\g, \ldots \}$ for some {$\g\geq 1$} and with transition probabilities defined as follows:  
\[ \begin{cases}
	P(k\g, (k+1)\g)=p  & \text{ for $k \geq 1$}\,,\\
	P(k\g, (k-1)\g )=1-p & \text{  for $k \geq 2$}\,,\\
 	P(\g,\g)=1-p\,.
 \end{cases}
 \] 
The equilibrium distribution is given by $\pi(k\g) = c(p/(1-p))^k$ for $k \geq 1$, $c$ being the normalizing constant.
Hence, $\bbP (Z_0=k\g) = \pi(k\g )$, for each $k\geq 1 $. Notice that $\bbE\bigl[ e^{ (1-\l) Z_0}  \bigr] = c \sum_{k\geq1} e^{(1-\l)k \g} \bigl(p/(1-p)\bigr)^k$ is infinite if and only if 
\begin{align} \label{eq:p_condition}
\l \leq 1-\frac{1}{\g} \log\frac{1-p}{p}\,.
\end{align}
}
\textnormal{  
We now show that we can choose the parameters such that  $\bbE\bigl[ e^{ (1-\l) Z_0}  \bigr] =\infty$ and   $\sum_k k r^\l_{0,x_k} (\o) >0$ for each $\omega$, the latter implying that $v_\bbY(\l), v_Y(\l)>0$ due to Theorem \ref{teo2}--(ii)  and the definition of $p_{0,x_k}^\l (\o)$  in \eqref{treno}.
\\
If $Z_0=j\g$, for some $j\geq 1$,
 the local drift  $\sum_k k r^\l_{0,x_k} (\o) $ can be bounded from below by the drift of the configuration with longer and longer interpoint distances to the right and shorter and shorter interpoint distances to the left: $Z_k=(j+k)\g$ for all $k\geq -j+1$ and $Z_k=\g$ for all $k\leq -j$. 
Note that in this case
\[
\begin{cases}
 x_k = \g \big[kj + \frac{k(k-1)}{2} \bigr]   & \text{ if } k \geq 1\,,\\
 x_{-k } = - \g \big[ kj - \frac{k(k+1)}{2} \big]  & \text{ if  } 1\leq k \leq j-1\,,\\
x_{-k} = - \g \big[ \frac{j(j-1)}{2} +k-j+1\big] & \text{ if } k \geq j \,.
\end{cases}
\]
Hence we can write
$$
\sum_k k r^\l_{0,x_k} (\o) \geq A(\l,\g,j)-B(\l,\g,j)-C(\l,\g,j),
$$
where $A(\l,\g,j)=\sum\limits_{ k\geq 1}k\,\exp\{-(1-\l)\g(kj+k(k-1)/2)\}$, $B(\l,\g,j)=\sum\limits_{1\leq k\leq j-1}k\,\exp\{-(1+\l)\g(kj-k(k+1)/2)\}$ and $C(\l,\g,j)=\sum\limits_{ k\geq j}k\,\exp\{-(1+\l)\g(j(j-1)/2+k-j+1)\}$.
We bound  $A(\l,\g,j)$ from {below}   with its first summand $\exp\{-(1-\l)\g j\}$ and prove that
\begin{align}\label{eq:ABC}
\lim_{\g\to\infty}\sup_{j\in\N}\big[B(\l,\g,j)+C(\l,\g,j)\big]\exp\{(1-\l)\g j\}<1,
\end{align}
since this will imply the positivity of the local drift  $\sum_k k r^\l_{0,x_k} (\o)$ for any possible $\o$, for $\g$ big enough. 
  Using that $Z_{-1}= \g( j-1)$  we bound $B(\l,\g,j)\leq j^2\exp\{-(1+\l)\g(j-1)\}$  and, using that $j(j-1)/2 + 1 \geq j/2$, we bound 
\begin{align*}
C(\l,\g,j)
	&\leq {\rm e}^{-\frac{(1+\l)}{2}\g j}\Big(\sum_{k\geq j}(k-j){\rm e}^{-(1+\l)\g (k-j)}+j\sum_{{k\geq 0}}{\rm e}^{-(1+\l)\g k}\Big)\leq jK{\rm e}^{-\tfrac{(1+\l)}2\g j},
\end{align*}
for some constant $K>0$ independent of $\l$ and $\g$ {(recall that $\g \geq 1$)}. With these bounds we see that \eqref{eq:ABC} holds as soon as $\l>1/3$, for $\g$ big enough. On the other hand, by \eqref{eq:p_condition} we can choose $p$ close enough to $1/2$ so that $\bbE\bigl[ e^{ (1-\l) Z_0}  \bigr]$  is infinite.
}
\end{example}

%
%
%
%
%
%
%
\begin{example}\label{sorpresa} The velocities 
  $v_{\bbY} (\lambda)$, $v_Y (\l)$  are  not continuous in general.
  
 \smallskip
 
\textnormal{Take  $u \equiv 0$.   Let the  $Z_k$ be  i.i.d.~random variables  larger than $1$ such that $e^{Z_0}$ has probability density 
$f(x) :=   \frac{c}{ x^\g (\ln x)^2} \mathds{1}_{  [e,+\infty)} (x)$, with  $\g \in (1,2)$ and $c$ is the normalizing constant. Since, for $x \geq e$,   $ \frac{1}{x^\g (\ln x)^2 }\leq \frac{1}{x  (\ln x)^2 }= -\frac{d}{dx}( 1/\ln x )$, the constant $c$ is well defined and $E[ e^{ (1-\l)  Z_0} ]= \int_e^\infty  \frac{c x^{1-\l}  }{ x^{\g} (\ln x)^2}dx < \infty  $ if and only if  $\l\geq 2-\g$. Note that $\l_c:= 2-\g \in (0,1)$. Then the above observations and Theorem \ref{teo1} imply that $v_\bbY (\l),v_Y (\l)$ are zero for $\l \in (0, \l_c)$ and are strictly positive for $ \l \in [\l_c , 1)$.
  }
\end{example} 


\section{A class of random walks on $\bbZ$ with jumps of size at most $\rho \in[1,+\infty]$}\label{zetino}

Take $\l \in [0,1)$.
Given $i,j \in \bbZ$ we replace, with a slight abuse of notation,  $  r^\l_{i,j}(\o) := r^\l _{x_i,x_j}(\o)$  and the associated conductance $ c_{i,j}  (\o):=e^{2 \l x_i} r^\l_{i,j}(\o) $ (note that in $c_{i,j}(\o)$  the dependence on $\l$ has been omitted).
Hence we have $c_{i,i} \equiv 0$ and
\begin{equation}\label{defofcond}
{c_{i,j} (\o)=
 e^{ \l(x_i+x_j) -|x_j-x_i  |+u (E_i,E_j)} = c_{j,i}  (\o)  \,\,  \qquad i \not = j \text{ in } \bbZ\,.}
 \end{equation}
Given   $\rho \in \N_+\cup \{+\infty\}$  we introduce the discrete time random walk $(X^{ \rho} _n)_{n \geq 0} $  with environment $\o$
as the Markov chain on $\bbZ$ 
such that the $\o$--dependent probability to jump from $i$ to $j$ in one step is given by
\begin{equation}\label{salto}
\begin{cases} c _{i,j}(\o)/\sum_{k\in\bbZ}c _{i,k}(\o), & \mbox{if }0 < |i-j|\leq \rho \\
0 & \mbox{if } |i-j|>\rho  \\
1- \sum\limits_{j: |j-i| \leq \rho }
 c _{i,j}(\o)/\sum_{k\in\bbZ}c_{i,k}(\o)
 & \mbox{if } i=j.
\end{cases}
\end{equation}

\begin{Warning}\label{attenzione}
When the Markov chain
$(X_n^\rho)_{n \geq 0}$ starts at $i\in \bbZ$, we write $P^{\o, \rho}_i$ for its law and $E^{\o, \rho}_i$ for the associated expectation.  In order to make the notation lighter, inside $P^{\o, \rho}_i(\cdot)$ and $E^{\o, \rho}_i[\cdot ]$ we will  usually write $X_n$ instead of $X_n^\rho$.
\end{Warning}

\medskip

It is convenient to introduce the random bijection $\psi: \bbZ \to \{ x_k\}_{k \in \bbZ}  $   defined as $\psi(i)= x_i$, and also the continuous time random walk $(\bbX^\infty _t)_{t \geq 0}$ on $\bbZ$ with probability rate  $r^\l_{i,j}(\o)$ for a jump from $i$ to $j$.
Since
\[ \frac{c_{i,j}(\o)}{\sum_{k\in\bbZ}c _{i,k}(\o)}=\frac{r^\l _{i,j}(\o)}{\sum_{k\in\bbZ}r^\l _{i,k}(\o)}\,,
\]
we conclude that realizations of $\bbY$ and $Y$ can be obtained as 
\begin{equation}\label{9898}
\bbY _t = \psi (\bbX^\infty _t) \,, \qquad  Y_n= \psi(X^\infty _n)\,.
\end{equation}
In particular, when the denominators are nonzero, we can write
\begin{equation}\label{7575}
\frac{\bbY_t}{t} =  \frac{ \psi (\bbX^\infty _t)  }{ \bbX^\infty _t} \frac{\bbX^\infty _t}{t}\,, \qquad
 \frac{Y_n}{n} =  \frac{ \psi (X^\infty _n)  }{ X^\infty _n} \frac{X^\infty _n}{n}\,.
\end{equation} 
By Assumptions (A1) and (A2), $ \lim _{ i \to \infty} \psi(i) /i= \bbE[Z_0]<\infty $,  $\bbP$--a.s.. By this limit, together with \eqref{9898} and \eqref{7575}, we will  see in Sections \ref{ballistic_part},  \ref{submarine}  that  in order  to prove Theorem \ref{teo1} it is enough to show the same properties for $\bbX ^\infty, X^\infty$ instead of $\bbY,Y$.

In what follows, we write $v_{X^\rho}(\l)=v$ if, for $\bbP$--a.a.~$\o$,  $\lim _{n \to \infty} \frac{ X_n^\rho}{n}=v$ { $P^{\o, \rho}_0$}--a.s.. A similar meaning is assigned to  $v_{\bbX^\rho}(\l)$.

\subsection{Estimates on effective conductances}\label{condottiero}


We take again $\rho \in \N_+ \cup \{+\infty\}$ and $\l \in [0, 1)$.
For $A,B$ disjoint subsets of $\Z$, we introduce the effective conductance between $A$ and $B$ as
\begin{equation}
C^{ \rho } _{\rm eff}(A\leftrightarrow B)
	:=\min\Big\{\sum_{i<j\in\Z: \,|i-j| \leq \rho }c_{i,j}(f(j)-f(i))^2:\,f|_A=0,\,f|_B=1\Big\}\,.
\end{equation}

We also set 
\begin{equation}\label{def_pi}
\pi^{ \rho } (i):=\sum_{j\in\Z\,:\, |j-i| \leq \rho }c_{i,j}\,, \qquad i \in \bbZ\,,
\end{equation}
and define $p_{\rm esc}^{  \rho}(i) $ as the escape probability  of $ X_n^{ \rho} $ from $i\in \bbZ$, i.e.
\begin{equation} p_{\rm esc}^{  \rho}(i) := P _i ^{\o, \rho} ( X_n\not =i \text{ for all } n \geq 1)\,.
\end{equation}
It is known (see the discussion before Theorem 2.3 in \cite[Section 2.2]{LP} and formula (2.4) therein)
that
\begin{equation}\label{fuga}
p_{\rm esc}^{ \rho}(i) := \lim_{N\to\infty}\frac{C^{  \rho} _{\rm eff}(i\leftrightarrow (-\infty,i-N]\cup[i+N,\infty))}{\pi^{ \infty}(i)}
\end{equation}
(recall that the probability for $X_n^{ \rho}$ to jump from $i$ to $j$ (for $0 < |i-j|\leq \rho $)  is given by $c_{i,j} / \pi ^{ \infty}(i)$, cf. \eqref{salto}). We will see (cf.~Corollary  \ref{cor:p_esc}) that the escape probability of each $\rho$--random walk can be uniformly bounded from below and above by the escape probability of the nearest neighbor walk times constants.

\begin{Warning} \label{carnevale}Note that  $C^{ \rho } _{\rm eff}(A\leftrightarrow B)$, $\pi^{ \rho } (i)$ and $p_{\rm esc}^{ \rho}(i)$ depend on the environment $\o$, although we have omitted $\o$ from the notation.
\end{Warning}

\begin{Proposition}\label{prop:effconductance_bound}
There exists  a constant $ K>0$ not depending on $\o, \, \rho,\,A$ and $B$ such that
$$
C^{ 1}_{\rm eff}(A\leftrightarrow B)\leq C^{ \rho} _{\rm eff}(A\leftrightarrow B)\leq K\,C^{ 1}_{\rm eff}(A\leftrightarrow B).
$$
\end{Proposition}

\begin{proof}
 Since $ C^{\rho} _{\rm eff}(A\leftrightarrow B)$ is increasing in $\rho$,   it is enough to show the second inequality for $\rho=\infty$. To this aim
take any valid $f:\Z\to\R$ and note that
\begin{align}
\sum_{i<j\in\Z}c _{i,j}(f(j)-f(i))^2
&=\sum_{i<j\in\Z}c_{i,j}\Big(\sum_{z=i}^{j-1}f(z+1)-f(z)\Big)^2 \nonumber\\
&\leq\sum_{i<j\in\Z}c_{i,j} \cdot(j-i)\sum_{z=i}^{j-1}\big(f(z+1)-f(z)\big)^2 \nonumber\\
&=\sum_{z\in\Z}\big(f(z+1)-f(z)\big)^2 \Big[\sum_{i\leq z}\sum_{j\geq z+1}c_{i,j}\cdot(j-i)\Big],
\end{align}
where we have used the Cauchy-Schwarz inequality for the second step.
Define the new conductances
$$
\bar c_{z,z+1}=\sum_{i\leq z}\sum_{j\geq z+1}c_{i,j}\cdot(j-i).
$$
Now we are left to show that $\bar c_{z,z+1} \leq K\,c_{z,z+1}$ for some $K$ 
and this will conclude the proof. Using the fact that $\forall k>k'$ we have $x_k-x_{k'}>d\,(k-k')$, we have
{\begin{align}
\bar c_{z,z+1}
	&=	\sum_{i\leq z}\sum_{j\geq z+1}{\rm e}^{\l (x_i+x_j)-(x_j-x_i)+u(E_i,E_j)}(j-i) \nonumber\\
	&\leq {\rm e}^{\l (x_z+x_{z+1})-(x_{z+1}-x_z)+u_{\max}}\sum_{i\leq z}\sum_{j\geq z+1}{\rm e}^{-(x_j-x_{z+1})(1-\l)}{\rm e}^{-(x_z-x_i)(1+\l)}(j-i) \nonumber\\
	&\leq c_{z,z+1} {\rm e}^{u_{\max}} \sum_{i\leq z}\sum_{j\geq z+1}{\rm e}^{-d(j-z-1)(1-\l)}{\rm e}^{-d(z-i)(1+\l)}(j-i) \nonumber\\
	&\leq c_{z,z+1} {\rm e}^{u_{\max}} \sum_{l=0}^\infty \sum_{h=0}^\infty{\rm e}^{-d(1-\l)h}{\rm e}^{-d(1+\l)l}(h+l+1). \nonumber
\end{align}}
Since the last double sum is bounded for each $\l\in [0,1)$, we obtain the claim.
\end{proof}

\begin{Lemma}\label{lorenzo}
 There exists  a constant $K>0$ which does not depend on $\o, \, \rho$ such that
 $$ \pi^{ 1} (k) \leq \pi^{  \rho}(k) \leq K  \pi^{ 1} (k) \,, \qquad \forall  k \in \bbZ\,.
 $$
\end{Lemma}
\begin{proof} Since $\pi^{ \rho}(k) $ is increasing in $\rho$ it is enough to prove that  $\pi^{ \infty}(k) \leq K  \pi^{ 1} (k)$ for all $k\in \bbZ$.
 We easily see that
{\begin{equation}\label{rione1}
\begin{split}
\sum_{j>k}c_{k,j}&
	={\rm e}^{\l (x_{k+1}+x_k)-(x_{k+1}-x_k)} \sum_{j>k}{\rm e}^{- (1-\l)(x_j-x_{k+1})+u(E_j,E_k)}\\
	&\leq K_1\,c_{k,k+1}\sum_{j>k} {\rm e}^{-d (j-k-1)(1-\l)} =K_2 \,c_{k,k+1}.
\end{split}
\end{equation}}
Analogously,
{\begin{equation}\label{rione2}
\begin{split}
	\sum_{j<k}c_{k,j}
	&={\rm e}^{\l (x_{k-1}+x_k)-(x_{k}-x_{k-1})} \sum_{j<k}{\rm e}^{- (x_{k-1}-x_j)(1-\l)+			u(E_j,E_k)}\\
	&\leq K_3\,c_{k-1,k}\sum_{j>k} {\rm e}^{- d(k-1-j))(1-\l)} =K_4 \,c_{k-1,k}.
\end{split}
\end{equation}}
\vspace{-0.25cm}
\end{proof}
As a byproduct of \eqref{fuga}, Prop. \ref{prop:effconductance_bound} and Lemma \ref{lorenzo}, we get:
\begin{Corollary}\label{cor:p_esc}
There exist constants $K_1,K_2>0$  which do not depend on $\o, \, \rho$ such that
\[
K_1  p_{\rm esc}^{1} (i)\leq p_{\rm esc}^{ \rho } (i)  \leq K_2  p_{\rm esc}^{ 1} (i)\,, \qquad \forall i \in \bbZ\,.
\]
\end{Corollary}

The following lemma is well known and corresponds to formula (2.1.4) in \cite{OferStFlour}:
\begin{Lemma}\label{lemma:effconductance_nn}
Let $\{\bar c_{k,k+1}\}_{k\in\Z}$ be any system of strictly positive conductances on the nearest neighbor bonds of  $\Z$. Let $H_A$ be the first hitting time of the set $A\subset \Z$ for the associated discrete time  nearest--neighbor random walk among the conductances $\{\bar c_{k,k+1}\}_{k\in\Z}$, which jumps  from  $k$ to $k\pm 1$ with probability
$\bar c_{k, k\pm 1} / ( \bar c_{k,k-1} +\bar c_{k,k+1})$.
Take $-\infty<M<x<N<\infty$, with $M,x,N\in\Z$ and write $H_M, H_N$ for $H_{\{M\}}, H_{\{N\}}$. Then
\[
\label{hittingtimes}
P^{n.n.}_x(H_M< H_N)=
	\frac{C_{\rm eff}^{n.n.}(x\leftrightarrow (-\infty,M])}{C_{\rm eff}^{n.n.}(x\leftrightarrow (-\infty,M]\cup[N,\infty))},
\]
where $P^{n.n.}_x$ is the probability for the nearest-neighbor random walk starting at $x$, and $C_{\rm eff}^{n.n.}(A \leftrightarrow B)$ is the effective conductance of the nearest-neighbor walk between $A$ and $B$.
\end{Lemma}

 We state another technical lemma which will be frequently used when dealing with conductances:
\begin{Lemma}\label{goldie}
 $\bbP\left( \sum_{j=0}^\infty \frac{1}{c_{j,j+1} } <+\infty\right)= 1$.
\end{Lemma}
\begin{proof}
By assumption (A1), $(x_{j+1}-x_j) _{j \in \bbZ}$ is a stationary ergodic sequence. By writing $x_j= \sum _{k=0}^{j-1}(x_{k+1}-x_k)$, the ergodic theorem  implies that
$ \lim _{j \to \infty} \frac{x_j}{j}= \bbE[x_1]$,  $ \bbP$--a.s.
As a consequence we get that
\[ \lim _{j \to \infty}  \frac{ - \l (x_j+x_{j+1} )+ (x_{j+1}-x_j)}{j}= -2 \l \bbE[x_1]< 0\,, \qquad \bbP\text{--a.s.}
\]
Since $ \sum_{j=0}^\infty \frac{1}{c_{j,j+1} } = \sum_{j=0}^\infty {\rm e}^{- \l (x_j+x_{j+1} )+ (x_{j+1}-x_j) }$, the 
claim follows.
\end{proof}
We conclude this section with a simple estimate leading to an exponential decay of the transition probabilities:
\begin{Lemma}\label{pierpaolo} There exists a constant $K$ which does not depend on $ \o , \rho$, such that $\bbP$--a.s.
\begin{equation}\label{eqn:condition_c}
P ^{\o, \rho}_i ( |X_1-i| >s) \leq     \sum_{j:|j-i|>s}\frac{c _{i,j} }{\sum_{k\in\Z}c_{i,k} }\leq
K e^{-d s (1-\l) }   \qquad \forall s,\rho  \in\N_+\cup \{+\infty\} \,, \; \forall i \in \bbZ\,.
 \end{equation}
\end{Lemma}
\begin{proof} The first inequality follows from the definitions. To prove the second one, 
 we can estimate 
\begin{equation}\label{rione10}
\begin{split}
\sum_{j>i+s}c_{i,j}&
	={\rm e}^{\l (x_{i+1}+x_i)-(x_{i+1}-x_i)}  \sum_{j>i+s}  
	{{\rm e}^{- (x_j-x_{i+1})(1-\l)+u(E_j,E_i)}} \\
	&\leq K_1\,c_{i,i+1}\sum_{j>i+s} {\rm e}^{-d (j-i-1)(1-\l)} =K_2 e^{-d s (1-\l) } c_{i,i+1}\,,
\end{split}
\end{equation}
\begin{equation}\label{rione20}
\begin{split}
\sum_{j<i-s}c_{i,j}&
	={\rm e}^{\l (x_{i-1}+x_i)-(x_{i}-x_{i-1})} \sum_{j<i-s}
	{ {\rm e}^{- (x_{i-1}-x_j)(1-\l)+u(E_j,E_i)}}\\
	&\leq K_3\,c_{i-1,i}^\l\sum_{j>i-s} {\rm e}^{- d(i-1-j)(1-\l)} =K_4e^{-d s (1-\l) }  c_{i-1,i}.
\end{split}
\end{equation}
The second bound in \eqref{eqn:condition_c} now follows from \eqref{rione10}, \eqref{rione20}  and Lemma \ref{lorenzo}.
\end{proof}

\subsection{Expected number of visits}\label{ECD}
We fix some notations which will be frequently used below.
For $I\subseteq \{0,1,2, \dots \}$ and $A\subset \Z$,  we denote by $N^\rho _I(A)$ the time spent by the random walk $X_n^\rho$ in the set $A$ during the time interval $I$:
\begin{equation*}
N^{\rho}_I(A):=\sum_{k\in I}\mathds{1}_{X^{ \rho} _k\in A}\,.
\end{equation*} If $I:= \{0,1,2,\dots\}$ we simply write $N_\infty^{ \rho} (A)$ and if $A=\{x\}$ we write
 $N_\infty^{ \rho} (x)$.

\begin{Warning}\label{att1}
When appearing inside $\bbP^{\o, \rho}(\cdot)$ or $\bbE^{\o, \rho}(\cdot)$, $N_I(A), N_\infty (A)$  will usually  replace $N^\rho_I(A), N^\rho_\infty (A)$.
\end{Warning}



We can state our main result on the expected number of visits to a site $k$ for a given environment:
\begin{Proposition}\label{prop:function_g}
There exists a constant $K_0$, not depending on $\l, \rho, \o$, such that the function
$g_\o: \{0,1,\dots\}\to \bbR_+ $,  defined as
\begin{equation}\label{def_g}g_\o(n):= K_0 \pi^1 (-n) \sum _{j=0}^\infty
{\rm e}^{-2 \l x_j+ (1-\l)(x_{j+1}-x_j)}, \qquad n \geq 0\,,
\end{equation}
satisfies
\begin{equation}\label{sicilia27}
E^{\o,\rho }_0[N_\infty(k)]\leq g_\o (|k|) \,, \qquad \forall k \leq 0\,.
\end{equation} 
\end{Proposition}
We recall that $\pi^1 (k)=c_{k-1,k}+c_{k,k+1}$ for all $k \in \bbZ$.  Moreover,
we point out that the series in \eqref{def_g} is finite, since it can be  bounded by $  \sum _{j=0}^\infty
\exp\bigl\{-j (2\l d - (1-\l)\tfrac{x_{j+1}-x_j}{j}) \bigr\} $, while $(x_{j+1}-x_j)/j \to 0$ for $\bbP$--a.a.~$\o$ (see the argument in the proof of Lemma \ref{goldie}).

We remark that   estimate \eqref{sicilia27} is not uniform in the environment $\o$, and in general one cannot expect   a uniform bound. This technical fact represents a major difference with the setting of \cite{CP}, where the existence of an $\o$-independent upper bound of the expected number of visits is required (cf.~Condition D therein).
  
\begin{proof}Fix $k<0$.
Starting from $0$, the random variable $N_\infty^\rho(k)$ is equal to
$$
N_\infty^\rho(k)=\begin{cases} 0 & \mbox{with probability } 1-P^{\o,\rho}_0(\mbox{ $X_\cdot$  eventually reaches  $k$ }) \\
Y(k) & \mbox{with probability } P^{\o,\rho}_0(\mbox{ $X_\cdot$  eventually reaches  $k$ })
\end{cases}
$$
where $Y(k)$ is a geometric random variable whose parameter is the escape probability $p^{\rho} _{\rm esc}(k)$ from $k$ (recall Warning \ref{carnevale}).
Therefore
\begin{equation}\label{annamaria}
E^{\o,\rho}_0[N_\infty (k)]=
\frac{1}{p^{ \rho}_{\rm esc}(k)} P^{\o,\rho}_0(\mbox{ $X_\cdot$  eventually reaches  $k$ }).
\end{equation}
Let us start by giving an upper bound for the probability of reaching $k$ in finite time:
\begin{align}\label{eqn:reachx}
P^{\o,\rho}_0(\mbox{ $X_\cdot$  eventually reaches $k$ }) &\leq P^{\o,\rho}_0(\mbox{ $X_\cdot$  eventually reaches  $A:=(-\infty,k]$ }) \nonumber\\
	&=\lim_{N\to\infty}P^{\o,\rho}_0(H_{B_N}> H_A ),
\end{align}
where $B_N:=[N,\infty)$ and the $H$'s are the hitting times of the respective sets. By a well-known formula (see
\cite[Proof of Fact 2]{BGP})
\begin{equation}\label{eqn:taus1}
P^{\o,\rho}_0(H_{B_N}>H_A )\leq\frac{C^{ \rho}_{\rm eff}(0\leftrightarrow A) }
{C^{ \rho}_{\rm eff}(0\leftrightarrow A\cup B_N) }\,.
\end{equation}
%
Using now Proposition \ref{prop:effconductance_bound} we have that there exists a $K$ such that
\begin{align}\label{eqn:taus2}
P^{\o,\rho}_0(\mbox{ $X_\cdot$  eventually reaches $k$ })
	& \leq \lim_{N\to\infty}K\frac{C_{\rm eff}^{1}(0\leftrightarrow A) }
		{C_{\rm eff}^{1}(0\leftrightarrow A\cup B_N) }\nonumber\\
	&=K\frac{C_{\rm eff}^{1}(0\leftrightarrow A) }{C_{\rm eff}^{1}(0\leftrightarrow A\cup B_\infty) }.
\end{align}
where $C_{\rm eff}^{1}(0\leftrightarrow A\cup B_\infty):= \lim_{N \to \infty} C_{\rm eff}^{1}(0\leftrightarrow A\cup B_N)$.

\smallskip

Call $C_N:=(-\infty,-N+k]\cup[N+k,\infty)$. By Corollary \ref{cor:p_esc} and equation \eqref{fuga}, we know that
\begin{equation}\label{giovanni}
p^{\rho}_{\rm esc}(k)\geq  \frac 1K\, \lim_{N\to\infty}\frac{C_{\rm eff}^{1}(k\leftrightarrow C_N)}{\pi^1(k)}=\frac 1K\,\frac{C_{\rm eff}^{1}(k\leftrightarrow C_\infty)}{\pi^{1}(k)}\,,
\end{equation}
where $C_{\rm eff}^{1}(k\leftrightarrow C_\infty):= \lim _{N \to \infty} C_{\rm eff}^{1}(k\leftrightarrow C_N)$.

Since we have conductances in series, we can write 
\begin{equation}\label{decompongo}
	C_{\rm eff}^{1}(k\leftrightarrow C_\infty)=\Big(\displaystyle{\sum_{j=-\infty}^{k-1}}\tfrac{1}		{c_{j,j+1} }\Big)^{-1} +
		\Big(\displaystyle{\sum_{j=k}^{\infty}}\tfrac{1}{c_{j,j+1} }\Big)^{-1}\,.
\end{equation}
We claim that
\begin{equation}\label{renato_zero}
 \sum_{j=-\infty}^{k-1}\tfrac{1}{c_{j,j+1}} =+\infty\,, \qquad
\sum_{j=k}^{\infty}\tfrac{1}{c_{j,j+1} }<+\infty
 \qquad \text{$\bbP$--a.s.}
\end{equation}
Indeed, the first series diverges a.s.~since, for $j \leq -1$, $1/c_{j,j+1}  \geq  K e^{-\l(x_j+x_{j+1} ) +(x_{j+1}-x_j)} \geq K$ (note that $x_j, x_{j+1}\leq 0$). The second series is finite a.s.~due to Lemma \ref{goldie}.

\smallskip
Due to \eqref{annamaria}, \eqref{eqn:taus2}, \eqref{giovanni}, \eqref{decompongo} and  \eqref{renato_zero} we can write
\begin{align}
E^{\o,\rho}_0[N_\infty (k)]
	&\leq \bar K\,\frac{\pi^{1}(k)}{C_{\rm eff}^{1}(k\leftrightarrow C_\infty)}
		\cdot\frac{C_{\rm eff}^{1}(0\leftrightarrow A) }{C_{\rm eff}^{1}(0\leftrightarrow A\cup B_\infty) }\nonumber\\
	&=\bar K \,\frac{\pi^1(k) }
		{\Big(\displaystyle{\sum_{j=-\infty}^{k-1}}\tfrac{1}{c_{j,j+1}}\Big)^{-1} +
		\Big(\displaystyle{\sum_{j=k}^{\infty}}\tfrac{1}{c_{j,j+1}}\Big)^{-1}}
		\cdot\frac{\Big(\displaystyle{\sum_{j=k}^{-1}}\tfrac{1}{c_{j,j+1}}\Big)^{-1}}
		{\Big(\displaystyle{\sum_{j=k}^{-1}}\tfrac{1}{c_{j,j+1}}\Big)^{-1} +
		\Big(\displaystyle{\sum_{j=0}^{\infty}}\tfrac{1}{c_{j,j+1}}\Big)^{-1}}\nonumber\\
	&=\bar K\, \pi^1(k)
		\Big(\displaystyle{\sum_{j=0}^{\infty}}\tfrac{1}{c_{j,j+1}}\Big) \leq  K_0  \pi^1(k) \,\sum_{j=0}^{\infty} {\rm e}^{-\l(x_j+x_{j+1})+(x_{j+1}-x_j)}  \leq  g_\o  (|k|)\,.	
\label{riposino}
\end{align}

We now consider the case $k=0$. By \eqref{annamaria}, \eqref{giovanni} and \eqref{renato_zero}
 we have
\begin{equation*}
\begin{split}
 E^{\o,\rho}_0[N_\infty (0)]& =
\frac{1}{p^{  \rho}_{\rm esc}(0)}\leq K  \frac{\pi^{1}(0)}{C_{\rm eff}^{1}(0\leftrightarrow C_\infty)}
	 = K \pi^{1}(0) \displaystyle{\sum_{j=0}^{\infty}}\tfrac{1}{c_{j,j+1} }
\end{split}
\end{equation*}
and we can conclude as in \eqref{riposino}.
\end{proof}

We now collect some properties of the function $g_\o$:
\begin{Lemma}\label{piccolino} 
There exist constants $K_*>0$ which do not depend on $\rho, \o$, such that 
\begin{align}
& \pi^1 (k) \leq K_* e^{2 \l d k}  \,, \qquad k \leq 0\,, \label{spring0}\\
 &  \bbE[ g_\o (k) ] \leq K_*
 \frac{{\rm e}^{-2\l dk}}{1-e^{-2\l d}}\bbE\bigl[ e^{(1-\l)x_1}\bigr]\,, \qquad k \geq 0\,,\label{spring1}\\
 &  g_\o(k) \geq g_{\t_\ell \o} ( k + \ell)\,, \qquad k, \ell  \geq 0\,,\label{spring2}\\
 &  \bbE E_k^{\o, \rho }[N_\infty (\Z_-)] \leq
K_* \left( \frac{1 }{(1-e^{-2\l d})^2}+ \frac{|k|}{1-e^{-2\l d}}   \right) \bbE\bigl[ e^{(1-\l)x_1}\bigr]    \,, \qquad k \leq 0 \,. \label{spring3}
\end{align}
\end{Lemma}
Trivially, the first  and fourth  estimates are effective when $\bbE\bigl[ e^{(1-\l)x_1}\bigr]< \infty$.
\begin{proof} We first prove \eqref{spring0}. Recall $\pi^1(k)= c_{k-1,k}+c_{k,k+1}$.
 Given $i \leq 0 $ we have $x_i \leq i d$, implying 
{$c_{i-1,i} \leq {\rm e}^{u_{\rm max} } {\rm e}^{ \l (x_{i-1}+x_i) - (x_i-x_{i-1} ) }
 	\leq K e^{ 2 \l d i}
$.  }
By the same argument, for  $i <  0$ one gets $c_{i,i+1} \leq K e^{2 \l d i}$ and, for $i=0$,
  ${c_{0,1} = {\rm e} ^{ \l x_1 -x_1 +u(E_0, E_1)}\leq K}$.
\smallskip 

\eqref{spring1} is obtained noting that, by \eqref{spring0},   $\bbE[ g_\o (k) ] \leq K_*  e^{-2 \l d k } \sum_{j=0}^\infty e^{- 2 \l j d} \bbE\bigl[ e^{(1-\l)x_1}\bigr]$.


 \smallskip

To  get \eqref{spring2} we first observe that $x_{i-\ell} ( \t_\ell \o) = x_i (\o) - x _\ell (\o) $ and $E_{i- \ell} ( \t_\ell \o)= E_i (\o) $ for all $i \in \bbZ$. As a consequence,
we get $\pi^1 ( -k-\ell)[\t_\ell \o]= {\rm e }^{ - 2 \l x_\ell  } \pi^1 (-k)$ (the r.h.s.~refers to the environment $\o$). Therefore,  using also that $x_j(\t_\ell \o)= x_{j+\ell }(\o)- x_\ell(\o)$
and that $x_{j+1}(\t_\ell \o)-x_j(\t_\ell \o) = x_{j+1+\ell }( \o)-x_{j+\ell}( \o)$,
 we have
\begin{equation}\label{aurora100}
\begin{split}
g_{\t_\ell \o} (k+\ell)& =   K_0 \pi^1 (-k)  {\rm e }^{ - 2 \l x_\ell (\o) }   \sum _{j=0}^\infty
{\rm e}^{-2\l x_j( \t_\ell \o) + (1-\l)(x_{j+1}(\t_\ell \o)-x_j(\t_\ell \o) )}\\
& = K_0 \pi^1 (-k)  \sum _{j=0}^\infty
{\rm e}^{-2\l x_{j+\ell} + (1-\l)(x_{j+1+\ell }-x_{j+\ell} )}  \leq g_\o ( k)\,,
\end{split}
\end{equation}
thus completing the proof of \eqref{spring2}.

Finally, for \eqref{spring3}, we write, thanks to Proposition \ref{prop:function_g},
\begin{align*}
\bbE E_k^{\o, \rho }[N_\infty (\Z_-)]
	&=\sum_{z\leq k }\bbE E_k^{\o, \rho }[N_\infty (z)]+\sum_{k <z\leq 0}\bbE E_k^{\o,\rho}[N_\infty(z)] \nonumber\\
	&\leq\sum_{z\leq k}\bbE E_k^{\o, \rho}[N_\infty (z)]+\sum_{k<z\leq 0}\bbE E_z^{\o, \rho }[N_\infty (z)] \qquad 		
		\mbox{(Markov Property)}\nonumber\\
	&\leq \sum_{i\geq 0}  \bbE[ g_{ \tau_k\o}(i )]+|k|\,\bbE[g_\o(0)]\,,
\end{align*}
and the claim then follows from \eqref{spring1}.
\end{proof}

\begin{Remark}\label{estensione}  In the spirit of Remark \ref{remark:nearest}, we point out 
that we could consider weaker conditions than  $\bbE\bigl[ e^{(1-\l) Z_0} \bigr] <+\infty$, at the cost of dealing with rather involved formulas. Take for simplicity $u \equiv 0$.
In our case, $\bbE\bigl[ e^{(1-\l) Z_0} \bigr] <+\infty$ guarantees, by Lemma \ref{piccolino}, that $\bbE[ g_\o (k) ] $ is finite and summable over $k \geq 0$. But what is actually required is that $ g_\o (k)$  bounds from above  the quantity $ \a_\o(k):=K \pi^1(-k) \sum _{j \geq 0} \frac{1}{c_{j,j+1} }$ (see the proof of Prop.~\ref{prop:function_g}). By stationarity,  one has 
$$
\bbE [ c_{k, k+1} / c_{ k+i, k+i+1} ] = \bbE [ e^{-(1+\l) Z_0 -2 \l ( Z_1+ \dots+ Z_{i-1}) +(1-\l) Z_i } ].
$$ 
This identity  allows to provide conditions  for  $\sum_{k \geq0} \bbE[
 \a_\o(k)]$ {to be finite}, which are  weaker than $\bbE\bigl[ e^{(1-\l) Z_0} \bigr] <+\infty$. One could go on in weakening conditions, also  inside Prop.~\ref{aereo},  and still get the ballisticity of  the Mott random walks $\bbY_t$ and $Y_n$.

\end{Remark}

\begin{Corollary}\label{cor:comparison_number_visits}
There exist constants $K_1,K_2>0$ which do not depend on $ \rho, \o $ such that
\begin{align}
& E^{\o,\rho}_0[N_\infty (k)] \leq K_1 \, E^{\o,1}_0[N_\infty (k)] \leq K_2\, g_\o(|k|) &\forall k \leq 0\,,\label{venezia1}\\
&  E^{\o,\rho}_0[N_\infty (k)] \leq K_1 \,E^{\o,1}_0[N_\infty (k)]  & \forall k> 0\,.\label{venezia2}
\end{align}
\end{Corollary}
\begin{proof}  First we consider \eqref{venezia1}. Its second inequality  is a restatement of Prop.~\ref{prop:function_g}.
For the first inequality we distinguish the cases $k<0$ and $k=0$.
When $k<0$  note that  \eqref{eqn:taus2} and  Lemma \ref{lemma:effconductance_nn} imply that \begin{align}\label{eq:comparison_reach_k}
P^{\o,\rho}_0(\mbox{  $X_\cdot$  eventually reaches  $k$ })
	& \leq  K\,P^{\o,1}_0(\mbox{ $X_\cdot$ eventually reaches $k$ })\,.
\end{align}
Then
 put together equation \eqref{annamaria} (and its analogous version for $\rho=1$), equation \eqref{eq:comparison_reach_k} and  Corollary \ref{cor:p_esc}.
For $k=0$ use that $  E^{\o,\rho}_0[N_\infty (0)] =
\frac{1}{p^{ \rho}_{\rm esc}(0)}$ (also in the  case $\rho=1$) and use  Corollary \ref{cor:p_esc}.

Let us now consider equation \eqref{venezia2}.  Start with \eqref{annamaria}.
 Due to Corollary \ref{cor:p_esc}
 and the fact that $P^{\o,1}_0(\mbox{ $X_\cdot$  eventually reaches  $k$ })=1$ for each $k>0$  (cf.~Lemma \ref{sicilia} below)
 it is simple to conclude.
\end{proof}


\subsection{Probability to hit a site on the right}\label{subsec:hit}
Following \cite{CP}, given $k,z \in \bbZ$,  we set
\[ T_z^\rho:=\inf\{n\geq 0:\,X_n^\rho\geq z\}\,, \qquad  T^\rho:= T^\rho_0\,,\qquad  r^\rho_k(z):=P^{\o,\rho}_k(X_{T_z}=z)\,.\]
Note that the dependence of $\o$ has been omitted. Again (see Warnings \ref{attenzione} and  \ref{att1}),  we simply write   $T_z$, $r_k(z)$ inside
  $P^{\o,\rho}_k(\cdot )$, $E^{\o,\rho}_k(\cdot )$.

\begin{Lemma}\label{sicilia}
For $\bbP$--a.a.~$\o$ and for each $\rho \in \N_+ \cup \{\infty\}$  it holds 
that
$$ P^{\o, \rho}_k ( T_z  < \infty)=1  \qquad \forall k < z  \text{ in } \bbZ\,.
$$
\end{Lemma}
\begin{proof}
Without loss of generality we take $k<0=: z$ and  prove that $ P^{\o, \rho}_k ( T_0  = \infty)=0$.
As in \eqref{eqn:taus1}, setting  $C_N:= (-\infty, -N]$ and $D = [0,\infty)$,  we can bound
\[  P^{\o, \rho}_k ( T_0  = \infty) = \lim_{N \to \infty} \P^{\o, \rho}_k ( \t_{C_N} < \t _D) \leq \liminf _{N \to \infty} \frac{ C^{\rho}_{\rm eff} (k \leftrightarrow C_N) }{  C^{ \rho}_{\rm eff} (k \leftrightarrow C_N \cup D)} 
\,.
\]
We observe that $ C^{ 1}_{\rm eff} (k \leftrightarrow C_N \cup D)=  C^{ 1}_{\rm eff} (k \leftrightarrow C_N )+
C^{ 1}_{\rm eff} (k \leftrightarrow D)$, while (recall \eqref{renato_zero})
\[ 
\lim_{ N \to \infty}  C^{ 1}_{\rm eff} (k \leftrightarrow C_N ) =  \Big( \sum _{j=-\infty }^{k-1} \frac{1}{c_{j,j+1}} \Big)^{-1}=0\,, \qquad   C^{ 1}_{\rm eff} (k \leftrightarrow D) =\Big( \sum _{j=k}^{-1} \frac{1}{c_{j,j+1} } \Big)^{-1}> 0  \,.
\] 
Together with Proposition \ref{prop:effconductance_bound}, this allows to conclude that $P^{\o, \rho}_k ( T_0  = \infty) = 0$.
\end{proof}

Our next result,  Lemma \ref{lemma:epsilon}, is the analog of Lemma 3.1 in \cite{CP}. Our proof follows a different strategy in order to avoid to deal with Conditions  D, E  of \cite{CP},  which are not satisfied in our context.

\begin{Lemma}\label{lemma:epsilon}
There exists $\varepsilon>0$ which does not depend on $\rho, \o$ such that, $\bbP$--a.s., $ r^\rho_k(0)\geq 2\varepsilon$
for all $k< 0$ and for all $\rho\in \N_+ \cup \{\infty\}$.
\end{Lemma}
\begin{proof}
We just make a pathwise analysis. By the Markov property we get
\begin{align}
r^\rho_k(0)
	&=  \sum_{-\rho \leq j<0}\sum_{n=1}^\infty P_k^{\o,\rho }(X_n=0,\,X_{n-1=j},\, X_0,X_1,...,X_{n-2}<0)\nonumber \\
	&= \sum_{-\rho \leq j<0}\sum_{n=1}^\infty P_j^{\o,\rho }(X_1=0)
		P_k^{\o,\rho }(X_{n-1=j},\, X_0,X_1,...,X_{n-2}<0). \label{roma}
\end{align}
We claim  that there exists $\e>0$  such that, for all $j$ and $\o$,
$$
P_j^{\o,\rho}(X_1=0)\geq  2 \e\, P_j^{\o,\rho }(X_1\geq 0).
$$
In fact, given $j$ with $-\rho \leq j <0$, we can write 
\begin{align*}
\frac{P_j^{\o,\rho }(X_1=0)}{P_j^{\o,\rho }(X_1\geq 0)}
	&\geq  \frac{c_{j,0}}{\sum_{l=0}^\infty c_{j,l}} \geq  K\,\frac{{\rm e}^{\l{x_j}+x_j}}{\sum_{l=0}^\infty {\rm e}^{\l(x_l+x_j)-(x_l-x_j)}  } =K\,\frac{1 }{\sum_{l=0}^\infty {\rm e}^{-(1-\l)x_l}  } \\&\geq K\,\frac{1}{\sum_{l=0}^\infty {\rm e}^{-(1-\l)d l}}=:2 \varepsilon.
\end{align*}
Coming back to \eqref{roma}, using   the Markov property and the fact that $T_0<\infty$ a.s., we get
\begin{align*}
r^\rho_k(0)
	&\geq 2 \e  \sum_{-\rho \leq j<0}\sum_{n=1}^\infty P_j^{\o,\rho }(X_1\geq 0)
		P_k^{\o,\rho }(X_{n-1=j},\, X_0,X_1,...,X_{n-2}<0)\nonumber \\
		&=   2\e\,\sum_{-\rho \leq j<0}\sum_{n=1}^\infty P_k^{\o,\rho }(X_n\geq 0,\,X_{n-1=j},\, X_0,X_1,...,X_{n-2}<0)
		 \nonumber\\&
		= 2\e \,P_k^{\o,\rho }(X_{T_0}\geq 0)= 2\e. \qedhere
\end{align*}
\end{proof}


\section{Regenerative structure for the $\rho$--truncated random walk with $\rho< \infty$ }\label{rigenero}
 In this section we take $\rho < \infty$.
We recall the regenerative  structure of  \cite{CP} 
for the $\rho$-truncated random walk with $\rho$ finite.

\begin{Warning}\label{olanda}
In order to avoid heavy notation, in this section $\rho$ is fixed once and for all in $\N_+$ and we  write $\bbP_x^\o$, $T_k$,  $r_k(z)$, $X_n$,...    instead of $\bbP_x^{ \o, \rho}$, $T_k^\rho$, $r_k^\rho(z)$, $X_n^\rho$,... The whole section refers to the $\rho$-truncated random walk. Only in Subsection \ref{veloce}, in which we collect the main conclusions, we will indicate $\rho$ explicitly according to the usual notation.
 \end{Warning}

Consider a sequence of i.i.d.~Bernoulli r.v.'s~$\zeta_1,\zeta_2,...$ with parameter $P(\zeta_1=1)=\varepsilon$ (the same $\varepsilon$ as in Lemma \ref{lemma:epsilon}) which does not depend on the environment $\o$.  $P$ and $E$ denote the probability law and the  expectation of the $\zeta$'s. 
 We couple the sequence $\zeta=(\zeta_1,\zeta_2,...)$ with the random walk $X_n $ in such a way that $\zeta_j=1$ implies $X_{T_{j\rho}}=j\rho$.

To this aim
we construct  the quenched measure $P^{\o,\zeta}_0$  of the random walk  starting at $0$ once both $\o$ and $\zeta$ are fixed. Recall Lemma \ref{lemma:epsilon}.  First, the law of $(X_n)_{n\leq T_\rho}$ is defined by
\begin{equation}
\mathds{1}_{\{\zeta_1=1\}}P^{\o }_0(\cdot|X _{T_\rho }=\rho)   +
\mathds{1}_{\{\zeta_1=0\}}\Big[\frac{r_0(\rho)-\e}{1-\e} P^{\o }_0(\cdot|X_{T_\rho}=\rho) +
\frac{1-r_0(\rho)}{1-\e} P^{\o }_0(\cdot|X_{T_\rho}>\rho)\Big].
\end{equation}
Then, given $j \geq 1$ and  $X_{T_{j\rho}}=y\in[j\rho,(j+1)\rho)$, the law of $(X_n )_{n\in[T_{j\rho}+1,T_{(j+1)\rho}]}$ is
\begin{equation}\label{elisa}
\begin{split}
&\mathds{1}_{\{\zeta_{j+1}=1\}} P^{\o}_y(\cdot|X_{T_{(j+1)\rho}}=(j+1)\rho)   \\
	\quad+&\mathds{1}_{\{\zeta_{j+1}=0\}}\Big[\frac{r _y((j+1)\rho)-\e}{1-\e} P^{\o}_y(\cdot|		X_{T_{(j+1)\rho}}=(j+1)\rho) \\
	&\quad\qquad\qquad+	\frac{1-r_y((j+1)\rho)}{1-\e} P^{\o}_y(\cdot|X_{T_{(j+1)\rho}}>(j			+1)\rho)\Big].
\end{split}
\end{equation}
One can check that, by   averaging $P_0^{\o,\z}$ over $\z$,  one obtains  the law   $P^{\o }_0$  of  the   original random walk $(X_n )_{n \geq0}$.

\medskip 
We introduce by iteration the sequence $(\ell_k)_{k \geq 0}$ as follows:
$$
\ell_0  := 0\,, \qquad \ell_{k+1}=\min\{j> \ell_k:\zeta_j=1\} \qquad k \geq 0\,.
$$
Note that
 by construction we have
$ X _{T_{\ell_k \rho}  } = \ell_k \rho$.


Given $k \geq 0$ let $\cC_k:=\bigl(\t _{X_ j}\o\,:\, T _{\ell_k \rho} \leq j <  T _{\ell_{k+1} \rho}\bigr)$.
As in \cite{CP} one can prove the following  result (cf. \cite[Lemma 3.2]{CP} and the corresponding  proof):
\begin{Lemma}\label{lemma:cycle_stationary}
 Let $\rho<\infty$.
Then the sequence of random pieces  $(\cC_k)_{k \geq 0}$ is stationary and ergodic under the measure $P \otimes \bbP \otimes P^{\o, \z }_0$. In particular, $\t_{\ell_k \rho}\o$ has the same law $\bbP$ as $\o$ for all $k=1,2,...$.
\end{Lemma}

As in \cite{T}, the fact that $(\cC_k)_{k \geq 1}$ is stationary and ergodic can be restated  as follows:
under  $P \otimes \bbP \otimes P^{\o, \z }_0$ the random path $(X _n)_{n \geq 0}$ with time points $0<T_{\ell_1 \rho } < T_{\ell_2 \rho} < \ldots $ is \emph{ cycle--stationary and ergodic}. This is the regenerative structure pointed out in \cite{CP}.

\medskip

In what follows, we will consider also the random walk $(X_n)_{n \geq 0}$  starting at $x$ and with law $P^{\o, \z}_x$.  This random walk is built as follows. Fix $a$ such that $x \in [a \rho, (a+1)\rho)$. Then, the law of $(X_n)_{n\leq T_{(a+1)\rho} }$ is defined by \eqref{elisa} with $j$ replaced by $a$ and $y$ replaced by $x$. Note that $T_{a \rho}=0$.  Given $j \geq a+1$ and $X_{T_{j\rho}}=y\in[j\rho,(j+1)\rho)$, the law of $(X_n )_{n\in[T_{j\rho}+1,T_{(j+1)\rho}]}$ is then given by \eqref{elisa}. Again, 
the average over $\z$ of $P^{\o, \z}_x$ gives $P^{\o}_x$.

\subsection{Estimates on the regeneration times}\label{subsec:regeneration}

As in \cite{CP} we set
$$ \bbP':= P\otimes \bbP \,, \qquad \bbE'[\cdot] = E\left[ \bbE[\cdot ] \right]\,.
$$
In what follows we assume that $\bbE\left[ e^{(1-\l) x_1} \right]< \infty$.
\begin{Lemma}\label{gnocchi} Let $\rho < \infty$.
There exist constants $K_1,K_2>0$ not depending on  $\o, \rho $ such that
\begin{align}
& E[E^{\o,\zeta }_0[T _{\ell_1\rho}]] \geq K_1 \rho \,,\label{basso}\\
& \bbE'[ E^{\o,\zeta }_0[T _{\ell_1\rho}]]\leq K_2 \Big(\frac 1{   (1-{\rm e}^{-2\l d})^2     } + \frac{\rho}{1-{\rm e}^{-2\l d}}\Big)\bbE\bigl[ e^{(1-\l)x_1}\bigr]\,.\label{eq:regeneration_bound}
\end{align}
\end{Lemma}

\begin{proof}  Take a sequence  $Y_1, Y_2, \dots $ of i.i.d.~positive random variables with $P( Y_i \geq s)= ( K e^{- d s(1-\l)} ) \wedge 1$ for $s \geq 1$ integer, $K$ being the constant appearing in Lemma \ref{pierpaolo}. Due to this lemma,  under $P^{\o}_0$,  $X_k$ is stochastically dominated by $Y_1+\dots+Y_k$ for any $k \geq0$. Hence, one can prove \eqref{basso}   as in the derivation of the first inequality of (19) in \cite{CP}.

\smallskip

 We concentrate on \eqref{eq:regeneration_bound}.
%
%
Exactly like on page 731, formulas (21) and (22) of \cite{CP}, we also have that for any $\zeta$ and for all $j\geq 0$
\begin{equation}\label{eq:get_rid_zeta}
	E^{\o,\zeta}_y[T_{(j+1)\rho}]\leq\frac{1}{\e(1-\e)} E^{\o}_y[T_{(j+1)\rho}]
\end{equation}
for all $y\in[j\rho,(j+1)\rho-1)$.

When $\ell_1=k$ we can write
$$
T_{\ell_1\rho}=T_\rho+(T_{2\rho}-T_{\rho})+...+(T_{k\rho}-T_{(k-1)\rho}).
$$
Now for each $j\geq 1$ we have
\begin{align}
 E^{\o,\zeta}_0[T_{(j+1)\rho} -T_{j\rho} ]
	& =\sum_{y\in[j\rho,(j+1)\rho)} E^{\o,\zeta}_y[T_{(j+1)\rho} ]\,
		P^{\o,\zeta}_0(X_{T_{j\rho} }=y)\nonumber\\
	& \leq \frac{1}{\varepsilon(1-\varepsilon)}\sum_{y\in[j\rho,(j+1)\rho)}E^{\o}_y[T_{(j+1)\rho}]\,
	P^{\o,\zeta}_0(X_{T_{j\rho}}=y)\nonumber
\end{align}
where we have used \eqref{eq:get_rid_zeta}. Now we see that, for any $y\in[j\rho,(j+1)\rho)$,
\begin{align}
E^{\o}_y[T_{(j+1)\rho}]
	&\leq E^{\o}_y[N_\infty((-\infty,(j+1)\rho])] \nonumber\\
	&\leq K\,E^{\o,1}_y[N_\infty((-\infty,(j+1)\rho])] \nonumber\\
	&\leq K\,E^{\o,1}_{j\rho}[N_\infty((-\infty,(j+1)\rho])] \nonumber,
\end{align}
where the second inequality is due to Corollary \ref{cor:comparison_number_visits}.

Hence
\begin{align}
E^{\o,\zeta}_0[T_{(j+1)\rho}-T_{j\rho}]
	&\leq K\, \frac{1}{\varepsilon(1-\varepsilon)} E^{\o,1}_{j\rho}[N_\infty((-\infty,(j+1)\rho])]\nonumber\\
	&\leq K\,\frac{1}{\varepsilon(1-\varepsilon)}\Big(\sum_{z\leq j\rho} E^{\o,1}_{j\rho}[N_\infty(z)]
		+ \sum_{j\rho<z\leq (j+1)\rho} E^{\o,1}_{j\rho}[N_\infty(z)]\Big)\nonumber\\
	&=K\, \frac{1}{\varepsilon(1-\varepsilon)}\Big(\sum_{z\leq j\rho} E^{\o,1}_{j\rho}[N_\infty(z)]
		+ \sum_{j\rho<z\leq (j+1)\rho} E^{\o,1}_{z}[N_\infty(z)]\Big)\nonumber
\end{align}

Using the results of Lemma \ref{piccolino} and Corollary \ref{cor:comparison_number_visits} we obtain for every $j\geq 1$
\begin{align}
\bbE E^{\o,\zeta}_0[T_{(j+1)\rho}-T_{j\rho}]
	&\leq K\,\frac{1}{\varepsilon(1-\varepsilon)}\Big(\sum_{k\geq 0}\bbE[g_{\tau_{j\rho}\o}(k)]
		+ \sum_{j\rho<z\leq (j+1)\rho}\bbE[g_{\tau_{z}\o}(0)]\Big)\nonumber\\
	&\leq  K\,\frac{1}{\varepsilon(1-\varepsilon)} \Big(\frac 1{   (1-{\rm e}^{-2\l d})^2     }
		+ \frac{\rho}{1-{\rm e}^{-2\l d}}\Big)\bbE\bigl[ e^{(1-\l)x_1}\bigr]\nonumber
\end{align}
and hence
$$
\bbE E^{\o,\zeta}_0[T_{k\rho}]
	\leq  K\,\frac{k}{\varepsilon(1-\varepsilon)}
	\Big(\frac 1{   (1-{\rm e}^{-2\l d})^2     }
		+ \frac{\rho}{1-{\rm e}^{-2\l d}}\Big)\bbE\bigl[ e^{(1-\l)x_1}\bigr]\,.
$$
Since $P (\ell_1=k)=\varepsilon(1-\varepsilon)^{k-1}$, we obtain
\begin{align}
\bbE'\big [E^{\o,\zeta}_0[T_{\ell_1\rho}]\big]
	&\leq K\,	\Big(\frac 1{   (1-{\rm e}^{-2\l d})^2 } + \frac{\rho}{1-{\rm e}^{-2\l d}}\Big)\bbE\bigl[ e^{(1-\l)x_1}\bigr]
		\sum_{k=1}^\infty  k\,(1-\varepsilon)^{k-2} \nonumber\\
	&=	\bar K\,	\Big(\frac 1{   (1-{\rm e}^{-2\l d})^2 } + \frac{\rho}{1-{\rm e}^{-2\l d}}\Big) \bbE\bigl[ e^{(1-\l)x_1}\bigr]\,.
\end{align}
\end{proof}

Recall the definition of the function $g_\o$ given in Prop.~\ref{prop:function_g}.

\begin{Lemma}\label{carnevale1} Let $\rho< \infty$.
Given $k \leq 0$ it holds
\begin{equation}\label{frappe1}
E^{\o, \z}_0 \left[N_\infty (k)\right ] \leq \frac{1}{\e(1-\e)} \sum_{j=0}^\infty g_{\t_{j \rho} \o} ( |k|+j\rho)\,.
\end{equation}
\end{Lemma}
\begin{proof}
As for the derivation of (33) in  \cite{CP} one can prove that, if $y \in [j \rho, (j+1) \rho)$,  then
\begin{equation}\label{passo1}  E^{\o, \z}_y \left[ N_{ [T_{j \rho}, T_{(j+1) \rho} )}(k) \right]\leq \frac{1}{\e(1-\e)} E^\o_y \left[ N_{ [T_{j \rho}, T_{(j+1) \rho} )}(k) \right].
\end{equation}
On the other hand, by applying Prop.~\ref{prop:function_g} , we get
\begin{equation}\label{passo2} E^\o_y \left[ N_{ [T_{j \rho}, T_{(j+1) \rho} )}(k) \right]\leq E^\o_y \left[ N_\infty (k) \right]= E^{\t_y \o}_0 \left[ N_\infty (k-y) \right] \leq g_{\t_y\o }(|k|+y)\,.  
\end{equation}
At this point we write $y$ as $y= j \rho +\ell$ and set $\o':= \t_{j \rho }\o$. Then, by applying  \eqref{spring2} in  Lemma \ref{piccolino}, we get
\begin{equation}\label{passo3} g_{\t_y\o }(|k|+y) = g_{\t _\ell \o'}( |k|+ j \rho +\ell) \leq g_{\o'} ( |k|+ j \rho) = g_{\tau_{j \rho} \o} ( |k|+ j \rho)  \,.
\end{equation}
As a byproduct of \eqref{passo1}, \eqref{passo2} and \eqref{passo3} we conclude that
\begin{equation}
E^{\o, \z}_y \left[ N_{ [T_{j \rho}, T_{(j+1) \rho} )}(k) \right]\leq \frac{1}{\e(1-\e)}  g_{\tau_{j \rho} \o} ( |k|+ j \rho)\,.
\end{equation}
The above bound and the strong Markov property  applied at time $T_{j \rho}$ (which holds by construction of $P_0 ^{\o, \z}$)  imply that
\begin{equation}\label{frappe101}
E^{\o, \z}_0  \left[ N_{ [T_{j \rho}, T_{(j+1) \rho} )}(k) \right] = E^{\o, \z}_0  \left[
 E^{\o, \z}_{ X_{T_{j\rho} } }   \left[  N_{ [T_{j \rho}, T_{(j+1) \rho} )}(k) \right]
  \right] \leq   \frac{1}{\e(1-\e)}  g_{\tau_{j \rho} \o} ( |k|+ j \rho)\,.
\end{equation}
Since $ N_\infty (k) = \sum _{j=0}^\infty N_{ [T_{j \rho}, T_{(j+1) \rho} )}(k) $, the above bound
\eqref{frappe101} implies \eqref{frappe1}.
\end{proof}

\subsection{Speed for the truncated  process}\label{veloce}  
Recall that $\rho< \infty$ is fixed and recall  Warning \ref{olanda}. Here we follow the usual notation, indicating explicitly $\rho$, and we also write $P_0^{\o, \z, \rho}$ instead of $P_0^{\o,\z}$  to stress the dependence on $\rho$.

\begin{Proposition}\label{prop_vel} Fix $\rho <+\infty$.
For $\bbP$--a.a.~$\o \in \O$ it holds 
\begin{equation}\label{eqn:truncated_speed}
\vrho  (\l):= \lim_{n \to \infty} \frac{ X_n^\rho }{n}=\frac{\rho E[\ell_1]}{\bbE'[ E^{\o,\zeta}_0[T_{\ell_{1}\rho}^\rho  ]]}=\frac{\rho }{\e \bbE' [E^{\o,\zeta}_0[T_{\ell_{1}\rho}^\rho]]} \qquad
P^{\o, \rho}_0\text{--a.s.}
\end{equation}
where $\varepsilon$ is the same as in Lemma \ref{lemma:epsilon}.
Moreover, $\vrho   (\l) $ does not depend on $\o$ and 
\begin{equation}\label{vunifbounds}
\vrho(\l)  \in (c_1,c_2)
\end{equation}
for strictly positive constants $c_1,c_2$, which do neither depend on $\omega$ nor on $\rho$.
\end{Proposition}
\begin{proof} We work on  the probability space $( \Theta, P \otimes \bbP \otimes P_0^{\o,\z,\rho})$ where $\Theta:= \{0,1\}^{\bbN_+} \times \O  \times \bbZ^{\bbN}$.
 For $n\in [T_{\ell_k\rho}^\rho , T_{\ell_{k+1}\rho}^\rho)$ we have $\ell_{k+1}\rho-(T_{\ell_{k+1}\rho}^\rho-T_{\ell_{k}\rho}^\rho )\rho<X_n^\rho<\ell_{k+1}\rho$ (note in particular that $X_n^\rho$ has to be thought as a  function on $\Theta$). It then follows
\begin{equation}\label{ki-energy}
\frac{\ell_{k+1}\rho-(T^\rho_{\ell_{k+1}}-T^\rho_{\ell_{k}\rho})\rho}{T^\rho _{\ell_{k+1}\rho}}<\frac{X_n^\rho}{n}<\frac{\ell_{k+1}\rho}{T^\rho _{\ell_k\rho}}\,.
\end{equation}

Due to the cycle stationarity and ergodicity stated in Lemma \ref{lemma:cycle_stationary},   we  let $n \to \infty$ in \eqref{ki-energy} and obtain  that the limit  in  \eqref{eqn:truncated_speed} holds $P \otimes \bbP \otimes P_0^{\o,\z,\rho}$--a.s.
 On the other hand  $|X_n^\rho /n| \leq \rho$. Hence, by the dominated convergence theorem, 
  $E[ X_n ^\rho /n ] $ converges to ${\rho E[\ell_1]}/{\bbE'[ E^{\o,\zeta}_0[T_{\ell_{1}\rho}^\rho  ]]}$. 
To conclude the proof of \eqref{eqn:truncated_speed}, it is enough to recall  that,  by   averaging $P_0^{\o,\z,\rho}$ over $\z$,  one obtains  the law   $P^{\o ,\rho }_0$  of  the   original random walk $(X_n^\rho )_{n \geq0}$.

Finally, we observe that $\vrho$ does not depend on  $\o$ since the last term in \eqref{eqn:truncated_speed} doesn't, and that 
$\vrho   (\l) \in (c_1,c_2)$ due to 
 \eqref{basso} and  \eqref{eq:regeneration_bound}.
\end{proof}
\section{Stationary distribution $\bbQ^\rho$ of the environment viewed from the $\rho$--walker }\label{tronco}

   In this section we assume that $\bbE \bigl[ e^{(1-\l) x_1}\bigr]< \infty$ and we fix $\rho <\infty$.
 We consider the process \emph{environment viewed from the $\rho$--walker}, which is   the Markov chain $(\tau_{X_n^\rho}\o)_{n \in\N }$ on the space of environments $\Omega$ with transition mechanism induced by $P^{\o,\rho}_0$. When starting with initial distribution $Q$, we denote by $ \cP ^\rho _{Q}$ its law as probability distribution on $\O^\bbN$.
 
  Lemma \ref{lemma:cycle_stationary} and  bound \eqref{eq:regeneration_bound} in Lemma \ref{gnocchi} guarantee  (cf.~\cite[Sec.~4, Chapter 8]{T}) the existence of a stationary distribution  $\bbQ^\rho$ of the process \emph{environment viewed from the $\rho$--walker}, such that $\bbQ^\rho$ is absolutely continuous with respect to $\bbP$. 
    
From \cite{T}[Eq.~(4.14), Chapter 8],   $\bbQ^\rho$ can be characterized by its expectation:
\begin{equation}
\bbE^\rho[f(\o)]=\frac{1}{\bbE'[ E_0^{\o,\zeta, \rho }[T_{\ell_1\rho}]]}\bbE' \Big[ E^{\o,\zeta, \rho }_0\Big[\sum_{k=1}^{T _{\ell_1\rho}}f(\tau_{X_k }\o)\Big]\Big].
\end{equation}

%


As in \cite{CP}[Prop.~3.4] one can prove that $\bbQ^\rho$ is absolutely continuous to $\bbP$ with Radon--Nikodym derivative
\begin{equation}\label{umbria}
\frac{ d\bbQ^\rho}{d \bbP} (\o) =\frac{1}{\bbE'[ E_0^{\o,\zeta, \rho }[T_{\ell_1\rho}]]}\sum _{k \in \bbZ} E  E^{\t_{-k} \o, \z, \rho}_0 \left[ N_{T_{\ell_1 \rho}} (k)\right]\,.
\end{equation}
Note that the denominator in the r.h.s. is finite due to \eqref{eq:regeneration_bound} and the numerator is   positive. As a consequence, $\bbP$ is also absolutely continuous to $\bbQ^\rho$.

\begin{Lemma}\label{lemma:ergodicity_q_rho}
Fix $\rho\in \bbN_+$. Then $\bbQ^\rho$ is ergodic with respect to shifts for the environment seen from the $\rho$-walker.
\end{Lemma}

   \begin{Remark}\label{caffettino}The above ergodicity means that any Borel subset of the path space   $ \O^{\bbN}$, which is left invariant by shifts, has $\cP^\rho _{\bbQ^\rho}$--probability equal to $0$ or $1$.
   
   Due to Theorem 6.9 in \cite{V} (cf.\@ also \cite[Chapter  IV]{Ros}), the above ergodicity 
 is equivalent to the following fact:
$\bbQ^\rho (A) \in \{0,1\}$ whenever $A \subset \O$ 
 is an invariant Borel set, in the sense that 
 ``$\t_{X_n ^\rho} \o \in  A$  for any $n \in \bbN$''  holds $\bbQ^\rho \otimes P_0^{\o, \rho}$--a.s. on $\{\o \in A\}$ and 
 ``$\t_{X_n ^\rho} \o \in  A^c$ for any $n \in \bbN$'' holds $\bbQ^\rho \otimes P_0^{\o, \rho}$--a.s. on $\{\o \in A^c\}$. As usual, $\bbQ^\rho \otimes P_0^{\o, \rho}$ is the probability measure on $ \O \times \bbZ^{\bbN}$ such that the expectation of a function $f$ is given by $ \int \bbQ^\rho (d\o) E_0^{\o, \rho}  \bigl[ f(\o, (X_n)_{n \geq 0})\bigr]$.\end{Remark}

\begin{proof}[Proof of Lemma \ref{lemma:ergodicity_q_rho}]
 The proof   can be obtained as in \cite[page 735--736]{CP}. The only difference is that in \cite{CP}  the authors use their formula (29),  which is not satisfied in our case. More precisely , they use their formula (29) to argue that 
 $0 < \bbP(A) < 1$ for any $\bbQ^\rho$--nontrivial set $A$. On the other hand, this claim follows simply  from the absolute  continuity of  $\bbQ^\rho $ to $ \bbP$.
 \end{proof}

  \smallskip
  
  The rest of this section is devoted to the proof of the following result:
  \begin{Proposition}\label{digiuno} Suppose $\bbE \bigl[ e^{(1-\l) x_1}\bigr]< \infty$ and that $u: \bbR \times \bbR \to \bbR$ is continuous. 
 Then the sequence $( \bbQ^\rho)_{\rho \in \N_+} $ converges weakly to a unique measure $\bbQ^\infty $ as $\rho\to\infty$.  $\bbQ^\infty $ is absolutely continuous  to $\bbP$ and, $\bbP$--a.s., $0<\gamma\leq \frac{d \bbQ^\infty}{ d \bbP} \leq F $ (cf.~\eqref{def_F}).
 Furthermore, 
  $\bbQ^\infty$ is invariant and ergodic
   for the dynamics from the point of view of the $\infty$--walker.\footnote{
   Ergodicity means that the law $\cP^\infty_{\bbQ ^\infty}$ on the path space $\O^{\bbN}$  is ergodic with respect to shifts  (cf.~Remark \ref{caffettino}).}
  \end{Proposition}
Having Lemma \ref{pierpaolo} and  Lemma  \ref{chiave} below, Proposition \ref{digiuno} can be proved by the same arguments used in \cite[p.~735]{CP}, with some slight modifications. For completeness, we give the proof in Appendix \ref{ape}.


\subsection{Upper bound for the Radon-Nikodym derivative  $d\bbQ^\rho/d\P$}\label{up_bound}
\begin{Proposition}\label{aereo}
Suppose $\bbE[ e^{(1-\l) x_1}] < \infty$. Then, uniformly in $\rho \in \N_+$,
\begin{equation}\label{eqn:bound_radonnikodym}
\frac{ d\bbQ^\rho}{d \bbP} (\o)  \leq F(\o) \qquad \bbP-a.s.,
\end{equation}
where
\begin{equation}\label{def_F}
F (\o):=K \pi^1(0)  \sum_{j=0}^\infty  (j+2)^2{\rm e}^{-2\l x_j +(1-\l) (x_{j+1}-x_{j}) } \,,
\end{equation}
for some constant $K>0$.
Moreover,
$\bbE[F]<\infty$.
\end{Proposition}

Before proving Prop.~\ref{aereo} we state a technical result:
\begin{Lemma}\label{tennis}
 Let $F_* (\o):=K_0  \sum_{i=0}^\infty(i+1) {\rm e} ^{-2\l x_{i} +(1-\l) ( x_{i+1} -x_{i})} $, with $K_0$ as in Proposition  \ref{prop:function_g}. Then
\begin{equation}\label{nadal}
\sum_{j=0}^\infty g_{\t_{j \rho} \o} ( |k|+j\rho) \leq \sum_{r=0}^\infty g_{\t_{r} \o} ( |k|+r) \leq \pi^1(-|k|) F_*(\o)\,.
\end{equation}
\end{Lemma}
\begin{proof}
The first inequality in \eqref{nadal} is trivial. We prove the second one.
By \eqref{spring0} and \eqref{aurora100} we can write 
\begin{equation}
\begin{split}
\sum_{r=0}^\infty g_{\t_{r} \o} ( |k|+r)& \leq K_0 \pi^1(-|k|)
 \sum_{r=0}^\infty  \sum_{j=0}^\infty {\rm e} ^{-2\l x_{j+r} +(1-\l) ( x_{j+1+r} -x_{j+r})}\\
 & \leq  K_0 \pi^1(-|k|) \sum_{i=0}^\infty(i+1) {\rm e} ^{-2\l x_{i} +(1-\l) ( x_{i+1} -x_{i})}\qedhere
  \end{split}
  \end{equation}
\end{proof}

We can now prove Prop. \ref{aereo}:

\begin{proof}[Proof of Prop.~\ref{aereo}]
Due to \eqref{basso} and \eqref{umbria} we can bound, for some constant $C>0$,
\begin{equation}\label{8978}
\frac{ d\bbQ^\rho}{d \bbP} (\o) \leq \frac{C}{\rho}\left[  H_+(\o)+ H_-(\o)\right]
\end{equation} where
\begin{align*}
 H_+ (\o):=  \sum _{k > 0 } E  E^{\t_{-k} \o, \z, \rho}_0 \left[ N_{T_{\ell_1 \rho}} (k)\right]\,,
\qquad
 H_-(\o):=  \sum _{k \leq 0 } E  E^{\t_{-k} \o, \z, \rho}_0 \left[ N_{T_{\ell_1 \rho}} (k)\right]\,.
\end{align*}
As a byproduct of Lemma \ref{carnevale1} and Lemma \ref{tennis} it holds (see the proof of \eqref{spring2} for the equality below)
\begin{equation}\label{federer}
H_-(\o) \leq \frac{1}{\e(1-\e)} \sum_{ k \leq 0} 
\pi^1(k)[ \t_{-k} \o] 
 F_*(\t_{-k} \o)=\frac{1}{\e(1-\e)} \pi^1(0) \sum_{ k \leq 0} {\rm e}^{- 2 \l x_{-k}}
 F_*(\t_{-k} \o)
\,.
\end{equation}

Let us bound $H_+(\o)$. We can write
\begin{equation}\label{ciocco}
\begin{split}
\sum _{k  =  0 }^\infty  E  E^{\t_{-k} \o, \z, \rho}_0 &\left[ N_{T_{\ell_1 \rho}} (k)\right]
 		= \sum _{ m =0 }^\infty \sum_{ k \in [m \rho, (m+1) \rho)} \sum_{i=1}^\infty E\,
			\left[ \mathds{1}_{\ell_1=i} E^{\t_{-k} \o, \z,\rho}_0 \left[ N_{ T_{i \rho} } (k) \right]			\right]\\
	&= \sum _{ m =0 }^\infty \sum_{ k \in [m \rho, (m+1) \rho)} \sum_{i=1}^\infty \sum _{j=0}^i
		E\,\left[ \mathds{1}_{\ell_1=i}
		E^{\t_{-k} \o, \z,\rho}_0 \left[ N_{ [T_{j\rho}, T_{(j+1) \rho})} (k) \right] \right]\,.
\end{split}
\end{equation}
 Note that, given $m>j \geq 0$ and $ k \in [m \rho, (m+1) \rho)$, it holds  $N_{ [T_{j\rho}, T_{(j+1) \rho})} (k)=0$, hence
in the last expression of  \eqref{ciocco} we can restrict to $0 \leq m  \leq j \leq i $.
Moreover note that (cf. \eqref{passo1})
\begin{equation}\label{finlandia}
E^{\t_{-k} \o, \z,\rho}_0 \left[ N_{ [T_{j\rho}, T_{(j+1) \rho})} (k) \right] \leq \frac{1}{\e(1-\e)}
E^{\t_{-k} \o, \rho}_0 \left[ N_{ [T_{j\rho}, T_{(j+1) \rho})} (k) \right]\,.
\end{equation}

Consider then the case $k \in [m \rho, (m+1) \rho)$ with $0 \leq m  \leq j \leq i $. Note 
that  $X^{\rho} _{T_{j \rho } } \in [j \rho, (j+1) \rho)$ due to the maximal length of the jump. Fix $ y \in [j \rho, (j+1) \rho)$. Then, for any environment $\o$, we have
\begin{equation}\label{leone}
E^{  \o, \rho}_y \left[ N_{ [T_{j\rho}, T_{(j+1) \rho})} (k) \right]=  E^{  \o, \rho}_y\left[ N_{  T_{(j+1) \rho} } (k) \right] \leq
\begin{cases}
g_{ \t_k \o} (0) &\text{ if } j=m\,,\\
g_{\t_{j\rho}\o } (j\rho-k )  &\text{ if } j >m\,.
\end{cases}
\end{equation}
Indeed, consider first the case $j>m$. Then $k <y$ and by Prop. \ref{prop:function_g}
\begin{equation}\label{stellina}
E^{  \o, \rho}_y\left[ N_{  T_{(j+1) \rho} } (k) \right]  \leq E^{  \o, \rho}_y\left[ N_{  T_\infty } (k) \right]= E^{  \t_y\o, \rho}_0\left[ N_{  T_\infty } (k-y) \right]  \leq g_{\t_y \o} ( y-k)
\end{equation}
Write
$y= j \rho+ \ell$ and $\o':= \t_{j \rho} \o$. Then we have \[
g_{\t_y \o} ( y-k) =g_{ \t_\ell  \o'}(j \rho -k+ \ell)\leq g _{\o'} (j\rho-k)=g_{\t_{j\rho}\o } (j\rho-k )  \]
 (in the last step we have used \eqref{spring2}). This proves \eqref{leone} for $j>m$.
If $j=m$   we bound (by the Markov property at the first visit of  $k$)
\[E^{  \o, \rho}_y\left[ N_{  T_{(j+1) \rho} } (k) \right]  \leq E^{  \o, \rho}_y\left[ N_{  T_\infty } (k) \right]\leq   E^{  \o, \rho}_k \left[ N_{  T_\infty } (k) \right] = E^{  \t_k\o, \rho}_0 \left[ N_{  T_\infty } (0) \right]\,.
\]
At this point  \eqref{leone} for $j=m$  follows from Prop. \ref{prop:function_g}.

The above bound \eqref{leone},
 the Markov property  and \eqref{finlandia} imply
\begin{equation}
E^{\t_{-k} \o, \z,\rho}_0 \left[ N_{ [T_{j\rho}, T_{(j+1) \rho})} (k) \right] \leq \frac{1}{\e(1-\e)} \cdot
\begin{cases}
g_{  \o} (0) &\text{ if } j=m\,,\\
g_{\t_{j\rho-k}\o } (j\rho-k )  &\text{ if } j >m\,.
\end{cases}
\end{equation}

Coming back to \eqref{ciocco} and due to the above observations
 we can bound
 \begin{equation}\label{pasqua}
H_+(\o) \leq  \sum _{k  =  0 }^\infty  E  E^{\t_{-k} \o, \z, \rho}_0 \left[ N_{T_{\ell_1 \rho}} (k)\right] \leq A(\o)+B(\o)\,,
 \end{equation}
 where (distinguishing the cases $m=j$ and $m<j$)
 \begin{align}
 & A(\o):=
  \sum_{i=1}^\infty \sum _{j=0}^i  \sum _{k \in  [j \rho, (j+1) \rho) }\label{frutta}
E\,
\left[ \mathds{1}_{\ell_1=i}\right] g_\o(0) = \rho \left ( E(\ell_1)+1\right) g_\o(0)\,,
 \\
 & B(\o):=\sum_{i=1}^\infty \sum _{j=1 }^i  \sum _{ m =0 }^{j-1}  \sum_{ k \in [m \rho, (m+1) \rho)}  E\,
\left[ \mathds{1}_{\ell_1=i}\right]  g_{\t_{j\rho-k}\o } (j\rho-k )
 \,.\nonumber
 \end{align}

 For what concerns $B(\o)$ observe that $ \sum _{i=j }^\infty E\,
\left[ \mathds{1}_{\ell_1=i}\right]  = (1-\e)^{j-1}$, hence\begin{equation}\label{orientarsi}
\begin{split}
B(\o) & \leq \sum_{j=1}^\infty (1-\e)^{j-1}  \sum _{ m =0 }^{j-1}  \sum_{ k \in [m \rho, (m+1) \rho)}   g_{\t_{j\rho-k}\o } (j\rho-k ) \\
 & = \sum_{j=1}^\infty (1-\e)^{j-1}    \sum_{ k \in [0, j \rho)}   g_{\t_{j\rho-k}\o } (j\rho-k ) =
 \sum_{j=1}^\infty (1-\e)^{j-1}    \sum_{h=1}^{j \rho}    g_{\t_h\o } (h ) \,.\end{split}
\end{equation}
Since, by  \eqref{spring2}, $g_{\t_h\o } (h )\leq g_\o(0) $, we get  that $B(\o) \leq  \sum_{j=1}^\infty (1-\e)^{j-1} j\rho     g_\o(0)= C  \rho g_\o (0)$. 
 Combining this estimate  with \eqref{frutta} we conclude that $ H_+(\o) \leq   C \rho   g_\o(0)$. 
Coming back to \eqref{federer} and \eqref{8978} we have
\begin{equation}\label{cetriolo1}
\frac{ d\bbQ^\rho}{d \bbP} (\o) \leq 
 C    \pi^1(0) \sum_{ k \leq 0} {\rm e}^{- 2 \l x_{-k} }
 F_*(\t_{-k} \o)+ C    g_\o(0) \leq    \hat C  \pi^1(0) \sum_{ k \leq 0}{\rm e}^{- 2 \l x_{-k}  }
 F_*(\t_{-k} \o)  \,.
\end{equation}
Since $x_a( \t_{-k} \o)= x_{a-k} (\o) - x_{-k} (\o) $, by definition of $F_*$ (and setting $r=i-k$)
we can write 
\begin{equation}\label{cetriolo2}
\begin{split}
\sum_{ k \leq 0}{\rm e}^{- 2 \l x_{-k}  }
 F_*(\t_{-k} \o)& =K_0 \sum _{k \leq 0} \sum _{ i \geq 0} (i+1) {\rm e}^{- 2\l x_{i-k} +(1-\l) ( x_{i-k+1}-x_{i-k})}\\
 & = 
 K_0  \sum _{r \geq 0} {\rm e} ^{- 2 \l x_r + (1-\l) (x_{r+1}-x_r) }\frac{(r+1)(r+2)}{2}\,.
 \end{split}
 \end{equation}
 As byproduct of \eqref{cetriolo1} and \eqref{cetriolo2} we get \eqref{eqn:bound_radonnikodym}. 
 
 Finally, by using \eqref{spring0} and that $x_j \geq j d$ for $j \geq 0$, we can bound $\bbE[F] \leq C \sum _{ j\geq 0} (j+2)^2 {\rm e}^{-2 \l j} \bbE[ {\rm e} ^{(1-\l) x_1 }] $.
  Since by assumption $\bbE[ {\rm e} ^{(1-\l) x_1 }] <\infty$, we conclude that $\bbE[F] < \infty$.
\end{proof}
%

\subsection{Uniform lower bound for   $d\bbQ^\rho/d\P$}\label{lo_bound}

We remark that, following the proof of Proposition 3.4 in \cite{CP}, we could easily obtain a lower bound on $d\bbQ^\rho/d\P$ which is independent of $\rho$, but which would in principle depend on the particular argument $\o$. Here we will do more: We will exhibit a lower bound that is uniform in both $\rho$ and $\o$ (see Corollary \ref{cor:rn_lower_bound}  below).

\medskip

For fixed $\o \in \O$, we denote by $Q_n^\o$  the empirical measure at time $n$ for the environment viewed from the $\rho$--walker. More precisely,  $Q_n^\o$  is a random probability measure on $\O$ defined as    
$$
Q^\o _n:=\frac 1n \sum_{j=1}^n \d _{\tau_{X^\rho _j}\omega }\,.
$$
Averaging over the paths of the walk we obtain the  probability $E^{\o,\rho}_0 [Q^\o_n(\cdot)]$.
For fixed $\o \in \O$, we define another probability measure on $\Omega$, given by
$$
R^\o _n:=\frac 1{m(n)}\sum_{j=1}^{m(n)} \d_{\tau_j\omega}\,,
$$
where $m(n):=n\cdot   \vrho /2  $ and $\vrho $ is the positive limiting speed of the truncated random walk given in \eqref{eqn:truncated_speed} (we are omitting the dependence on $\l$; the $1/2$  could be replaced by any other constant smaller than 1). 

\smallskip
We remark that $R_n^\o$ and $E^{\o,\rho}_0 [Q^\o_n(\cdot)]$ can be thought of as random variables on $(\O, \bbP)$ with values in $\cP (\O)$, the space of probability measures on $\O$ endowed with the weak topology. Note also that $\bbP, \bbQ^\rho \in \cP(\O)$.  Furthermore,  $Q_n^\o$ can be thought of as a random variable on the probability space $(\O \times\bbZ ^{\N},\,\P\otimes P_0^{\o,\rho}  ) $ with values in $\cP (\O)$.

\begin{Proposition}\label{prop:convergences} For $\bbP$--almost every  $\o \in \O$
we have that   $R^\o_n\rightarrow \P$ and $E^{\o,\rho}_0[Q^\o_n(\cdot)]\rightarrow \bbQ^\rho$  weakly 
in $\cP(\O)$.  Moreover, $\P\otimes P_0^{\o,\rho}$-a.s., we have that 
 $Q_n^\o \rightarrow \bbQ^\rho$ weakly 
in $\cP(\O)$. 
\end{Proposition}

\begin{proof}
The a.s.~convergence of $R_n^\o $ to $\bbP$  comes directly from the ergodicity of $\P$ with respect to shifts.

We claim  that 
 $Q_n^\o \rightarrow \bbQ^\rho$ weakly 
in $\cP(\O)$,   $\bbQ^\rho \otimes P_0^{\o,\rho}$-a.s. This follows from Birkhoff's ergodic theorem applied to the Markov chain $\t_{ X_n ^\rho} \o$ starting from the ergodic distribution $\bbQ^\rho$ (cf. Lemma \ref{lemma:ergodicity_q_rho}).
As already observed after equation \eqref{umbria}, $\bbP$ is absolutely continuous to $\bbQ^\rho$. Hence, due to the above claim, $Q_n^\o \rightarrow \bbQ^\rho$ weakly 
in $\cP(\O)$ also  $\bbP  \otimes P_0^{\o,\rho}$-a.s. 

Finally, the last a.s.~convergence  and the dominated convergence theorem imply that 
 $E^{\o,\rho} _0[Q^\o_n(\cdot)]\rightarrow \bbQ^\rho$  weakly 
in $\cP(\O)$, $\bbP$--a.s.
\end{proof}

\begin{Lemma}\label{lemma:finite_n_derivative} There exists  $\g>0$, depending neither on $\omega$ nor on $\rho$, such that  the following holds:
For $\P$-almost every $\omega$, there exists an $\bar n_\omega$ such that, $\forall n\geq \bar n_\omega$,
$$
\frac{E^{\o,\rho}_0 [Q^\o_n(\tau_k\o)]}{ R^\o_n(\tau_k\o)}>\gamma\,, \qquad k=1,...,m(n)\,.
$$
\end{Lemma}

\begin{proof}
For all $k=1,...,m(n)$, we have
\begin{equation}\label{zanzara1}
E^{\o,\rho}_0 [Q^\o_n(\cdot)] \geq \frac 1n   P^{\o,\rho}_0 (\exists j\leq n:\,X_j=k) \d_{\tau_k\omega}\, .
\end{equation}
We claim that, for $n$ big enough and $k=1,...,m(n)$, it holds 
\begin{equation}\label{zanzara2}
P^{\o,\rho}_0 (\exists j\leq n:\,X_j=k) \geq \varepsilon,
\end{equation} 
where $\varepsilon>0$ is the same as in Lemma \ref{lemma:epsilon}.  To prove our claim, we bound
\begin{align}
P^{\o,\rho}_0 (\exists j\leq n:\,X_j=k)
	&\geq P^{\o,\rho}_0 ( X_{T_k}=k,\,T_k\leq n)\nonumber\\
	&\geq P^{\o,\rho}_0 ( X_{T_k}=k)-P^{\o,\rho}_0 ( T_k> n)\nonumber\\
	&\geq 2\varepsilon-P^{\o,\rho}_0 ( T_{m(n)}> n),\nonumber
\end{align}
where in the last line we have used Lemma \ref{lemma:epsilon}. On the other hand,  we also know, by the definition of the limiting speed, that
for almost every $\omega\in\Omega$, there exists an $\bar n_\omega$ such that, $\forall n>\bar n_\omega$, $P^{\o,\rho}_0 ( T_{m(n)}> n)  \leq  P^{\omega,\rho}_0(X_n<m(n))<\varepsilon$. This completes  the proof of the claim.

Hence,   putting together \eqref{zanzara1} and \eqref{zanzara2}, for all $n\geq \bar n_\omega$ and $k=1,...,m(n)$, we have 
$$
E^{\o,\rho}_0 [Q_n^\o (\cdot)] \geq \frac \varepsilon n \d_{\tau_k\omega}\,  .
$$
On the other hand, by definition, 
 $R^\o _n(\tau_k \omega)=\frac 1{m(n)}$  for all $k=1,...,m(n)$  and for  $\bbP$--a.a. $\o$  (since  periodic environments have $\bbP$--measure zero  by Assumption (A3)). 
It then  follows that, for all $k=1,...,m(n)$ and for $\bbP$--a.a. $\o$,  
\begin{equation}\label{radbound}
\frac{E^{\o,\rho}_0 [Q^\o_n(\tau_k\o)]}{ R^\o_n(\tau_k\o)}\geq \frac{\tfrac \varepsilon n}{\tfrac 1{m(n)}}=\frac{\varepsilon \vrho }2
\geq \frac{\varepsilon c_1 }2 =:\gamma>0,
\end{equation}
where $c_1$ is from \eqref{vunifbounds}.
Note that $\gamma$ does not depend on $\omega$.
\end{proof}

We finally need to show that the lower bound extends also to the Radon-Nikodym derivative of the limiting measures.

\begin{Corollary}\label{cor:rn_lower_bound} The Radon-Nikodym derivative $\frac{{\rm d} \bbQ^\rho}{{\rm d} \P}$ is uniformly bounded from below:
$
\frac{{\rm d} \bbQ^\rho}{{\rm d} \P}\geq \gamma
$, where $\gamma$ is from \eqref{radbound}.
\end{Corollary}
\begin{proof}
Take any  $f\geq 0$ continuous and bounded. Lemma \ref{lemma:finite_n_derivative}  and the fact that $R^\o_n$ has support in $\{ \t_k \o\,:\,  k=1,...,m(n)\}$ guarantee
 that, for all $n$ large enough,
$$
E^{\o,\rho}_0[Q^\o_n(f)] \geq \gamma R^\o_n(f) \qquad \text{for $\bbP$--a.e.~}\o\,.
$$
Passing to the limit $n \to \infty$, and observing that, by Proposition \ref{prop:convergences}, $E^{\o,\rho}_0[Q_n^\o(f)]\to \bbQ^\rho(f)$ and $R^\o_n(f)\to \P(f)$ for $\bbP$--a.e.~$\o$, we have that
$\bbQ^\rho(f)\geq \gamma\, \P(f)$. The claim follows from the arbitrariness of $f$.
\end{proof}


\subsection{The weak limit of $\bbQ^\rho$ as $\rho \to \infty$}\label{deboluccio} 
 Recall the definition of the function $F$ given in \eqref{def_F} and of the constant $\g$ given in Corollary \ref{cor:rn_lower_bound}. 
\begin{Lemma}\label{chiave} Suppose $\bbE \bigl[ e^{(1-\l) x_1}\bigr]< \infty$. Then the  following holds:
\begin{itemize}
\item[(i)] The family of probability measures $( \bbQ^\rho) _{\rho \in \N_+} $ is tight;
\item[(ii)] Any subsequential limit $\bbQ^\infty $ of $( \bbQ^\rho) _{\rho \in \N_+} $ is absolutely continuous to $\bbP$ and
$$
0<\gamma\leq \frac{d \bbQ^\infty }{ d \bbP} \leq F \qquad \text{$\bbP$--a.s.}
$$
\end{itemize}
\end{Lemma}

\begin{proof}
For proving part (i), fix an increasing  sequence of compact subsets $K_n$ exhausting all of $\O$.  Thanks to Proposition \ref{aereo} we have
\[
\bbQ^\rho ( K_n ^c) =  \bbE \left[  \frac{d\bbQ^\rho}{d \bbP} \mathds{1} _{K_n^c} \right] \leq \bbE \left[ F   \mathds{1} _{K_n^c} \right]\,.
\]
Setting $f_n := F \mathds{1} _{K_n^c}$ we have that $0 \leq f_n \leq F $ and $ f_n (\o) \to 0$ everywhere. By the dominated convergence theorem, given $\e>0$  we conclude that
$\bbQ^\rho ( K_n ^c) \leq \e$ eventually in $n$, hence the tightness.

We turn now to (ii). By Prohorov's Theorem there exists a sequence $\rho_k\to \infty$ such that   $\bbQ^{\rho_k}$ converges weakly to some probability measure $\bbQ^\infty$.  We want to prove the absolute continuity of $\bbQ^\infty$ with respect to $\bbP$.  To this aim fix a measurable set $A \subset \O$ with $\bbP(A)=0$. We need to show that $\bbQ^\infty(A)=0$. Due to \cite{B}[Thm. 1.1], for each integer $m \geq 1$  there exists an open subset $G_m $ with $A \subset G_m \subset \O$ and  $\bbP(G_m)= \bbP(G_m \setminus A) \leq 1/m$.  Due to the  Portmanteau Theorem (cf. \cite{B}[Thm. 2.1]) we conclude that
\begin{equation}\label{pugilato}
\bbQ^\infty(A)
	\leq  \bbQ^\infty(G_m)
	\leq \liminf_{\rho \to \infty} \bbQ^{\rho} (G_m)
	= \liminf_{\rho \to \infty}  \bbE \left[ \frac{d\bbQ^{\rho} }{d \bbP}  \mathds{1} _{G_m} \right]
	\leq   \bbE \left[F  \mathds{1} _{ G_m} \right] \,.
  \end{equation}
  Since the  sequence of subsets $\{G_m\}_{m \geq 1}$ can be taken decreasing and since $F \in L^1 (\bbP)$ (cf.~Prop.~\ref{aereo}),  we derive that the r.h.s.~of \eqref{pugilato} goes to zero as $m \to \infty$ by the dominated convergence theorem. This proves that $\bbQ^\infty(A)=0$, thus implying that $\bbQ^\infty \ll \bbP$.

Let us prove that $ \frac{d \bbQ^\infty}{ d \bbP} \leq F$, $\bbP$--a.s. To this aim we take $G \subset \O$ open. By the Portmanteau theorem, we have $\liminf _{\rho \to  \infty} \bbQ^\rho (G) \geq \bbQ^\infty(G)$. On the other hand, $\bbQ^\rho (G) = \bbE [\frac{d \bbQ^\rho }{ d \bbP}   \mathds{1}_G]  \leq \bbE [ F \mathds{1}_G]$. Hence
\begin{equation}\label{pizza100}
\bbE [ F \mathds{1}_G ] -\bbQ^\infty(G)
	=\bbE\Big[ \Big(F- \frac{d \bbQ^\infty}{d \bbP}\Big) \mathds{1}_G\Big]
	\geq 0
\end{equation}
for any $G $ open. Suppose by contradiction that $\bbP(A)>0$ where $A:= \{
F- \frac{d \bbQ^\infty}{ d \bbP} <0\}$. By  \cite{B}[Thm. 1.1]  there exists a decreasing sequence $(G_m)_{m \geq 1} $ of open subsets such that $A \subset G_m$ and $\bbP( G_m \setminus A) \leq 1/m$ for any $m$. The last bound implies that $\mathds{1}_{G_m \setminus A} \to 0$ in $L^1 (\bbP)$ as $m \to \infty$, hence at the cost of extracting a subsequence we can assume that $\mathds{1}_{G_m \setminus A} \to 0$ $\bbP$--a.s.~as $m \to \infty$. By applying now the dominated convergence theorem we  conclude that
 $ \lim _{m \to \infty} \bbE[(F- \frac{d \bbQ^\infty}{ d \bbP}) \mathds{1}_{G_m} ] =\bbE[(F- \frac{d \bbQ^\infty}{ d \bbP}) \mathds{1}_A] $. By definition of $A$ and since $\bbP(A) >0$,  it must be $\bbE[(F- \frac{d \bbQ^\infty}{ d \bbP}) \mathds{1}_A]<0$. On the other hand, due to \eqref{pizza100}, $\bbE[(F- \frac{d \bbQ^\infty}{ d \bbP}) \mathds{1}_{G_m} ] \geq 0$, thus leading to a contradiction.

The proof that  $ \frac{d \bbQ^\infty}{ d \bbP}\geq \g $, $\bbP$--a.s.,  follows similar arguments.  In particular, by the Portmanteau theorem, one gets that  $\g \bbP(C) \leq \bbQ^\infty(C)$ for all $C \subset \O$ closed.  Moreover, by  \cite{B}[Thm. 1.1],  for any $A \subset \O$ Borel there exists an increasing sequence $(C_m)_{m \geq 1}$ of closed sets such that $C_m \subset A$ and $\bbP( A \setminus C_m) \leq 1/m$.
\end{proof}


\section{Proof of Theorem \ref{teo1}: transience to the right }\label{proof_transience}
 By the discussion at the end of Section \ref{zetino},  it is enough to show the a.s.~transience to the right of $X^\infty_n$ and $\bbX^\infty_t$. Since the former is the jump chain associated to  the latter,  we only need to derive the a.s.~transience to the right of $X^\infty_n$.
 To this aim, it is sufficient to show that, for any $m\in\N$, there exists some $n(m,\o)<\infty$ such that $X^\infty_n>m$ for all $n\geq n(m,\o)$.

First of all notice that, by Proposition \ref{prop:function_g}, for $\bbP$--almost every $\o\in\Omega$ and $ i \in \bbZ$  we have
\begin{align}
	E_i^{\o, \infty} [N_\infty((-\infty,i])]
	&\leq \sum_{k=0}^\infty g_{ \t_i \o}(k) \nonumber\\
	& = K_0 \Big( \sum_{k=0}^\infty K_0 \pi^1 (-k)[ \t_i \o]  \Big) \cdot \Big(
 		\sum _{j=0}^\infty
		e^{-2\l x_j (\t_i \o) + (1-\l)(x_{j+1}(\t_i\o) -x_j(\t_i\o) )}\Big) , \nonumber
\end{align}
which is $\bbP-$almost surely finite  (see \eqref{spring0} and the discussion after 
Prop.~\ref{prop:function_g}). Hence
\begin{equation}\label{eccomi}
P_i^{\o,\infty}(N_\infty((-\infty,i])<\infty)=1.
\end{equation}

Now fix $m\in\N$ and consider $T_m$, the first time the random walk is larger or equal than $m$. Applying the Markov property at time $T_m$ and using \eqref{eccomi} one gets the claim.

\section{Proof of Theorem 1: The ballistic regime}\label{ballistic_part}
In this section we assume that $\bbE[ {\rm e}^{(1-\l) x_1} ] <+\infty$ and that $u: \bbR \times  \bbR  \to \bbR$ is continuous. 
Recall that $(\bbY)_{t \geq 0}$ and  $(Y_n)_{n \geq 0}$ denote the continuous time Mott random walk and the associated jump process, respectively. Recall also the definition of the Markov chains $( \bbX_t^\infty)_{ t \geq 0}$ and $(X_n^\infty)_{n \in\N}$, given in Section \ref{zetino}
and that $ \cP ^\rho _{Q}$ is the law of the process \emph{environment viewed from the $\rho$--walker} $(\tau_{X_n^\rho}\o)_{n \in\N }$ when started with some initial distribution $Q$.

Given $\rho \in \bbN_+ \cup \{+\infty\}$, by writing $(X_n ^\rho)_{n\in\N}$ as a functional of $(\tau_{X_n^\rho}\o)_{n \in\N}$  and using the ergodicity of $\bbQ^\rho$ (cf.~Lemma \ref{lemma:ergodicity_q_rho}
 and Proposition \ref{digiuno}) we get that the asymptotic velocity of $(X_n ^\rho)_{n \geq 0}$ exists $\cP^\rho _{\bbQ^\rho}$--a.s.~and therefore $\cP^\rho _{\bbP}$--a.s.~since  $\bbQ^\rho$ and $\bbP$ are mutually absolutely continuous:
\begin{equation}\label{ankara}
v_{X^\rho}(\l) := \lim _{ n \to \infty} \frac{X_n^\rho}{n}  \qquad \text{  $\cP^\rho _{\bbQ^\rho}$--a.s.  and $\cP^\rho _{\bbP}$--a.s. }
\end{equation}
Moreover, $ v_{X^\rho}(\l) $ does not depend on $\o$ and can be characterized as
  \begin{equation}\label{int_rapr}
 v_{X^\rho}(\l) :=   \E^\rho\big[E^{\o,\rho}_0[X_1]\big]=\E^\rho\Big[\sum_{m\in\Z}m\,P^{\o,\rho}_0(X_1=m)\Big],\qquad  \forall \rho \in \bbN_+\cup \{+\infty\}\,.
 \end{equation}
 Here, $\E^\rho$ denotes  the expectation with respect to $\bbQ^\rho$. 
 Recall that for $\rho < \infty$ we have also  an alternative representation for $ v_{X^\rho}(\l)$ (see Proposition \ref{prop_vel}).

  \medskip
  
We now prove that 
\begin{equation}\label{valeria}
	\lim _{\rho \to \infty} v_{X^\rho } (\l)= v_{X^\infty}(\l)\,.
\end{equation}
By the exponential decay of the jump probabilities (see \eqref{eqn:condition_c}), for all $\delta>0$ there exists $m_0\in\N$ such that, for all $\rho$,
$$
\sum_{|m|>m_0}|m| P^{\o,\rho}_0(X_1=m)<\delta \qquad\P\text{-a.s.}
$$
We now observe that, for $\rho>|m| >0$, we have  
\begin{equation}\label{uffa}\bbP^{\o,\rho}_0(X_1=m)=P^{\o,\infty}_0(X_1=m)= \frac{ c_{0,m}
 (\o)}{ \sum_{ k \in \bbZ} c_{0,k} (\o) } \,,
 \end{equation}
and the r.h.s. of \eqref{uffa}  is continuous in $\o$ due to the continuity assumption on $u$ and since $\|c_{0,k} (\cdot) \|_\infty \leq e^{-(1-\l) dk +\|u\|_\infty}$.
Since $\bbQ^\rho\xrightarrow{w}\bbQ^\infty$, it is now simple to get \eqref{valeria}.

\medskip
Finally, we also have that $ v_{X^\infty}(\l) \in [c_1,c_2] $ because of the limit \eqref{valeria} and since, by Proposition \ref{prop_vel},   $\vrho(\l)  \in (c_1,c_2)$ for suitable strictly positive constants $c_1,c_2$.

\medskip
By  the previous observations and by the second identity in \eqref{7575}, we also obtain that the limit
\begin{equation}\label{radisson}
v_Y (\l)  := \lim _{n \to \infty } \frac{Y_n}{n}
\end{equation}
exists $\cP^\infty_{\bbP}$--a.s.~and equals  $\bbE[Z_0] v_{X^\infty} (\l) $. As a consequence,  
$v_Y (\l)$
 is deterministic, finite and strictly positive.

\medskip 

By a suitable time change we can recover the LLN for $(\bbX^\infty_t)_{t \geq 0}$ from the LLN for  $(X^\infty_n)_{n \geq 0}$ as follows.
   By enlarging the probability space $( \O ^{\bbN}, \cP^\infty _{\bbQ^\infty} )$ with a product space, we introduce a sequence of i.i.d.~exponential random variables $(\beta_n)_{n \geq 0}$ of mean one, all independent from the process \textit{environment viewed from the $\infty$--walker} $(\t_{X_n ^\infty} \o)_{n\in\N}$. We call $( \O ^{\bbN}\otimes \bbR_+ ^\bbN, \bar \cP^\infty _{\bbQ^\infty} )$ the resulting probability space. Note that  $\bar \cP^\infty _{\bbQ^\infty} $  is stationary and ergodic with respect to shifts. On $( \O ^{\bbN}\otimes \bbR_+ ^\bbN, \bar \cP^\infty _{\bbQ^\infty} )$ we define the random variable 
\[ S_n:= \sum_{k=0}^{n-1} \frac{ \beta_k}{r  (\t_{ X_k^\infty} \o  )}\,, \qquad r (\o):= \pi^\infty (0)[\o]= \sum _{ k \in \bbZ} c_{0,k} (\o) \,. 
\]
We note that $r(\o)$ coincides with $r_0^\l(\o)$ of Section \ref{notazione}.
By the ergodicity of $\bar \cP^\infty _{\bbQ^\infty} $ we have
\begin{equation}\label{mosca}
\lim _{n \to \infty} \frac{S_n}{n}= \bbE^\infty \bigl[ 1/r \bigr ] \qquad \bar \cP^\infty _{\bbQ^\infty}\text{--a.s.} 
\end{equation}
Since, by Proposition \ref{digiuno}, $\bbQ ^\infty\ll \bbP$ and  $\frac{d \bbQ^\infty}{ d \bbP} \leq F $ with $F$ defined in \eqref{def_F}, using Lemma \ref{lorenzo}, Assumption (A4) and the hypothesis $\bbE[ {\rm e}^{(1-\l) Z_0} ] <+\infty$ we get 
\begin{equation}\label{stima_key}
\begin{split}
 	0< 
		\bbE^\infty \bigl[ 1/r \bigr ] & \leq K \,\bbE\Big[\frac{ \pi^1(0)}{r}  \sum_{j=0}^\infty  (j		+2)^2{\rm e}^{-2\l x_j +(1-\l) (x_{j+1}-x_{j}) } \Big]\\
 	& \leq K' \sum _{j=0}^\infty (j+2)^2  e^{- 2 \l d} \bbE[ {\rm e}^{(1-\l) Z_0} ] <+\infty\,.
 \end{split}
\end{equation}
For any $t \geq 0$  we define  $n(t)$ on  $( \O ^{\bbN}\otimes \bbR_+ ^\bbN, \bar \cP^\infty _{\bbQ^\infty} )$ as the only integer $n$  such that
$ S_n \leq t < S_{n+1}$.  By \eqref{mosca} and \eqref{stima_key} we get that 
$n(t) \to \infty$ as $ t \to \infty$, $\bar \cP^\infty _{\bbQ^\infty}\text{--a.s.}$ As a byproduct of  the above limit,   of \eqref{mosca} and  the bound 
\begin{equation}\label{alexey}
\frac{S_{n(t)}}{n(t)} \leq \frac{t}{n(t)}  < \frac{S_{n(t)+1} }{n(t)}\,,
\end{equation}
 we conclude that   
\begin{equation}\label{turchia}
\lim _{n \to \infty} \frac{n(t)}{t}= \frac{1}{\bbE^\infty \bigl[ 1/r \bigr ] }\qquad \bar \cP^\infty _{\bbQ^\infty}\text{--a.s.} 
\end{equation} 
By writing $\frac{X^\infty _{n(t)}}{t}= \frac{X^\infty _{n(t)}}{n(t)}\frac{n(t)}{t}$,  from 
\eqref{ankara} and 
\eqref{turchia} we get that 
\begin{equation}\label{nato}
\lim _{t\to  \infty} \frac{X^\infty _{n(t)}}{t}= \frac{v_{X ^\infty}(\l) }{ \bbE^\infty \bigl[ 1/r \bigr ]}\,, \qquad  \bar \cP^\infty _{\bbQ^\infty}\text{--a.s.} 
\end{equation}
At this point it is enough to observe that the process $(X^\infty _{n(t)})_{t\geq 0} $ defined on the probability space $( \O ^{\bbN}\otimes \bbR_+ ^\bbN, \bar \cP^\infty _{\bbQ^\infty} )$
has the same law as the process $(\bbX^\infty_t)_{t \geq 0}$. Using also \eqref{stima_key} and the fact that $\cP^\infty _{\bbP}\ll  \cP^\infty _{\bbQ^\infty}   $, we conclude that 
\begin{equation} \label{quasi_fatto}
v_{\bbX^\infty}(\l):=\lim _{t \to \infty } \frac{\bbX^\infty_t}{t}=
 \frac{v_{X ^\infty}(\l) }{ \bbE^\infty \bigl[ 1/r \bigr ]}\in (0,+\infty)
\end{equation}
holds $P^{\o, \infty}_0$--a.s., for $\bbP$--a.e.~$\o$.
Finally, using \eqref{7575}, we conclude that 
\begin{equation} \label{fatto}
v_{\bbY}(\l):=\lim _{t \to \infty } \frac{\bbY_t}{t}=
 \frac{\bbE( Z_0)  }{ \bbE^\infty \bigl[ 1/r \bigr ]}  v_{X ^\infty}(\l)  \in (0,+\infty)
\end{equation}
holds for almost all trajectories of the Mott random walk, for $\bbP$--a.e.~$\o$. As already observed, the r.h.s.~of \eqref{fatto} is deterministic and this concludes the proof of Theorem \ref{teo1}--(ii) and its counterpart for the jump process $(Y_n)_{n \geq 0}$
(cf.~\eqref{radisson}).

\section{Proof of Theorem \ref{teo1}: The sub-ballistic regime}\label{submarine}
First we point out that it will be sufficient to prove that $v_{X^\infty} (\l)=0$ a.s., for $\bbP$--a.e. realization of the environment $\o$:
Recall the identities \eqref{9898} and  \eqref{7575} of Section \ref{zetino}.  {By Assumptions (A1) and }  (A2), $ \lim _{ i \to \infty} \psi(i) /i= \bbE[Z_0]<\infty $,  $\bbP$--a.s.. On the other hand,  as proved in Section \ref{proof_transience}, the random walks 
$X^\infty_n$ and $\bbX^\infty_t$ are a.s.~transient to the right. As a byproduct, due to 
 \eqref{9898} and \eqref{7575},  we have  $v_Y(\l)=0$, $v_{\bbY}(\l)=0$ whenever  $v_{X^\infty} (\l)=0$, $v_{\bbX^\infty}(\l)=0$, respectively.  
 But we also have that $v_{X^\infty} (\l)=0$ implies $v_{\bbX^\infty} (\l)=0$. Indeed, the continuous time random walk $(\bbX^\infty_t)_{t \geq 0} $ is obtained from the discrete time random walk $(X^\infty_n)_{n \geq 0}$ by the rule that, when site $k$ is reached, $\bbX^\infty$ remains at $k$ for an exponential time with parameter $r_k^\l (\o)$. Since $\sup_{ k \in \bbZ, \o \in \O}r_k^\l (\o) =: C < \infty$ (cf.~Section \ref{notazione}), we can speed up $\bbX^\infty$ by  replacing  all parameters $r_k^\l (\o)$ by $C$. The resulting random walk can be realized as $t \mapsto X^\infty_{ n(t)} $ where $\bigl( n(t) \bigr)_{t \geq 0}$ is a Poisson process with intensity $C$.  Hence, its velocity is zero whenever $v_{X^\infty} (\l)=0$. 

\medskip
We  first show in Proposition \ref{scarola}  a  sufficient condition
for $v_{X^\infty} (\l)=0$. In Corollary \ref{cor:0_speed} we  prove that this condition is equivalent to  
the hypothesis \eqref{immacolata} of Theorem \ref{teo1}--(iii) and  in Corollary \ref{cor:iid_0_speed} we discuss  some stronger conditions corresponding to the last statement in  Theorem \ref{teo1}--(iii).

\begin{Proposition}\label{scarola}
Suppose that
\begin{equation}\label{eqn:0_speed_condition2}
\E\Big[ \,\Big( \sup_{z\leq 0}P^{\o,\infty}_z(X_1\geq 1)\Big)^{-1} \,\Big]=\infty.
\end{equation}
Then $v_{X^\infty }(\l)=0$.
\end{Proposition}

A basic tool in the proof of the above proposition will be the following coupling:
\begin{Lemma}[Quantile coupling]\label{cicoria}
For a distribution function $G$ and a value $u\in[0,1]$, define the function
$$
\phi(G,u):=\inf\{x\in\R:\,G(x)>u\}\,.
$$
 Let $F$ and $F'$ be two distribution functions such that $F(x)\leq F'(x)$ for all $x\in\R$. Take $U$ to be a uniform random variable on $[0,1]$ and let
$
Y:=\phi(F,U)
$
and
$
Y':=\phi(F',U).
$
Then   $Y$ is distributed according to $F$, $Y'$ is distributed according to $F'$ and $Y\geq Y'$ almost surely.
\end{Lemma}
The proof of the above fact can be found in \cite{T}. 
Usually, as in \cite{T}, the quantile coupling is defined with $\phi_q(G,u)$ instead of $\phi(G,u)$, where $\phi_q(G,u)$ is the quantile function $\phi_q(G,u):=\inf\{x\in\R:\,G(x)\geq u\}$. One can easily prove that $\phi(G,U)= \phi_q(G,U)$ a.s.

\begin{proof}[Proof of Proposition \ref{scarola}]
Call $F_\xi$ the distribution function of the random variable $\xi:= L + G$, where  $L\in\N$ is some constant such that
\begin{equation}\label{pesto}
{\frac {{\rm e}^{u_{\max}-u_{\min}}{\rm e}^{-(1-\l)dL}}{1- \rm{e}^{-(1- \l)d}} }  <1\,,
\end{equation}
and $G$ is a  geometric random variable with parameter $\gamma=1-{e^{-(1-\l)d}}$. 
 Note that given an integer $a$ it holds
 \begin{align}
 1- F_\xi(a)= \begin{cases}
 1 & \text{ if } a-L \leq 0\,,\\
 (1-\g)^{a-L}= e^{-(1-\l) d (a-L) }  & \text{ if } a-L \geq 1\,.
 \end{cases}
 \end{align}
 In particular, given an integer $M \geq L+2$, due to \eqref{pesto}
  we have 
 \begin{align}\label{david}
 {\frac{{\rm e}^{u_{\max}-u_{\min}}{\rm e}^{-(1-\l)d(M-1)}}{1- \rm{e}^{-(1- \l)d}}   < {\rm e}^{-(1-\l)d(M-1-L)} =
 1- F_\xi(M-1) \,.}
 \end{align}

We will now inductively construct a sequence of probability spaces $( \O\times\Z^\N\times[0,1]^n, P^{(n)})$, on which we will define some random variables.


\medskip

\noindent \underline{STEP 1}. We  first consider the space $\O\times\Z^\N\times[0,1]$, the   generic element  of which is denoted by 
 $(\o,\bar x,u_1)$.

 We introduce  a probability $P^{(1)}$ on  $\O\times\Z^\N\times[0,1]$ by the following rules. 
The marginal of $P^{(1)}$ on 
  $\O $ is $\bbP$, its  marginal on $[0,1]$ is the uniform distribution and, under $P^{(1)}$, the coordinate functions $(\o,\bar x,u_1) \mapsto \o$ and $(\o,\bar x,u_1) \mapsto u_1 $ are independent random variables. Finally,
  we require that 
  \begin{equation}\label{forza}
  P^{(1)}(X^{(1)}_\cdot \in A|\, \o, u_1)=P^{\o, \infty}_0(X_\cdot\in A|\,X_{T_1}=  \phi(F^{(1)}_\o ,u_1)  )
\end{equation} 
for any measurable set $A\subseteq \Z^\N$, where $(X^{(1)}_n)_{n\in\N}$ is the second--coordinate function $(\o,\bar x,u_1) \mapsto \bar x$ and 
\[ 
	F_\o^{(1)}(y)=P^{\o,\infty} _0(X_{T_1}\leq y)\, , \qquad T_1=\inf\{n\in\N:\,X^\infty_n>0\}.
\]
From now on we consider the space $\O\times\Z^\N\times[0,1]$  endowed with the probability $P^{(1)}$.

 It is convenient to introduce the random variables $U_1,\xi_1,W_1$ defined as follows\footnote{We will denote the first--coordinate function again by $\o$, without introducing a new symbol.}:  
$$ 
U_1(\o,\bar x,u_1):=u_1\,, \qquad 
\xi_1(\o,\bar x,u_1):=\phi(F_\xi,u_1)\,, \qquad W_1 (\o,\bar x,u_1) :=\phi(F^{(1)}_\o,u_1)\,.
$$
Note that,  by the quantile coupling (cf. Lemma \ref{cicoria}),   $\xi_1$ is distributed as $\xi$ and $W_1$ under $P^{(1)}( \,\cdot\, |\,\o)$ is distributed as $X^\infty_{T_1}$ under $P^{\o, \infty}_0$.

The interpretation to keep in mind is the following: $(X^{(1)}_n)_{n\in\N}$ plays the role of our initial random walk in environment $\o$; $W_1$ is the overshoot at time $T_1$, i.e.~how far from $0$ the random walk will land the first time it jumps beyond the point $0$;  $\xi_1$ is a positive random variable that dominates $W_1$ (see Claim \ref{strazietto}) and that is distributed like $\xi$ .

\begin{Claim}\label{adriana}
For any  integer $M \geq 1$ it holds 
\begin{equation}\label{villa}
P^{\o,\infty}_0(X_{T_1}\geq M) \leq \sup_{z\leq 0} P^{\o,\infty}_z(X_1\geq M|\,X_1\geq 1)\,.
\end{equation}
\end{Claim}
\begin{proof}[Proof of Claim \ref{adriana}] Given $j \geq 1$ and integers $z_1,z_2, \dots, z_{j-1} \leq 0 $ 
we denote by\\
$E(z_1, z_2,\dots, z_{j-1})$ the event $\{X^\infty_1= z_1, \dots, X^\infty_{j-1}= z_{j-1} \}$. 
Note that, by the Markov property, 
$$
\frac{
P^{\o,\infty}_0( X_j \geq M,  E(z_1, \dots, z_{j-1} ) ) 
}{
P^{\o,\infty}_0( X_j \geq 1, E(z_1, \dots, z_{j-1} ))
}
	= \frac{P^{\o,\infty}_{z_{j-1} }(X_1 \geq M) }{P^{\o,\infty}_{z_{j-1} }(X_1 \geq 1) }
	= P^{\o,\infty}_{z_{j-1} }(X_1\geq M\,|\,X_1 \geq 1) \,.
$$
By the above identity  we can write 
\begin{align}
	&P^{\o,\infty}_0(X_{T_1}\geq M)  \nonumber\\
	& = \sum_{j =1}^\infty \sum_{z_1, \dots, z_{j-1} \leq 0} P^{\o,\infty}_0\left(X_{j}\geq M\,|\, 		X_j \geq 1, E(z_1, \dots, z_{j-1})\right)P^{\o,\infty}_0\left( X_j \geq 1, E(z_1, \dots, 		z_{j-1})\right)\nonumber\\
	& \leq \sup_{z\leq 0} P^{\o,\infty}_z(X_1\geq M|\,X_1\geq 1)\sum_{j =1}^\infty \sum_{z_1, 		\dots, z_{j-1} \leq 0}  P^{\o,\infty}_0\left( X_j \geq 1\,,\;E(z_1, \dots, z_{j-1})\right) 		\nonumber\\
	& \leq  \sup_{z\leq 0} P^{\o,\infty}_z(X_1\geq M|\,X_1\geq 1)\,.\qedhere
\end{align}
\end{proof} 
\begin{Claim}\label{strazietto} The following holds:
\begin{itemize}
\item[(i)] $P^{(1)}(\xi_1\geq W_1)=1$;
\item[(ii)] $\xi_1$ is independent of $\omega$ under $P^{(1)}$;
\item[(iii)] $P^{(1)}(X^{(1)}_\cdot \in B|\,\o)=P^{\o,\infty}_0(X_\cdot \in B)$  for each measurable set $B\subset \Z^\N$.
\end{itemize}
\end{Claim}
\begin{proof}[Proof of Claim \ref{strazietto}]
In order to show (i), we just have to prove that $F^{(1)}_\o(x)\leq F_\xi(x)$ for all $\o \in \O$ and  $x\in\R$ (in fact, it is enough to prove it for all $x\in\N$) thanks to Lemma \ref{cicoria}. To this aim, recall the definition of $L$ (see \eqref{pesto}) and notice that for all $\omega\in\Omega$ and all integers  $M \geq L+2$, one has 
\begin{align}\label{intramezzo}
	1-F_\o^{(1)}(M-1) & = P^{\o,\infty}_0(X_{T_1}\geq M)
		\leq \sup_{z\leq 0} P^{\o,\infty}_z(X_1\geq M|\,X_1\geq 1) \nonumber \\
	&= \sup_{z\leq 0}\frac{\sum_{j\geq M}{ {\rm e}^{-(1-\l)(x_j-x_z)+ u(E_z,E_j)}}} 
		{\sum_{j\geq 1}{\rm e}^{-(1-\l)(x_j-x_z)+ u(E_z,E_j)}} \leq  {\rm e}^{u_{\max}-u_{\min}}		{\sup_{z\leq 0} }   \frac{\sum_{j\geq M}{\rm e}^{-(1-\l)(x_j-x_z)}}
		{ {\rm e}^{-(1-\l)(x_1-x_z)}}  \nonumber \\
	&= {\rm e}^{u_{\max}-u_{\min}} \sum_{j\geq M}{\rm e}^{-(1-\l)(x_j-x_1)}\nonumber 			\leq { {\rm e}^{u_{\max}-u_{\min}} } \sum_{j\geq M}{\rm e}^{-(1-\l)d(j-1)}\nonumber\\ 
	&= {\rm e}^{u_{\max}-u_{\min}}  \frac{{\rm e}^{-(1-\l)d(M-1)}}{1-{\rm e}^{-(1-\l)d}} \leq 1-F_		\xi(M-1)\,,
\end{align}
where in the first line we have used Claim \ref{adriana} and in the last bound we have used \eqref{david} and the fact that $M \geq L+2$.
This proves that $F^{(1)}_\o (a)\geq F_\xi(a)$ for all $a \in \bbN$ with $a \geq L+1$. The same  inequality trivially holds also for $a\leq L$ since in this case $F_\xi(a)=0$ (because $\xi >L$).
%

Part (ii) is clear 
since $\xi_1$ is determined only by $U_1$, while $U_1$ and $\o$ are independent by construction.

 For part (iii) take some measurable set $B\subset \Z^\N$ and notice that (recalling \eqref{forza} and the independence of $\omega$ and $U-1$) 
\begin{align*}
P^{(1)}(X^{(1)}_\cdot \in B|\,\o)
&= \int_{[0,1]}P^{(1)}(X^{(1)}_\cdot \in B|\,\o,\,U_1=u_1) P^{(1)}(U_1\in{\rm d}u_1)\\
&=\int_{[0,1]}P^{\o,\infty} _0(X_\cdot \in B|\,X_{T_1}=\phi(F^{(1)}_\o ,u_1)   )\,{\rm d}u_1\\
&=\sum_{j=1}^\infty P^{\o,\infty}_0(X_\cdot \in B|\,X_{T_1}=j)P^{\o,\infty} _0(X_{T_1}=j)=P^{\o,\infty}_0(X_\cdot \in B)\,. \qedhere
\end{align*}
\end{proof}
\medskip

\noindent \underline{STEP k+1}.
Suppose  now we have achieved our construction up to step $k$. In particular, we have built   the probability $P^{(k)}$ on the space $\O\times\Z^\N\times[0,1]^{k}$ and several random variables on $(\O\times\Z^\N\times[0,1]^{k}, P^{(k)})$ that we list:
\begin{itemize}
\item 
$U_1, \dots,U_k$ are independent and uniformly distributed  random variables such that $(U_1, \dots, U_k)$ is the projection function on  $[0,1]^k$;
\item $\xi_1, \dots, \xi_k$ is defined as $\xi_j= \phi (F_\xi, U_j) $, $j=1,\dots, k$;
\item $( X_n^{(k)})_{n \geq 0}$, defined as the projection function on $\Z^\N$, 
whose law under $P^{(k)}(\cdot |\,\o) $ is $P^{\o, \infty} _0$;
\item $W_1,W_2, \ldots, W_k$  such that $P^{(k)} ( \xi_{i} \geq W_i \text{ for all } i =1, \dots,k)=1$.
\end{itemize}

\begin{figure}
  \centering
  \vspace{-2.2cm}
  \setlength{\unitlength}{0.12\textwidth}  
  \begin{picture}(5,5)(0,1)
    \put(-2,0){\includegraphics[scale=0.50]{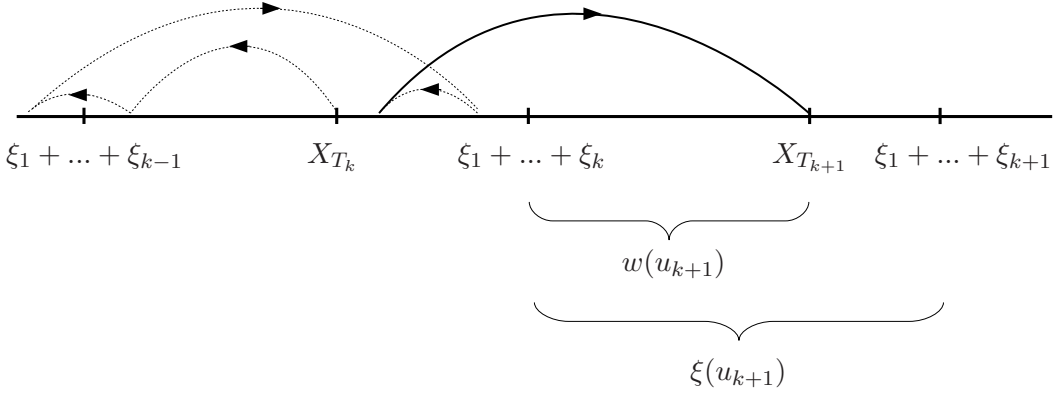}}
	 \put(0.5,2.7){$X_{T_k}$}
 	 \put(3.9,2.7){$X_{T_{k+1}}$}
 	 \put(-1.7,2.7){$\xi_1+...+\xi_{k-1}$}
 	 \put(1.6,2.7){$\xi_1+...+\xi_{k}$}
	 \put(4.65,2.7){$\xi_1+...+\xi_{k+1}$}
	 \put(2.8,1.9){$w(u_{k+1})$}
	 \put(3.3,1.1){$\xi(u_{k+1})$}
  \end{picture}
  \caption{\textit{ $T_{k+1}$ is the first time the random walk overjumps the point $\xi_1+...+\xi_k$. The overshoot $w(u_{k+1})$ is dominated by $\xi(u_{k+1})$ by construction. }}

\end{figure}

We introduce  a probability $P^{(k+1)}$ on  $\O\times\Z^\N\times[0,1]^{k+1}$ by the following rules. 
The marginal of $P^{(k+1)}$ on 
  $\O $ is $\bbP$, its  marginal on $[0,1]^{k+1}$ is the uniform distribution and, under $P^{(k+1)}$, the projection  functions $(\o,\bar x,u_1, \dots, u_{k+1}) \mapsto \o$ and $(\o,\bar x,u_1, \dots, u_{k+1}) \mapsto (u_1, \dots, u_{k+1}) $ are independent random variables. Finally,
  we require that 
  \begin{equation}\label{forza1000}
  \begin{split}
 &  P^{(k+1)}\big(X^{(k+1)}_\cdot \in A|\, \o, u_1, \dots, u_k ,u_{k+1}\big)\\& \qquad   = 
 P^{(k)} \Big( X^{(k)}_\cdot \in A|\,\o, u_1, \dots, u_k, X^{(k)}_{T_{k+1}}=\xi_1+...+\xi_k+
 \phi(F^{(k+1)}_{\o, u_1, \dots, u_k},u_{k+1})   \Big)
  \end{split}
\end{equation} 
for any measurable set $A\subseteq \Z^\N$, where 
\begin{align*}
&  F^{(k+1)}_{\o, u_1, \dots, u_k}(y):=P^{(k)}(X^{(k)}_{T_{k+1}}\leq \xi_1+...+\xi_k+y\,|\, \o, u_1, \dots, u_k ) \, , \\
& T_{k+ 1}:=\inf\{n\in\N:\, X_n^{(k)}>\xi_1+...+\xi_{k}\} \,.  
\end{align*}
Note that $T_{k+ 1}$ is a  random variable on $(\O \times \Z^\N\times[0,1]^{k}, P^{(k)})$. We stress that the conditional probability in the r.h.s. of \eqref{forza1000} has to be thought of as the 
regular conditional probability 
$P^{(k)} \big(\cdot |\,\o, u_1, \dots, u_k\big)$ further conditioned on the event $\{ X^{(k)}_{T_{k+1}}=\xi_1+...+\xi_k+
 \phi(F^{(k+1)}_{\o, u_1, \dots, u_k} ,u_{k+1})\}$.

\begin{Claim}\label{strazio1}
The marginal of $P^{(k+1)}$ on  $\O\times\Z^\N\times[0,1]^{k}$ is exactly  $P^{(k)}$. 
\end{Claim}
\begin{proof}[Proof of Claim \ref{strazio1}]
Since  the marginal of  $P^{(k+1)}$ along the coordinate $u_{k+1}$ is the uniform distribution, by integrating \eqref{forza1000} over $u_{k+1}$,  we get 
\begin{multline}\label{sonno_bello1} P^{(k+1)}\big(X^{(k+1)}_\cdot \in A|\, \o, u_1, \dots, u_k\big)  = \\
\sum_{j =1}^\infty 
 P^{(k)} \big( X^{(k)}_\cdot \in A|\,\o, u_1, \dots, u_k, X^{(k)}_{T_{k+1}}=\xi_1+...+\xi_k+
j \big) \int _0^1 \mathds{1}\bigl(
 \phi(F^{(k+1)}_{\o, u_1, \dots, u_k},u)=j\bigr)  d u  \,. 
\end{multline}
Above we have used Lemma \ref{cicoria} to deduce that $\phi(F^{(k+1)}_{\o, u_1, \dots, u_k},u)$ has integer values. Applying again Lemma \ref{cicoria} and the definition of $F^{(k+1)}_{\o, u_1, \dots, u_k}$ we have
\begin{equation}\label{sonno_bello2}
\int_0^1  \mathds{1}\bigl(
\phi(F^{(k+1)}_{\o, u_1, \dots, u_k},u)=j\bigr) du 
  = P^{(k)}(X^{(k)}_{T_{k+1}}= \xi_1+...+\xi_k+j\,|\, \o, u_1, \dots, u_k ) \,.
 \end{equation}
Plugging  \eqref{sonno_bello2} into \eqref{sonno_bello1}, we get 
\begin{equation}\label{sonno_bello3} P^{(k+1)}\big(X^{(k+1)}_\cdot \in A|\, \o, u_1, \dots, u_k\big)    = 
 P^{(k)} \big( X^{(k)}_\cdot \in A|\,\o, u_1, \dots, u_k\big)\,.
 \end{equation}
 On the other hand, the projections of $P^{(k+1)} $ and $P^{(k)}$ on $\O \times [0,1]^k$, i.e.~along the coordinates $\o, u_1,\dots, u_k$, are equal by construction, thus concluding the proof of our claim.
\end{proof}

\medskip

Due to the above claim, any random variable $Y$ defined on $(\O \times \bbZ^{\N} \times[0,1]^{k}, P^{(k)})$ can be thought of as a random variable on $(\O \times \bbZ^{\N} \times[0,1]^{k+1}, P^{(k+1)})$, by considering the map $ ( \o, \bar x , u_1, \dots, u_k,u_{k+1}) \mapsto Y( \o, \bar x , u_1, \dots, u_k)$. With some abuse of notation, we denote by $Y$ also the last random variable. 

As  a consequence, $U_1, \dots, U_k, \xi_1, \dots, \xi_k, W_1, \dots,W_k$ can be thought as random variables  on $(\O \times \bbZ^{\N} \times[0,1]^{k+1}, P^{(k+1)})$. Finally, we introduce the new random variables $U_{k+1}, \xi_{k+1}, W_{k+1}$ on $(\O \times \bbZ^{\N} \times[0,1]^{k+1}, P^{(k+1)})$  defined as 
\begin{align*}
&  U_{k+1}( \o, \bar x , u_1, \dots, u_{k+1}):= u_{k+1} \,,\\
&  \xi_{k+1}( \o, \bar x , u_1, \dots, u_{k+1}):=\phi(F_\xi,u_{k+1})\,,\\
&  W_{k+1} ( \o, \bar x , u_1, \dots, u_{k+1}) :=\phi(F^{(k+1 )}_{\o, u_1, \dots, u_k},u_{k+1})\,.
\end{align*}


The interpretation is similar as in STEP 1: $W_{k+1}$ is the overshoot at time $T_{k+1}$, i.e.~how far from $\xi_1+...+\xi_{k}$ the random walk will land the first time it jumps beyond that point;  $\xi_{k+1}$ is a positive random variable that dominates $W_{k+1}$ (see Claim \ref{mega_strazio}) and that is distributed as $\xi$.

\begin{Claim}\label{mega_strazio}
The following three facts hold true:
\begin{itemize}
\item[(i)]   $P^{(k+1)}(\xi_{k+1}\geq W_{k+1})=1$;
\item[(ii)] $\xi_{k+1}$ is independent of $\omega, U_1,...,U_k$ under $P^{(k+1)}$;
\item[(iii)] For each measurable set $B\subset \Z^\N$,
$$
P^{(k+1)}( X^{(k+1)}_\cdot \in B|\,\o)=P^{\o, \infty} _0( X_\cdot \in B).
$$
\end{itemize}
\end{Claim}
\begin{proof}[Proof of Claim \ref{mega_strazio}]
The three facts can be proved
 in a similar way as  
  Claim \ref{strazietto}. We give the proof for completeness.

For Part (i) we want to show that 
$F^{(k+1)}_{\o, u_1, \dots, u_k}
(M-1)\geq F_\xi(M-1) $ for all 
$M \geq L+2$,  
with $M\in\N$. In fact, as for Claim \ref{strazietto}, this inequality can easily be extended to all $M \in \N$  and the conclusion follows.

First of all we notice that, by iteratively applying \eqref{forza1000} and 
using Claim \ref{strazietto}--(iii),
 we have
\begin{align}
& 
1- F^{(k+1)}_{\o, u_1, \dots, u_k}
(M-1)=
P^{(k)}(X^{(k)}_{T_{k+1}}\geq \xi_1+...+\xi_k+M\,|\, \o, u_1, \dots, u_k)\nonumber\\
&\qquad=P_0^{\o,\infty}(X_{\inf\{n:\,X_n> \xi(u_1)+...+\xi(u_k)\}}\geq \xi(u_1)+...+\xi(u_k)+M\,|\, D_k),\label{laborich} 
\end{align}
where we have used the shortened notation $\xi(u):=\phi(F_\xi,u)$ and $D_k$ is the event
\begin{align}
D_k:&=
\{X_{T_1}=\phi(F_\o^{(1)},u_1),X_{\inf\{n:\,X_n>\xi(u_1)\}}=\xi(u_1)+\phi(F_{\o,u_1}^{(2)},u_2)
,...,\nonumber\\
&\qquad X_{\inf\{n:\,X_n>\xi(u_1)+...+\xi(u_{k-1})\}}=\xi(u_1)+...+\xi(u_{k-1})+\phi(F_{\o,u_1,...,u_{k-1}}^{(k)},u_k)\}\,.\nonumber
\end{align}
For convenience we call
{\begin{align}
D_k'&:= \{ X_{\inf\{n:\,X_n>\xi(u_1)+...+\xi(u_{k-1})\}}=y_k \} \,, \nonumber\\
y_k&:=y_k(u_1,...,u_k):=\xi(u_1)+...+	\xi(u_{k-1})+\phi(F_{\o,u_1,...,u_{k-1}}^{(k)},u_k)\nonumber\\
w_k&=w_k(u_1,...,u_k):=\phi(F_{\o,u_1,...,u_{k-1}}^{(k)},u_k).\nonumber
\end{align} 
We also note that $\xi(u_k) \geq w_k$ $P^{(k)}$--a.s. (see the list of properties at the beginning of STEP $k+1$).  Coming back to \eqref{laborich}, by using the strong Markov Property}, we obtain (see also the proof of Claim \ref{adriana})
\begin{align}\label{regina}
P^{(k)}(X^{(k)}_{T_{k+1}}&\geq \xi_1+...+\xi_k+M\,|\, \o, u_1, \dots, u_k)\nonumber\\
	&{=P_0^{\o,\infty}\big(X_{\inf\{n : X_n >  \xi(u_1)+...+\xi(u_k)  \} }\geq \xi(u_1)+...+		\xi(u_k)+M\,|  D_k'  \big)} \nonumber \\
	&=P_0^{\tau_{y_k}\o,\infty}(X_{\inf\{n:\,X_n>\xi(u_k)  -w_k \}}	\geq\xi(u_k) 
		-w_k+M)\nonumber\\
	&=\sum_{{i\in\N_+} }P_0^{\tau_{y_k}\o,\infty}(X_{i}\geq
	\xi(u_k) 
	-w_k+M\,|\, \inf\{n:\,X_n>
	\xi(u_k) 
	-w_k\}=i)  	\nonumber\\
	& \qquad \qquad \qquad\qquad \qquad \qquad \qquad \qquad  \times P_0^{\tau_{y_k}\o,\infty}(\inf\{n:\,X_n>
	\xi(u_k) 
	\}=i)\nonumber\\
	&\leq \sup_{z\leq
	\xi(u_k) 
	-w_k}P_z^{\tau_{y_k}\o,\infty}(X_1\geq M\,|\,X_1\geq 1).
\end{align}
The last inequality follows by conditioning to  the position of the  random walk at time $i-1$.
Knowing this, 
we can proceed as in
\eqref{intramezzo} getting that the last term in \eqref{regina} is bounded from above by $1- F_\xi (M-1)$.
 This concludes the proof of Part (i).

\medskip

Part (ii) is clear by the construction of $\xi_{k+1}$. Finally, we prove Part (iii). Since the projections of $P^{(k+1)} $ and of $P^{(k)}$ on $[0,1]^k$, i.e.~along  the coordinates $u_1, \dots, u_k$, are both the uniform distribution on $[0,1]^k$, integrating \eqref{sonno_bello3} over $u_1, \dots, u_k$ we get 
$P^{(k+1)}\big(X^{(k+1)}_\cdot \in A|\, \o \big)    = 
 P^{(k)} \big(X^{(k)}_\cdot \in  A|\,\o\big)$. The claim then follows by 
 the induction hypothesis  (see the discussion at the beginning of STEP  $ k+1$).\end{proof}
 
 {Due to the results  discussed above, the list of properties at the beginning of STEP $k+1$ is valid also for $P^{(k+1)}$.}
\medskip

\noindent \underline{STEP $+\infty$}:
By the Ionescu-Tulcea Extension Theorem, there exists a measure $P^{(\infty)}$ on the space $\O\times\Z^\N\times [0,1]^\N$, random variables $\xi_1,\xi_2,...$, $W_1,W_2,...$, $T_1,T_2,...$ and a random walk $(X_n^{(\infty)})_{n\in\N}$, such that: For all measurable $A\subset \O$, $P^{(\infty)}(\o\in A)=\P(\o\in A)$; the $\xi_k$'s are i.i.d., distributed like $\xi$ and independent of $\o$; $P^{(\infty)}(X_{T_k}^{(\infty)}= \xi_1+...+\xi_{k-1}+W_k)=1$; $P^{(\infty)}(\xi_k\geq W_k)=1$; for all measurable $B\subset\Z^\N$,  $P^{(\infty)}( (X_n^{(\infty)})_{n\in\N}\in B |\,\o )=P_0^\o( (X_n^{(\infty)})_{n\in\N}\in B )$.

\medskip

We are now ready to finish the proof. Notice that, under $P^{(\infty)}(\,\cdot\,|\,\o)$, the differences $(T_{k+1}-T_{k})_{k=0,1,...}$ have a rather complicated structure, but they stochastically dominate a sequence of pretty simple objects, call them $(S_k)_{k=0,1,...}$. Each $S_k$ is a geometric random variable of parameter
\begin{equation}\label{fuso}
s_k=\sup_{z\leq 0}P^{\tau_{\xi_1+...+\xi_k}\o}_z(X_1\geq 1).
\end{equation}
In fact, due to  Lemma \ref{lemma:epsilon},
 we can imagine that for each $n\geq T_k$ the random walk ``attempts'' to overjump $\xi_1+...+\xi_k$ and manages to do so with a probability that is clearly smaller than $s_k$. By Strassen's Theorem, on an  enlarged  probability space with new probability $\tilde P^{(\infty)}$,  we can couple each $S_k$ with $T_{k+1}-T_k$ so that $S_k\leq T_{k+1}-T_k$ almost surely. Moreover, due to the strong Markov property of the random walk, all the $S_k$'s can be taken independent once we have fixed the parameters $s_k$'s.
Now note the key fact that, since the $\xi_\cdot$'s are independent of the environment and that the GCD of the values attained with positive probability by the  $\xi_\cdot$'s is $1$, the shifts $(\tau_{\xi_1+...+\xi_k}\omega)_{k\in\N}$ form a stationary ergodic sequence under $P^{(\infty)}$. We refer to   Appendix \ref{appendix:ergodic} for a proof of this fact (see Lemma \ref{pasquetta}). 
This observation allows to prove that $(S_j)_{j\in\N}$ is a stationary  ergodic sequence with respect to shifts under $\tilde P^{(\infty)}$ (see 
 Lemma \ref{luce} in Appendix  \ref{appendix:ergodic}).

We now take $\o \in\O$ such that $\lim_{n \to \infty} X_n=+\infty$ $P_0^\o$--a.s. (which holds for $\bbP$--a.a.~$\o$ by Theorem \ref{teo1}--(i)). This implies that $\liminf_{n\to\infty}\frac{X_n}{n}\geq 0$, $P_0^\o$--a.s. 

\
We can bound (see \eqref{passatamutti})  
\begin{align*}
P_0^\o\Big(\limsup_{n\to\infty}\frac{X_n}{n}>0\Big)
& = P^{(\infty)}\Big(\limsup_{n\to\infty}\frac{X_n}{n}>0\,\Big|\,\o\Big)  \leq  P^{(\infty)}\Big(\limsup_{k\to\infty}\frac{X_{T_{k+1}}}{T_k}>0\,\Big|\,\o\Big) \\
& \leq    P^{(\infty)}\Big(\limsup_{k\to\infty}\frac{{ \xi_1+...+\xi_{k+1}}}{\sum_{j=0}^{k-1} (T_{j+1}-T_{j})}
>0\,\Big|\,\o\Big)       \\
& \leq  \tilde  P^{(\infty)}\Big(\limsup_{k\to\infty}\Big(\frac{  { \sum_{i=1}^{k+1}\xi_i}}{k}\Big)
\Big(\frac{\sum_{j=0}^{k-1} S_j}{k}\Big)^{-1} >0\,\Big|\,\o\Big)  .
\end{align*}

Let us concentrate on the last line.  The arithmetic mean of {$\xi_1, \dots \x_{k+1}$} converges almost surely to $L+1 /\gamma$, the mean of $\xi$, by the law of large numbers.
The arithmetic mean of $S_0, \dots, S_{k-1}$   converges instead to $\E[S_0]$ because of the ergodic theorem (for simplicity, we write simply $\bbE$ for the expectation with respect to $\tilde  P^{(\infty)}$). Since $\E[S_0]=\E[\E[S_0|s_0]]=\E[\tfrac 1{s_0}]=\infty$ by assumption, we obtain that $P_0^\o\bigl(\limsup_{n\to\infty}\frac{X_n}{n}>0\bigr)=0$ for almost all $\o\in\O$. Taking into account that $\liminf_{n\to\infty}\frac{X_n}{n}\geq 0$, $P_0^\o$--a.s., we get that $\lim_{n\to\infty}\frac{X_n}{n}=0$, $P_0^\o$--a.s.
\end{proof}

\begin{Lemma} \label{cor:0_speed} Condition \eqref{eqn:0_speed_condition2}  is equivalent to 
\begin{equation}\label{eqn:0_speed_condition}
\E\bigl[{\rm e}^{(1-\l)Z_0-(1+\l)Z_{-1}}\bigr]=\infty\,.
\end{equation}
\end{Lemma}
\begin{proof}
We want to show that condition \eqref{eqn:0_speed_condition} implies \eqref{eqn:0_speed_condition2}.
First of all, we claim that for all $\o\in\Omega$ and { $z\leq 0$}
we have
\begin{equation}\label{eq:overjump_1}
{P_0^\o(X_1\geq 1)\geq {\rm e}^{2 (u_{\rm min}-u_{\rm max})} P_z^\o(X_1\geq 1).}
\end{equation}
In fact,
$$
{P_0^\o(X_1\geq 1)
\geq {\rm e}^{(u_{\rm min}-u_{\rm max})}
\frac{\sum_{j\geq 1} {\rm e}^{-(1-\l) x_j}}
{\sum_{j\geq 1} {\rm e}^{-(1-\l) x_j}+\sum_{j\leq -1} {\rm e}^{(1+\l) x_j}} }
$$
and
$$
{
P_z^\o(X_1\geq 1)
\leq {\rm e}^{(u_{\rm max}-u_{\rm min})}
\frac{{\rm e}^{(1-\l) x_z}\sum_{j\geq 1} {\rm e}^{-(1-\l) x_j}}
{\sum_{j\geq z+1} {\rm e}^{-(1-\l) (x_j-x_z)}+\sum_{j\leq z-1} {\rm e}^{(1+\l) (x_j-x_z)}}.}
$$
Hence, \eqref{eq:overjump_1} is satisfied if
$$
\frac{   \sum_{j\geq 1} {\rm e}^{-(1-\l) x_j}}	{\sum_{j\geq 1} {\rm e}^{-(1-\l) x_j}+\sum_{j\leq -1} {\rm e}^{(1+\l) x_j}}
\geq\frac{{\rm e}^{(1-\l) x_z}\sum_{j\geq 1} {\rm e}^{-(1-\l) x_j}}
{\sum_{j\geq z+1} {\rm e}^{-(1-\l) (x_j-x_z)}+\sum_{j\leq z-1} {\rm e}^{(1+\l) (x_j-x_z)}},
$$
which is true if and only if
$$
{\rm e}^{-(1-\l)x_z}\Big(\sum_{j\geq z+1} {\rm e}^{-(1-\l) (x_j-x_z)}+\sum_{j\leq z-1} {\rm e}^{(1+\l) (x_j-x_z)}\Big)
\geq
\sum_{j\geq 1} {\rm e}^{-(1-\l) x_j}+\sum_{j\leq -1} {\rm e}^{(1+\l) x_j}.
$$ 
Simplifying the expression (the terms with $j\geq 1$ cancel out), the last display is equivalent to
$$
\sum_{z+1\leq j\leq -1} {\rm e}^{-(1-\l) x_j}+1+{\rm e}^{-2x_z}\sum_{j\leq z-1} {\rm e}^{(1+\l) x_j}
\geq \sum_{z+1\leq j\leq -1} {\rm e}^{(1+\l) x_j}+{\rm e}^{(1+\l) x_z}+\sum_{j\leq z-1} {\rm e}^{(1+\l) x_j}
$$
and the last inequality clearly holds since the l.h.s.~terms dominate one by one the r.h.s.~ones.

\eqref{eq:overjump_1} shows that $ P^\o_0(X_1\geq 1)\leq\sup_{z\leq 0}P^\o_z(X_1\geq 1)\leq C\cdot P^\o_0(X_1\geq 1)$ for a constant $C$ which does not depend on $\o$. On the other hand, using  estimates \eqref{rione1} and \eqref{rione2},
\begin{align*}
& P^\o_0(X_1\geq 1)=\frac{\sum_{j>1}c_{0,j}}{\sum_{j\not=0}c_{0,j}}\leq K_1\cdot\frac{c_{0,1}}{c_{0,-1}}=
{ K_1'}\cdot{\rm e}^{-(1-\l)Z_0+(1+\l)Z_{-1}}\\
&
P^\o_0(X_1\geq 1)\geq K_2\cdot\frac{c_{0,1}}{c_{0,-1}+c_{0,1}}={K_2'}\cdot\frac{{\rm e}^{-(1-\l)Z_0}}{ { {\rm e}^{-(1+\l)Z_{-1}}    } +  {\rm e}^{-(1-\l)Z_0}  }
\end{align*}
for constants { $K_1$, $K_1'$, $K_2$, $K_2'$}  which do not depend on $\o$.

Hence, we have  
$\eqref{eqn:0_speed_condition2}
\, \Longleftrightarrow\, \E\Big[\frac{1}{P^\o_0(X^\l_1\geq 1)}\Big]=\infty
\, \Longleftrightarrow \, \E[{\rm e}^{(1-\l)Z_0-(1+\l)Z_{-1}}] =\infty$.
\end{proof}

\begin{Corollary} \label{cor:iid_0_speed}
Suppose that $\bbE[Z_{-1}|Z_0]\leq C$  for some constant which does not depend on $\o$ (e.g.~if the $(Z_i)_{i\in\Z}$ are i.i.d.)  and that  
$
\E[{\rm e}^{(1-\l)Z_0}]=\infty$. Then condition \eqref{eqn:0_speed_condition} is satisfied and in particular 
 $v_{X^\infty}(\l)=0$.
\end{Corollary}

\begin{proof} Conditioning on $Z_0$ and using Jensen's inequality, we get 
\begin{equation*}
\begin{split}
\E[{\rm e}^{(1-\l)Z_0-(1+\l)Z_{-1}}]& = \E\bigl[{\rm e}^{(1-\l)Z_0}      \E[ {\rm e}^{-(1+\l)Z_{-1}}| Z_0]\bigr]\geq  \E\bigl[{\rm e}^{(1-\l)Z_0}      
 {\rm e}^{-(1+\l)\E[Z_{-1}  | Z_0]  }   \bigr] \\
& 
\geq {\rm e} ^{-(1+\l) C}  \E[{\rm e}^{(1-\l)Z_0}] = \infty\,.\qedhere
\end{split}\end{equation*} 
\end{proof}


\appendix
\section{Proof of Proposition \ref{digiuno}} \label{ape}
By the tightness stated in Lemma \ref{chiave}, $(\bbQ^\rho)_{\rho \in\N_+}$ admits some limit point  and any limit point $\bbQ^\infty$  is absolutely continuous to $\bbP$, with Radon--Nikodym  derivative  $\frac{ d\bbQ^\infty}{d \bbP}$ bounded by $F$ from above and by $\g$ from below.

We now show  that any limit point     is an  invariant distribution of the process given by the environment viewed from  the walker without truncation $(\tau_{X^\infty _n}\omega)_{n\in\N}$. To this end,  
let  $(\bbQ^{\rho_k})_{k \geq 1}$ be a subsequence weakly  converging to some probability  $\bbQ^\infty$  on $\O$. 
We take a bounded continuous function $f$ on $\Omega$ (without loss of generality we assume $\|f\|_\infty \leq 1$)  and we write
\begin{align}\label{eqn:b1b2b3}
\big|\E^\infty[f(\omega)]-\E^\infty E^{\o,\infty}_0[f(\tau_{X_1} \omega)]\big|
	\leq &\,\big|\E^\infty[f(\omega)]-\E^{\rho_k}[f(\omega)]\big|\nonumber\\
		&+ \big|\E^{\rho_k} E^{\o,\rho_k}_0[f(\tau_{X_1}\omega )]-\E^{\infty}E^{\o,\rho_k}			_0[f(\tau_{X_1}\omega)]\big|\nonumber\\
		&+	\big|\E^{\infty}E^{\o,\rho_k}_0[f(\tau_{X_1}\omega)] - \E^{\infty}E^{\o,\infty}				_0[f(\tau_{X_1}\omega)]\big|\nonumber\\
		=:&B_1+B_2+B_3\,.
\end{align}
Above,   $\E^\infty$ is the expectation with respect to the measure $\bbQ^\infty$ and in the second line we have used the fact that $\E^{\rho_k}$, the expectation with respect to the measure $\bbQ^{\rho_k}$, is invariant  for the process  $(\tau_{X_n^{\rho_k}}\omega)_{n\in\N}$. 
The term $B_1$ goes to zero as $k \to \infty$ since $\bbQ^{\rho_k} \to \bbQ^\infty$. 
To  deal  with term $B_2$   we observe that, by Lemma \ref{pierpaolo}, 
 for any $\delta>0$ there exists $h_0$ such that, for any $\rho\in\N_+\cup\{\infty\}$,
\begin{equation}\label{eqn:long_jumps}
P^{\o,\rho}_0(|X_1|>h_0)<\delta,\qquad \mathbb P\text{-a.s.}
\end{equation}
Then, for $\rho_k \geq  h_0$,   we write
\begin{align*}
B_2
	\leq &\Big |\E^{\rho_k}\Big[\sum_{|j|\leq h_0}P^{\o,\rho_k}_0(X_1=j)f(\tau_{j}\omega)\Big]-
		\E^{\infty}\Big[\sum_{|j|\leq h_0}P^{\o,\rho_k}_0(X_1=j)f(\tau_{j}\omega )\Big] \Big| 		+2\delta\\
	\leq &\Big |\E^{\rho_k}\Big[\sum_{|j|\leq h_0}P^{\o,\infty}_0(X_1=j)f(\tau_{j}\omega )\Big]-
		\E^{\infty}\Big[\sum_{|j|\leq h_0}P^{\o,\infty}_0(X_1=j)f(\tau_{j}\omega)\Big] \Big| \\
	&+ \E^{\rho_k}\Big[ P^{\o,\infty}_0(|X_1|>h_0)\Big] 
		+ \E^{\infty}\Big[ P^{\o,\infty}_0(|X_1|>h_0) \Big] + 2\delta\\
	\leq & \Big |\E^{\rho_k}\Big[\sum_{|j|\leq h_0}P^{\o,\infty}_0(X_1=j)f(\tau_{j}\omega )\Big]-
		\E^{\infty}\Big[\sum_{|j|\leq h_0}P^{\o,\infty}_0(X_1=j)f(\tau_{j}\omega)\Big] \Big|+4\d
\end{align*}
Note that  we have used \eqref{eqn:long_jumps}  in the first and third estimates. For the second  bound  we have used  that $h_0 \leq \rho_k$, $P^{\o,\rho_k}_0(X_1=j)= P^{\o,\infty}_0(X_1=j)$ for $0<|j| \leq \rho_k$, while $P^{\o,\rho_k}_0(X_1=0)= 1- 
\sum _{j : 0<|j-x| \leq \rho_k}   P^{\o,\infty}_0(X_1=j)$  and  $P^{\o,\infty}_0(X_1=0)= 0$ (cf. \eqref{salto}).

By the continuity assumption on $u$ and since {$\|c_{0,k} (\cdot) \|_\infty \leq e^{-(1-\l) dk +u_{\rm max}}$,}   the map $ \O \ni \o \mapsto P^{\o,\infty}_0(X_1=j)= \frac{ c_{0,j} (\o) }{ \sum _{ i \in \bbZ} c_{0,i}(\o)} \in \bbR_+$ is continuous. 
Hence, 
using   that $\bbQ^{\rho_k}$ converges to $\bbQ^\infty$ as $k\to \infty$,
 we can choose $k$ large enough so that $B_2\leq 5\delta$. $B_3$ is also smaller than $\delta$ for $k$ big enough, again by \eqref{eqn:long_jumps}. Altogether, letting $\rho \to\infty$, \eqref{eqn:b1b2b3} implies that $\bbQ^\infty$ is invariant  for $(\tau_{X^\infty _n }\omega)_{n\in\N}$ with transition mechanism induced by $P^{\o,\infty}_0$.

\smallskip

Having that $\bbQ^\infty \ll \bbP$, the ergodicity of  $\bbQ^\infty$   can be proved in the same way as Lemma \ref{lemma:ergodicity_q_rho}.



\smallskip
It remains to prove uniqueness  of the limit point. To this aim, 
take two  limit points  $\bbQ^\infty $ and $\bbQ'^\infty $ of $(\bbQ^\rho)_{ \rho \in\N_+}$. Recall that we write   $ \cP^\infty _{ \bbQ^\infty}$ and $\cP^\infty_{\bbQ'^\infty} $ for the law on the path space  $\O^{\bbZ} $ of the  Markov chains $(\tau_{X^\infty _n}\o)_{n\in\N }$, induced by $P^{\o, \infty}_0$, with initial distributions $\bbQ^\infty$ and $\bbQ'^\infty $, respectively. As proved above,  $ \cP^\infty _{ \bbQ^\infty}$ and $\cP^\infty_{\bbQ'^\infty} $  are stationary and ergodic with respect to shifts. In particular, they must be  either singular or the same. They cannot be singular, since $\bbQ^\infty$ and $\bbQ'^\infty$ are both mutually absolutely continuous with respect to $\P$ by Lemma \ref{chiave} and therefore absolutely continuous with respect to each other.  Hence,   $ \cP^\infty _{ \bbQ^\infty}$ and $\cP^\infty_{\bbQ'^\infty} $  are equal, and therefore $\bbQ^\infty =\bbQ'^\infty$.


\section{Ergodic issues } \label{appendix:ergodic}

In Lemmas \ref{pasquetta} and \ref{luce} we prove the results  we used in the proof of Proposition \ref{scarola}, see the discussion after equation \eqref{fuso}. In Lemma \ref{indi} we prove instead an assertion on assumption (A1) made in Subsection \ref{domenica}.  

\smallskip
For the first technical result, we slightly change the notation  to make it lighter: 
Take $\Omega:=\R^\Z$, the space of two-sided sequences with real values, and let $\mu$ be a stationary measure on $\Omega$, ergodic with respect to the usual shift $\tau_1$ for sequences. We indicate by $\o$ an element in $\Omega$.
Let $\Xi:=\N^\N$ and $P$ be a probability measure on it. $\eta =(\eta _i)_{i\in\N}\in\Xi$ is an i.i.d.~sequence of natural numbers under the measure $P$. We assume that the $\eta _i$'s are independent of the $\o$'s.

On the space $\Omega\times\Xi$ endowed with the product measure $\mathbb L= \mu\otimes P$, we define the transformation $T:\Omega\times\Xi\to\Omega\times\Xi$, with $T(\o,\eta )=(\tau_{\eta _1}\o,\tau_{1}\eta )$. 

\begin{Lemma}\label{pasquetta} 
Assume that the greatest common divisor of $\{k:P(\eta _1 =k)> 0\}$ equals $1$.
Assume also (just for simplicity) that the $\eta _i$'s have finite expectation.
Then, the transformation $T$ is ergodic.
\end{Lemma}

\begin{Remark} The statement is not true in general without the GCD condition. Indeed, take the very simple space with only two elements, $\o_1=(\ldots, 0,1,0,1,0,1,\ldots )$ and { $\o_2=\t_1 \o_1$}, and take $\mu$ putting $1/2$ probability to each of the two elements. Then $\mu$ is ergodic with respect to $\tau_1$.
 But, if we take $\eta _i$'s that can attain only even values, 
then the sequence $(\tau_{\eta _1+...+\eta _j}\o)_{j\in\N}$ is not ergodic under { $\bbL=\mu\times P$}.
\end{Remark}
 
\begin{proof}
Take a function $ f= f(\omega, \eta )$ which is invariant under $T$ and bounded. We are going to show that $f$ is constant, $\mathbb L$-almost surely, hence proving the claim.

Assume we have, for two sequences $\eta ^{(1)}, \eta ^{(2)}$, 
\begin{equation}\label{meetatn}
\sum_{k=1}^n \eta ^{(1)}_k = \sum_{k=1}^n \eta ^{(2)}_k
\end{equation}
for some $n$ and
$\eta ^{(1)}_k = \eta ^{(2)}_k \;\hbox{ for }\; k \geq n $.
Then $T^n(\omega, \eta ^{(1)}) = T^n(\omega, \eta ^{(2)})$ and hence $f(\omega, \eta ^{(1)}) = f(\omega, \eta ^{(2)})$.

 We define $\cF_n$ as the $\s$--algebra generated by $\o, \eta _1, \dots, \eta _n$. By the above observation we get
\begin{equation}
\E_\mathbb L\bigl[ f \,|\, \cF_n \bigr] ( \o,  \eta ^{(1)})= \E_\mathbb L\bigl[ f \,|\, \cF_n \bigr] ( \o,  \eta ^{(2)})  
\end{equation}
 if \eqref{meetatn} holds true for some $n$ (where $\E_\mathbb L$ denotes the expectation with respect to the measure $\mathbb L$).
  On the other hand, $f= \lim _{n \to \infty} \E_\mathbb L\bigl[ f \,|\, \cF_n \bigr]$ $\bbL$--a.s.
As a byproduct, we get that 
$f(\omega, \eta ^{(1)}) = f(\omega, \eta ^{(2)})$ for $\mu \otimes P \otimes P$ a.e.~$(\o, \eta ^{(1)}, \eta ^{(2)})$ such that \eqref{meetatn} happens for infinitely many $n$ (note that this event has probability one due to    the Chung--Fuchs Theorem   \cite{Du} applied to the random walk $Z_n:= \sum_{j=1}^n ( \eta ^{(1)}_j- \eta ^{(2)}_j)$).   Hence,
\begin{equation}\label{lazarus}  f(\omega, \eta ^{(1)}) = f(\omega, \eta ^{(2)}) \qquad \mu \otimes P \otimes P\text{--a.s.}
\end{equation}

We  now claim that for $\mu$--a.e.~$\o$ the function $f( \o, \cdot)$ is constant $P$--a.s. 
To this aim, it is  enough to show that for $\mu$--a.e.~$\o$ the $P$--variance  of $f( \o, \cdot)$ is zero, and this follows from \eqref{lazarus} and the identity 
\[ 
\text{Var}_P( f( \o, \cdot) )= \frac{1}{2}\int dP( \eta ^{(1)}) \int dP( \eta ^{(2) } ) \left[ f(\omega, \eta ^{(1)}) - f(\omega, \eta ^{(2)})\right]^2\,.
\]

%
Now let
$ 
A_{\ell, m} : = \Big\{\eta : \sum\limits_{i=1}^m\eta _i = \ell\Big\}$.
Since $f$ is invariant under $T$, 
$f(\omega, \eta ) = f(\tau_\ell\omega, \tau_m \eta )$ for $\eta  \in A_{\ell,m}$.
If $P(A_{\ell,m}) > 0$, we conclude that $ f(\omega, \cdot) = f(\tau_{\ell} \omega, \cdot)$ $P$-almost surely, for $\mu$--a.e. $\o$.
Since the greatest common divisor of $\{k:P(\eta _1 =k)> 0\}$ equals $1$, we conclude that there is some finite $L$ such that
$f(\omega, \cdot) = f(\tau_\ell\omega, \cdot)$ for all $\ell \geq L$, for $\mu$--a.e.~$\o$.
Since the law of $\omega$ is ergodic with respect to $\tau_1$, this implies easily that $f(\cdot, \cdot)$ is constant $\bbL$-almost surely.
\end{proof}

Now recall the definition of the random sequence $(S_k)_{k \geq 0}$  introduced at  the end of the  proof of Prop. \ref{scarola}, and the notation therein. 
\begin{Lemma}\label{luce}
The random sequence $(S_k)_{k \in\N}$ is stationary and ergodic with respect to shifts.
\end{Lemma}
\begin{proof}  We first show that the sequence  $(s_k)_{k \geq 0}$ (see \eqref{fuso})   is stationary and ergodic with respect to shifts,  under $P^{(\infty)}$. Indeed,  writing  \eqref{fuso} in a compact form as $(s_k)_{k \geq 0} = G( \o,  ( \xi_k)_{k \geq 1})$, it holds $(s_k)_{k \geq 1}= G( \t_{\xi_1} \o, ( \xi_k)_{k \geq 2} )$. Then stationarity and ergodicity of  $(s_k)_{k \geq 0}$ under $P^{(\infty)}$ follow from  the stationarity and ergodicity of  $(\o,  ( \xi_k)_{k \geq 1})$  under $P^{(\infty)}$ as in   Lemma \ref{pasquetta}.

 We move to $(S_k)_{k \geq 0}$. Since  $(s_k)_{k \geq 0}$, under $P^{(\infty)}$,  is stationary, one gets easily the stationarity of $(S_k)_{k \geq 0}$ under $\tilde P^{(\infty)}$.
 Take now  a shift invariant   Borel set $A \subset \bbN ^{\bbN_0}$ (i.e. $A= \{ (x_0,x_1,\dots)  \in \bbN ^{\bbN_0}\,:\, (x_1,x_2, \dots ) \in A \}$). 
 We claim that 
 \begin{equation}\label{salsiccia} 
 \tilde P^{(\infty) }\bigl(\, (S_0, S_1, \dots ) \in A\, \bigr)  \in \{0, 1\} \,. 
 \end{equation}
We define $f: \bbN^{\bbN_0} \to \bbR$ as the Borel function such that 
\[ f( s_0, s_1, s_2, \dots )= \tilde P^{(\infty)} \bigl(    (S_0, S_1, \dots ) \in A   | s_0,s_1, \dots \bigr) \,.
\]
Since $A$ is shift invariant, $A$ belongs to the tail $\s$--algebra of $\bbN ^{\bbN_0}$. 
By Kolmogorov's 0--1 law and due to the independence of $S_0,S_1, \dots $ under
$\tilde P^{(\infty)} ( \cdot | s_0,s_1, \dots )$, we get that $f$ has values in $\{0,1\}$.

Below,  for the sake of intuition  we condition  to events of zero probability although all  can be formalized by means of regular conditional probabilities. 
Using that  $ \{ (S_0, S_1, \dots ) \in A \} =  \{ (S_1,S_2,  \dots ) \in A \} $ due to the shift invariance of $A$ and using the definition of $(S_k)_{k \geq 0}$, we get 
\begin{align*}
 f(a_0,a_1, \dots )&=   \tilde P^{(\infty)} (    (S_0, S_1, \dots ) \in A   | s_0=a_0,s_1=a_1, \dots ) \\
&= \tilde P^{(\infty)} (    (S_1, S_2, \dots ) \in A   | s_0=a_0,s_1=a_1,s_2=a_2 \dots ) \\
& = \tilde P^{(\infty)} (    (S_0, S_1, \dots ) \in A   | s_0=a_1,s_1=a_2, \dots ) = 
f(a_1,a_2, \dots ) 
\end{align*}
Hence  $f$ is shift invariant. By the ergodicity of $(s_k)_{k \geq 0}$, we conclude that the $0/1$--function $f(s_0,s_1, \dots )$ is constant $P^{(\infty)}$--a.s.  An integration  over $(s_0, s_1, \dots )$ allows to  get \eqref{salsiccia}. 
\end{proof}

\begin{Lemma}\label{indi} 
Consider two independent 
  random sequences $(Z_k)_{k \in \bbZ}$ and $(E_k)_{k \in \bbZ}$, the former  stationary and ergodic with respect to shifts,  the latter  given by i.i.d.~random variables. Then the random sequence $(Z_k,E_k)_{k \in \bbZ}$ is  stationary and ergodic with respect to shifts.
 \end{Lemma}
\begin{proof} Call $P$ the law of $\left( (Z_k)_{k \in \bbZ}, (E_k)_{k \in \bbZ}\right)$, which is a probability measure on the space  $\bbR^\bbZ \times \bbR^\bbZ$, whose generic element will be denoted by $(\underline z , \underline e)$. We write $T$ for the shift 
$ [T(\underline z , \underline e)]_k= (z_{k+1}, e _{k+1} ) $.  
Let $A$ be a  shift--invariant Borel subset of $\bbR^\bbZ \times \bbR^\bbZ $.  We want to show that $P(A) \in \{0,1\}$.

\smallskip

We first claim that, given $r\geq 1$,   $A$   is independent of any set $B$ in the $\s$--algebra generated by $e_i$ with $|i| \leq r$.  To this aim, given $\e>0$, we fix a Borel set $A_n  \subset \bbR^\bbZ \times \bbR^\bbZ $  
 belonging to the  $\s$-algebra generated by $e_i,z_i$ with $|i| \leq n$, and  such that $P( A \D A_n)\leq \e $. We take $m$ large enough so that $[-r, r] \cap [ -n +m, n+m]= \emptyset$.  
We observe that 
\begin{align}
& P(A \cap B)= P(A_n \cap B) + O(\e) \,, \label{malasia1}\\
& P(A \cap B) = P( T^m A \cap B)= P( T^m A_n \cap B) + O(\e)=P( T^m A_n )P( B) + O(\e)  \,.\label{malasia2}
\end{align}
Indeed,  the first identity in \eqref{malasia2} follows from the shift invariance of $A$, while the second identity follows from the shift stationarity of $P$ implying that $P( T^m A_n \D T^m A) \leq \e$. To get the third identity in \eqref{malasia2}   we observe that $T^m A_n$ belongs to the $\s$--algebra generated by 
$e_i,z_i$ with $i \in [ -n +m, n+m]$. By our choice of $m$ and due to  the properties of $P$,  we get that $ T^m A_n$ and $B$ are independent, thus implying the third identity.

As a byproduct of \eqref{malasia1} and \eqref{malasia2} and the fact that $P( T^m A_n )=P(A)+O(\e)$, we get that
$ P(A \cap B)=P(A)P(B)+O(\e)$.  By the arbitrariness of $\e$ we conclude the proof of our claim.

Due to our claim, $\mathds{1} _A= P(A|\cF)$, $\cF$ being   the $\s$--algebra generated by $z_i$, $i \in \bbZ$.  We can think of $P(A|\cF)$ as function of $\underline z \in \bbR^{\bbZ}$. Due to the shift invariance of $A$, $P(A|\cF)  $ is shift invariant in  $\bbR^{\bbZ}$ except on an event of probability zero. Due to the ergodicity of the marginal of $P$ along $\underline z$, we conclude that $P(A|\cF)  $  is constant a.s. Since $\mathds{1} _A= P(A|\cF)$, $\mathds{1} _A$ is constant a.s., hence $P(A) \in \{0,1\}$.   
\end{proof}

\section{The nearest neighbor random walk $(X_n ^\rho)_{n \geq 1}$, $\rho=1$}\label{app_nn}
The biased Mott  random walk $(\bbY_t)_{t\geq 0}$   can be compared to the nearest neighbor random walk obtained by considering only nearest neighbor jumps on $\{x_j\}_{j \in \bbZ}$ with probability rate for a jump from $x$ to $y$  given by \eqref{eq_eq_eq}   when $x,y$ are nearest neighbors. By the same arguments as in Section \ref{ballistic_part}, it is simple to show that this  random walk is ballistic/subballistic if and only if the same holds for $(X_n ^\rho)_{n \in\N}$, $\rho=1$. The latter can be easily analyzed and the following holds:
\begin{Proposition}
The limit $v_{X^1}(\l) := \lim  _{n \to \infty} \frac{ X_n^1}{n}$ exists  $\bbP^{\o,1}_0$--a.s.~for $\bbP$--a.a.~$\o$, and  it does not dependent on $\o$. Moreover, the velocity $v_{X^1}(\l)$ is positive if and only if  condition \eqref{musica} is fulfilled,
otherwise it is zero.
\end{Proposition}
\begin{proof} We apply Theorem 2.1.9  in \cite{OferStFlour}  using the notations therein. Since $\rho_i = c_{i,i-1}/c_{i,i+1} $ we get that 
$\bar S=\frac{1}{c_{0,1} }\sum _{i=0} ^\infty ( c_{-i,-i-1}+ c_{-i,-i+1} )$. Therefore,   $\bbE( \bar S)< \infty$ if and only if $\sum_{i=0}^\infty \bbE \bigl( c_{-i,-i-1}/c_{0,1}\bigr)<\infty$. The last condition is equivalent to \eqref{musica} since the energy marks are bounded. 
On the other hand $ \bar F= \frac{1}{c_{-1,0} } \sum _{i=1}^\infty ( c_{i,i-1}+ c_{i,i+1} ) $. Hence, $\bbE( \bar F) =\infty$ if and only if $ \sum _{i=0}^\infty \bbE( c_{i,i+1}/ c_{-1,0} ) =\infty$. Since, when $u \equiv 0$,  $c_{i,i+1}/c_{-1,0}= \exp \{ (1+\l) Z_{-1} + 2\l (Z_0+ \dots+ Z_{i-1} ) -(1-\l) Z_i \} $, by Assumption (A4) it follows that $\bbE(\bar F)=+\infty$ always.
The claim then follows since, by  Theorem 2.1.9  in \cite{OferStFlour}, $v_{X^1}(\l) >0$ if $\bbE (\bar S) <\infty$, while $v_{X^1}(\l) =0$ if $\bbE (\bar S) =\infty$ and $\bbE( \bar F)=\infty$.
\end{proof}

\bigskip

\bigskip

\noindent
{\bf Acknowledgements}.  The authors are very much indepted to Noam Berger and thank him for many useful  and inspiring discussions.  A.F.~kindly acknowledges the  Short Visit Grant of the European Science Foundation (ESF)  in the framework of the ESF program ``Random Geometry of Large Interacting Systems and Statistical Physics'', and the Department of Mathematics  of the Technical University
  Munich for the kind hospitality. M.S.~thanks the Department of Mathematics of University La Sapienza in Rome for the kind hospitality,  the Department of Mathematics  of the Technical University
  Munich where he has started to  work  on this project as a  Post Doc, and the  funding received from the European Union's Horizon 2020 research and innovation programme under the Marie Sklodowska-Curie Grant agreement No 656047.


\begin{thebibliography}{WWW98}
  
  
 \bibitem{A}  S. Alili, Asymptotic behavior for random walks in random environments, J. Appl. Prob. 36, 334--349 (1999).

\bibitem{AHL} V. Ambegoakar, B.I.  Halperin, J.S.  Langer. \emph{Hopping Conductivity in Disordered Systems}. Phys, Rev B {\bf 4}, 2612--2620 (1971).

\bibitem{BGP}   N.~Berger, N.~Gantert and Y.~Peres, \emph{The speed of biased random walk on percolation clusters}, Probab. Theory Relat. Fields {\bf 126}, 221--242 (2003).




\bibitem{B} P. Billingsley. \emph{Convergence of probability measures}. John Wiley \& Sons, Inc.
New York. 1999.


\bibitem{CF}  P.\ Caputo, A.\ Faggionato. \emph{Diffusivity of 1--dimensional
generalized Mott variable range hopping}. Ann. Appl. Probab. {\bf
19}, 1459--1494 (2009).



\bibitem{CFG}  P.\ Caputo, A.\ Faggionato, A.\ Gaudilli\`{e}re. \emph{Recurrence and
transience for a random walk on a random point process}. Electron. J. Probab.  {\bf 14}, 2580--2616 (2009).


\bibitem{CFP}
P. Caputo, A. Faggionato, T.
Prescott. \emph{Invariance principle for Mott variable range hopping
and other walks on point processes}.  Ann. Inst. H. Poincar\'{e} Probab. Statist. {\bf 49} 654--697 (2013).




\bibitem{CP}  F.  Comets, S. Popov.  \emph{Ballistic regime for random walks in random environment with unbounded jumps and Knudsen billiards}. Ann. Inst. H. Poincar\'e  Probab. Statist. {\bf 48}, 721--744 (2012).



\bibitem{CC} E. Cs\'aki, M. Cs\"org\"o, \emph{ On additive functionals of Markov chains}. J. Theoret. Probab. 8,  905--919 (1995).
\bibitem{DV} D.J.\ Daley, D. Vere Jones. \emph{An Introduction to the Theory of Point Processes}. New York, Springer, 1988

\bibitem{DFGW} A.\ De Masi, P.A.\ Ferrari, S.\ Goldstein, W.D.\ Wick. \emph{An Invariance Principle for Reversible
Markov Processes. Applications to Random Motions in Random Environments.}  J. Stat. Phys. {\bf  55}, 787--855 (1989).


 \bibitem{Du} R. Durrett; \emph{Probability. Theory and examples.} 2nd Edition. Duxbury Press,  Washington, 1995. 


\bibitem{FM}
 A.\ Faggionato, P.\ Mathieu. \emph{Mott law as upper bound for a random
walk in a random environment}. Comm. Math. Phys. {\bf 281}, 263--286
(2008).




\bibitem{FSS} A.\ Faggionato, H.\ Schulz--Baldes, D.\ Spehner, {\em
     Mott law as lower bound for a random walk in a random
     environment.} Comm.\ Math.\ Phys., {\bf 263},
     21--64, 2006.



\bibitem{FKAS} P. Franken,  D. K\"{o}nig, U. Arndt, V. Schmidt; \emph{Queues
and Point Processes}.
John Wiley and Sons, Chichester, 1982.

\bibitem{GMP} N.~Gantert, P.~Mathieu, and A.~Piatnitski; \emph{Einstein relation for reversible diffusions in random environment}.
 Comm. Pure Appl. Math., 65: 187-228, 2012.

  \bibitem{LP} R. Lyons, Y. Peres. \emph{Probability on trees and networks}. 
Version of {6th  May 2016}. Available online.

  \bibitem{MA}
A. ~Miller, E.~Abrahams. \emph{Impurity Conduction at Low Concentrations}. Phys.~Rev.~{\bf 120}, 745--755 (1960).

\bibitem{MD}  N.F.  Mott, E.A.
Davis. \emph{Electronic Processes in Non-Crystaline Materials}.  Oxford University
Press, New York, 1979.



 \bibitem{KV} C. Kipnis, S.R.S. Varadhan. \emph{Central limit theorem for additive functionals of
reversible Markov processes and applications to simple exclusion}. Commun. Math. Phys. {\bf 104}
 (1986).


  \bibitem{POF}  M. Pollak, M. Ortu\~no, A. Frydman. \emph{The electron glass}. Cambridge University Press,  Cambridge, 2013.



\bibitem{Ros} M.  Rosenblatt,\emph{Markov Processes. Structure and asymptotic behavior.} Grundlehren der mathematischen Wissenschaften  {\bf 184}, Berlin,  Springer, 1971.

\bibitem{Shen} L.\ Shen, \emph{Asymptotic properties of certain anisotropic walks in random media}. Annal. Probab. 12, 477--510, 2002.

\bibitem{SE} B.~Shklovskii, A.L.~Efros, \emph{Electronic Properties of Doped Semiconductors}.   Springer, Berlin 1984.


\bibitem{So} F.\ Solomon, \emph{Random walks in a random environment}. 
Annal. Probab. 3, 1--31 (1975).

\bibitem{Sp} F.\ Spitzer, \emph{Principles of random walk}. Second edition, Graduate Texts in Mathematics, Vol. 34, Springer--Verlag, 1976.


\bibitem{T} H.\ Thorisson. \emph{Coupling, Stationarity, and Regeneration}, Springer,  Berlin  (2000).

\bibitem{V} S.R.S.~Varadhan, \emph{Probability Theory}.    Providence, American mathematical society,  Providence, 2001. 

\bibitem{OferStFlour} O. Zeitouni, \emph{Random walks in random environment}, 
  XXXI Summer School in Probability, St. Flour (2001). Lecture Notes in Math. 1837, Springer,  193--312, 2004.


\bibitem{SZ} A.--S. Sznitman, M.P.W.  Zerner, \emph{A law of large numbers for random walks in random environment}. 
 Ann. Probab. 27, 1851--1869 (1999).


\end{thebibliography}
 \end{document}